\documentclass[12pt,fleqn]{article}
\usepackage{mathtools}
\usepackage[margin=1 in]{geometry}
\usepackage{amsmath}
\usepackage{amsfonts}
\usepackage{amsthm}
\usepackage{comment}
\usepackage{multirow}
\usepackage{graphicx}
\usepackage{tikz}
\usepackage{lscape}
\usetikzlibrary{decorations}
\usetikzlibrary{decorations.markings,arrows}
\usepackage{caption}
\newtheorem{Thm}{Theorem}[section]

\newtheorem{Prop}[Thm]{Proposition}
\newtheorem{Lemma}[Thm]{Lemma}

\theoremstyle{definition}

\newcommand{\be}{\begin{equation}}
\newcommand{\ee}{\end{equation}}
\newfont{\msbm}{msbm10 scaled\magstep1}%blackboardbold
\newcommand{\bbz}{\mbox{$\mbox{\msbm Z}$}}
\newcommand{\ep}{\varepsilon}
\newcommand{\wt}{\widetilde}

    \def\Re{{\rm Re \,}}
    \def\Im{{\rm Im \,}}

    \def\bigO{{\cal O}}

    \def\P2n{{\rm P}_{{\rm II}}^{(n)}}

    \newtheorem{theorem}{Theorem}[section]

    \newtheorem{Definition}[theorem]{Definition}
    
    \newtheorem{Remark}[theorem]{Remark}
    \newenvironment{remark}{\begin{Remark}\rm}{\end{Remark}}
    \newtheorem{Example}[theorem]{Example}
    
    \newtheorem{Assumptions}[theorem]{Assumptions}
    
\title{Splitting of a Gap in the Bulk of the Spectrum of Random Matrices} 
\author{Benjamin Fahs  \and Igor Krasovsky }
\date{ }
\begin{document}
\maketitle

\begin{abstract}
We consider the probability of having two intervals (gaps) without eigenvalues 
in the bulk scaling limit of the Gaussian Unitary Ensemble of random matrices.
We describe uniform asymptotics for the transition between 
a single large gap and two large gaps. For the initial stage of the transition, 
we explicitly determine all the asymptotic terms (up to the decreasing ones) of the logarithm of the probability. 
We obtain our results by analyzing double-scaling asymptotics of a Toeplitz determinant whose symbol is supported on two arcs of the unit circle.
\end{abstract}

\section{Introduction}
Let $A$ be the union of $m$ open disjoint intervals on $\mathbb R$, and $K_s$ be the (trace-class) integral operator on $L^2(A,dx)$ given by the kernel
\begin{equation} \label{def sine kernel}
K_s(x,y)=\frac{\sin s(x-y)}{\pi(x-y)}.
\end{equation}
Consider the Fredholm determinant
\begin{equation} \label{def Ps}
P_s(A)=\det(I-K_s)_A.
\end{equation}
For a wide class of random matrix ensembles \cite{Oxford}, in particular for the Gaussian Unitary Ensemble, $P_s(A)$ is the probability that the set $\frac{s}{\pi}A=\{\frac{s}{\pi}x:x\in A\}$ contains no eigenvalues in the bulk scaling limit where the average distance between the eigenvalues is $1$. In this paper, we are interested in the asymptotics of $P_s(A)$ as $s\to \infty$, and we study the transition between a single interval $A_0=(\alpha,\beta)$ to the set $A$ composed of 2 disjoint intervals
\begin{equation}\label{A1A2}
A=A_1\bigcup A_2, \quad A_1=(\alpha_1, \beta_1), \, A_2= (\alpha_2,\beta_2).\end{equation} Such problems have a rich history, of which we mention some relevant results. For the single interval case $A_0=(\alpha,\beta)$, 
\begin{equation} \label{1 gap formula} \begin{aligned}
\log P_s(A_0)&=-\frac{(\beta-\alpha)^2s^2}{8}-\frac{1}{4}\log s-\frac{1}{4}\log\frac{\beta-\alpha}{2}+c_0+\mathcal O(s^{-1}),\\
c_0&=\frac{1}{12}\log 2+3\zeta'(-1),
\end{aligned}\end{equation}
as $s\to \infty$, where $\zeta$ is the Riemann zeta-function. The leading term and logarithmic term in \eqref{1 gap formula} were conjectured by des Cloizeaux and Mehta \cite{CM} in 1973, while the constant term $c_0$ remained undetermined until Dyson \cite{Dyson} conjectured an expression for it in 1976, relying on inverse scattering techniques and the work of Widom \cite{Widom 1} on Toeplitz determinants (see below). The constant $c_0$ became known as the Widom-Dyson constant. The first rigorous confirmation of the leading term in \eqref{1 gap formula} was given by Widom \cite{Widom 2} in 1994. In a landmark paper of 1997, Deift, Its, Zhou \cite{DIZ} were able to confirm the leading term and the logarithmic term, but the proof of the constant $c_0$ continued to defy their techniques. 
Finally, two independent proofs of the constant were later given by Erhardt \cite{Ehrhardt} and the second author \cite{K04}, and a further third proof given in \cite{DIKZ}. The proofs in \cite{K04}, \cite{DIKZ} use Riemann-Hilbert (RH) methods, while \cite{Ehrhardt} uses operator theoretical techniques.

When $A$ is composed of any (fixed) number of intervals, the main term was found and proved by Widom \cite{Widom 3} in 1995, where he was also able to identify the next term in the following result:
\begin{equation}
\frac{d}{ds}\log P_s(A)=sC_1 +C_2(s)+o(1),
\end{equation}
as $s\to \infty$. The constant $C_1$ is explicitly computable, while $C_2(s)$ is an oscillatory function given by a Jacobi inversion problem.
In \cite{DIZ}, which was mentioned above, the authors were also able to find the full asymptotic expansion for the logarithmic derivative of the determinant on any number of intervals and describe the oscillations in terms of $\theta$-functions. Here we present their result when $A$ is composed of 2 intervals as in \eqref{A1A2}:
\begin{equation} \label{2 gap formula}
\frac{d}{ds}\log P_s(A)=-2s G_0 +\frac{d}{ds}\log \theta(sV;\tau)+\mathcal O(s^{-1}).
\end{equation}
More precisely, for any $j=1,2,\dots$, the error term here is of the form
\be
\frac{G_1(s)}{s}+ \frac{G_2(s)}{s^2}+\cdots +\frac{G_j(s)}{s^j}+\mathcal O(s^{-j-1}).
\ee
where $G_j(s)$, $j=1,\dots$, are bounded periodic functions of $s$.
Here
\begin{equation} \label{def theta}
\theta(z)=\theta(z;\tau)=\sum_{m\in \mathbb Z} e^{2\pi i zm+\pi i \tau m^2}
\end{equation}
is the Jacobi Theta-function of the third kind, see e.g. \cite{WW}. The constants (in $s$) $V,\tau,G_0$,
are given in terms of elliptic integrals, and $G_1(s),G_2(s),\dots$ are given in terms of $\theta$-functions.
 Let 
\begin{equation}
r(z)=((z-\alpha_1)(z-\beta_1)(z-\alpha_2)(z-\beta_2))^{1/2},
\end{equation}
with branch cuts on $A$ such that $r(z)>0$ for $z>\beta_2$,
and let
 $q(z)$ be the unique monic polynomial of degree $2$ such that
\begin{equation}\label{q9}
\int_{A_j}\frac{q(z)dz}{r_+(z)}=0,\qquad j=1,2,
\end{equation} 
where $r_+(z)$ is the limit of $r(z+i\epsilon)$ as $\epsilon \to 0_+$ (where the "+" side is chosen merely for definiteness).
 Then $q/r$ has no residue at infinity. 
  Hence as $z \to \infty$, $q/r$ has the form
\begin{equation}\label{G010}
\frac{q(z)}{r(z)}=1+\frac{G_0}{z^2}+\mathcal O(z^{-3}),
\end{equation}
which defines the constant $G_0$ appearing in \eqref{2 gap formula}.

The parameters $V,\tau$ appearing in the arguments of the $\theta$-function in \eqref{2 gap formula} are as follows:
\begin{equation}\label{Vtau11}
V=-\frac{1}{\pi}\int_{\beta_1}^{\alpha_2} \frac{q(x)dx}{r(x)},\qquad 
\tau=i\frac{\int_{\beta_1}^{\alpha_2}\frac{dx}{|r(x)|}}{\int_{\alpha_2}^{\beta_2}\frac{dx}{|r_+(x)|}}.
\end{equation}

Integrating (\ref{2 gap formula}) in $s$ from some large value $s_0$ to $s$, and using the properties of $G_1(s)$,
Deift, Its, and Zhou concluded that
\be \label{2 gap formula int}
\log P_s(A)=-s^2 G_0 +\log \theta(sV;\tau)+ \widehat G_1 \log s + c_1+ \mathcal O(s^{-1}),\qquad
 \widehat G_1=\lim_{x\to\infty}\frac{1}{x}\int_{s_0}^x G_1(t)dt.
\ee
The value of the constant $c_1$ is unknown as there is no point $s_0$ for the lower limit of integration 
where $P_{s_0}(A)$ would be explicitly known.

In this paper we study the transition between the single interval formula \eqref{1 gap formula} and the two-interval formula \eqref{2 gap formula int}. We obtain an explicit expression (up 
to the decreasing terms in the expansion of $\log P_s(A)$) for
the asymptotics in the regime where the length $\ell$ of the interval between the gaps decreases sufficiently fast (slightly faster than $1/s$: see below) as $s\to\infty$. 
On the other hand we show that the asymptotics of \cite{DIZ}, obtained for fixed gaps,
can be extended (with proper adjustments) to the regime when $\ell$ is no longer fixed but 
decreases sufficiently slowly as $s\to\infty$. These two regimes overlap. Thus our analysis provides uniform asymptotics for the whole transition. Note, however, that since the constant
in \eqref{2 gap formula int} is not determined, the expression for the asymptotics
in the second regime is not fully explicit. Obtaining an explicit expression for these asymptotics
and establishing the constant in \eqref{2 gap formula int} is a separate problem and we plan
to address it in future work.

The initial phase of the transition between \eqref{1 gap formula} and \eqref{2 gap formula int}
resembles the birth of a cut --- emergence of an extra
interval of support of the limiting eigenvalue density in a unitary ensemble of random matrices:
asymptotics for the correlation kernel of the eigenvalues in that
case were obtained independently by Bertola \& Lee \cite{BertolaLee}, Claeys \cite{Claeys}, Mo \cite{Mo}. From the technical point of view, our analysis is very different as we are dealing with
so-called hard edges rather than soft edges in \cite{BertolaLee,Claeys,Mo} and in the context of a different model, so both the g-function needed in the analysis and the local parametrix are different. Moreover, the works \cite{BertolaLee,Claeys,Mo}
deal with correlation kernels and not determinants.

Consider the Toeplitz determinant whose symbol $f(z)$ is the characteristic function of a subset $J$ of the unit circle $C$:
\begin{equation}\label{T}
D_n(J)=\det \left(f_{j-k}\right)_{j,k=0}^{n-1}, \quad f_k=\int_{e^{i\theta}\in J} e^{-ik\theta} \frac{d\theta}{2\pi},
\end{equation}
where integration is in the positive direction around the unit circle.
The proofs of the expansion \eqref{1 gap formula} including the constant term $c_0$ in \cite{DIKZ,K04} were based on an analysis of the Toeplitz determinant $D_n(J_2)$ where $J_2$ is an arc of the unit circle
\begin{equation} J=J_2=\left\{ e^{i\theta}|\theta \in \left(-\pi,\theta_2 \right)\cup\left(\theta_1,\pi\right]\right\}.
\end{equation}
The asymptotics of $D_n(J_2)$, as $n\to\infty$, for a fixed arc $J_2$ were found by Widom \cite{Widom 1}. In \cite{DIKZ,K04}, Widom's result was extended to the case of $J_2=J_2^{(n)}$ varying with $n$ such that $|\theta_1-\theta_2|\to 0$ sufficiently slowly. Namely,
\begin{equation} \label{Toeplitz 1 arc}
\log D_n(J_2^{(n)})=n^2 \log \cos \frac{\theta_1-\theta_2}{4}-\frac{1}{4}\log \left(n \sin \frac{\theta_1-\theta_2}{4}\right)+c_0+\mathcal O\left(\frac{1}{n\sin\frac{\theta_1-\theta_2}{4}}\right)
\end{equation}
as $n\to \infty$, uniformly for $\frac{s_0}{n}\leq \frac{\theta_1-\theta_2}{2}\leq \pi-\epsilon$, for $\epsilon >0$ and with $s_0$ sufficiently large.
Asymptotics \eqref{1 gap formula} are obtained from  \eqref{Toeplitz 1 arc} by using the fact that
\begin{equation}
 \lim_{n\to \infty}D_n(J^{(n)}_2)= \det(I-K_s)_{A_0}
\end{equation}
for fixed $s$ and by taking the limit in \eqref{Toeplitz 1 arc} as $n\to \infty$ with $\theta_1=\frac{2s\beta}{n}$ and $\theta_2=\frac{2s\alpha}{n}$, where $\alpha,\beta$ are fixed.  The approach of the present paper is based on an analysis of $D_n(J)$ where $J=J^{(n)}$ is the union of 2 arcs
$J^{(n)}=J_1^{(n)}\cup J_2^{(n)}$, with $J_1^{(n)}\subset C\setminus J_2^{(n)}$ of sufficiently small length in comparison with $C\setminus J_2^{(n)}$ (see Theorem \ref{prop toeplitz} below).
We obtain our results on the sine-kernel determinant by taking the limit $n\to \infty$ of $D_n(J^{(n)})$. However, we believe Theorem \ref{prop toeplitz} below to be of independent interest for a future study of Toeplitz determinants with symbols supported on several arcs.

\subsection{Results}
%\begin{figure}
%\begin{tikzpicture}
%\draw [<->] (0,3) -- (0,0) -- (5,0);
%\draw (1,2) -- (5,2);
%\draw (1,1.3) to [out=-15,in=-185] (5,0.7);
%\node[below] at (5,0) {$s$};
%\node[left] at (0,3) {$\nu$};
%\node at (3,3) {I};
%\node at (3,1.4) {II};
%\node at (3,0.44) {III};
%\node[below] at (3,0) {IV};
%\end{tikzpicture}
%\caption{Phase diagram as $s\to \infty$:\\
%\end{figure}
The kernel \eqref{def sine kernel} is translationally invariant and so we can assume the following form for $A$:
\begin{equation} \label{notation A}
A=(\alpha,-\nu)\bigcup(\nu, \beta), \qquad \alpha<0<\nu<\beta.
\end{equation}
In this paper we provide the asymptotics of $\log P_s(A)$ (including the constant term) in the double scaling limit as $s\to \infty$ while $\nu \to 0$ in such a way that $s\nu\log \nu^{-1} \to 0$, and connect these asymptotics with those of \cite{DIZ}.

Let $\gamma=\frac{1}{8}(\beta^{-1}-\alpha^{-1})$ and 
\begin{equation}
\omega=\frac{s\sqrt{|\alpha \beta|}}{\log(\gamma\nu)^{-1}}>0.
\end{equation}
Clearly, $\omega$ is uniquely represented in the form
\begin{equation}
\omega=k+x, \quad k =0,1,2,\dots , \quad  x \in [-1/2,1/2).
\end{equation}
We note that $\nu$ has the form
\begin{equation}
\nu=\gamma^{-1}e^{-\frac{s\sqrt{|\alpha\beta|}}{k+x}}.
\end{equation}
We prove the following:

 \begin{Thm} \label{Theorem u0s to 0}
 As $s\to \infty$, uniformly for $\nu \in (0, \nu_0)$, where $ s\nu_0\log \nu_0^{-1} \to 0$, 
 \begin{equation} \label{equation 1}
 \begin{aligned}
 \log P_s (A)&=
 \log P_s(A_0) 
 +s\sqrt{|\alpha \beta|}\left(\omega-\frac{x^2}{\omega}\right)+c(k)+\delta_k(x)
\\& \qquad \qquad \qquad\qquad +\mathcal O (\max \{s\nu_0\log \nu_0^{-1},1/\log \nu_0^{-1},s^{-1}\}),\\
c(k)&=\log \left(\frac{2^{2k^2-k}}{\pi^k}\frac{G(k+1)^4}{G(2k+1)}\right),\\
\delta_k(x)&= \log \left(1+2\pi \kappa_{k-1}^2 (\gamma \nu)^{1+2x}\right)
+ \log \left(1+(2\pi \kappa_k^2)^{-1}(\gamma \nu)^{1-2x}
\right),
\end{aligned}\end{equation}  
where $G$ is the Barnes G-function, and where $\kappa_j$ is the leading coefficient of the Legendre polynomial of degree $j$ orthonormal on the interval $[-2,2]$,
%\begin{equation}
%\int_{(-2,2)}p_j(x)p_k(x)dx=\delta_{j,k}=\begin{cases} 1 & \textrm{if $j=k$,}\\
%0 & \textrm{otherwise.}\end{cases}
%\end{equation}
given by 
\begin{equation} \label{def Leg coeff}
\kappa_j =4^{-j-1/2}\sqrt{2j+1}\frac{(2j)!}{j!^2},  \quad  j =1,2,\dots, \qquad \kappa_0=1/2, \qquad \kappa_{-1}=0, 
\end{equation} 
and $\log P_s(A_0)$ is given in (\ref{1 gap formula}).

 As $s\to \infty$, uniformly for $\nu\in(\nu_1,\nu_0)$, where
$ s\nu_0\log \nu_0^{-1}\to0,\,\, \frac{s}{\log \nu_1^{-1}}\to \infty$ (i.e., $k\to \infty$), formula \eqref{equation 1} reduces to 
 \begin{equation}\label{equation 2}\begin{aligned}
 \log P_s(A)&=
s^2\left(-\frac{(\beta-\alpha)^2}{8}+\frac{|\alpha \beta|}{\log (\gamma \nu)^{-1}}\right)
-\frac{1}{2}\log s+\frac{1}{4}\log \log (\gamma \nu)^{-1}
-x^2 \log (\gamma \nu)^{-1}
\\& +\log \left(1+ (\gamma\nu)^{1-2|x|}\right)-\frac{1}{4} \log \left(\frac{\beta-\alpha}{2}\sqrt{|\alpha \beta|}\right)
+\frac{1}{6}\log 2+6\zeta'(-1)
\\&+\mathcal O \left(\max \left\{s\nu_0\log \nu_0^{-1},\frac{1}{\log \nu^{-1}_0},\frac{\log \nu^{-1}_1}{s} \right\}\right),
\end{aligned}\end{equation}
where $\zeta$ is the Riemann zeta-function.   
\end{Thm}

\begin{remark} 
Note that if $k=0$ and $s\to \infty$ while $x\in(0,1/2-\epsilon)$ for $\epsilon>0$, then \eqref{equation 1} shows that $\frac{P_s(A)}{P_s(A_0)}\to 1$.
\end{remark}

\begin{remark}\label{Remark12}
As we show in Section \ref{sec5} 
(Lemma \ref{lemmaconnect}), the Deift-Its-Zhou asymptotics (\ref{2 gap formula}) for 2 fixed gaps 
where we set $\alpha_2=-\beta_1=\nu$, $\alpha_1=\alpha$, $\beta_2=\beta$, can be extended 
(with a worse error term) to the region where $\nu\to 0$ in such a way
that 
\[
s\nu^{1/2+\varepsilon}\to\infty,
\]
for any $\varepsilon>0$.
Clearly, this region overlaps with the region of validity 
\[
s\nu\log(\gamma \nu)^{-1}\to 0
\]
of Theorem \ref{Theorem u0s to 0}. For example, $\nu=s^{-3/2}$ belongs to both regions. In Remark \ref{Remark52} we explicitly show
the coincidence of the main (order $s^2$) asymptotic terms. Full explicit formulas for this matching will be a subject of future work.  
\end{remark}

\begin{remark}
The function $c(k)$ can alternatively be described in terms of the coefficients \eqref{def Leg coeff} of the Legendre polynomials:
\begin{equation} \label{c(k) kappa}
c(k)=-\sum_{j=0}^{k-1}\log 2\pi \kappa_j^2,\quad   k=1,2,\dots, \qquad c(0)=0.
\end{equation}
Formula \eqref{c(k) kappa} shows that $\delta_k(x)+c(k)$ is continuous also at the points $|x|=1/2$.
\end{remark}

\begin{remark} \label{remark kernel}
The rescaled sine process (with expected distance between particles $\pi/s$) is the determinantal point process with the $m$'th correlation function $\rho_m$, for $m=1,2,\dots$, given by
\begin{equation}
\rho_m(x_1,\dots,x_m)=\det(K_s(x_i,x_j))_{i,j=1}^m.
\end{equation}
Consider the rescaled sine process conditioned to have no eigenvalues in $A$. Denote this process by $\mathbb P_A$ and its $m$'th correlation function by $\rho_m^A$. In Section \ref{CorrKernel}, we show that for $x_1, \dots, x_m \in (-1,1)$,
\begin{equation}\label{convergence KLeg}
\nu^m\rho_m^A(\nu x_1,\dots,\nu x_m) \to \det(2 K_{\textrm{Leg}}(2x_i,2 x_j))_{i,j=1}^m
\end{equation}
as $s\to \infty$ and $\nu\to0$ such that $k\in \mathbb N$ and $|x|<1/2$ remain fixed, where
\begin{equation}
K_{\textrm{Leg}}(x,y)=\frac{\kappa_{k-1}}{\kappa_k}\frac{
L_k(x)L_{k-1}(y)-L_k(y)L_{k-1}(x)}{x-y},
\end{equation}
and $L_k$ is the Legendre polynomial of degree $k$, orthonormal on $[-2,2]$:
\begin{equation} \int_{-2}^2L_j(x)L_i(x)dx=\delta_{ij}=\begin{cases} 0& \textrm{for $i\neq j$,}\\ 1& \textrm{for $i=j$.} \end{cases}
\end{equation}
 Recall that for a set $B\subset \mathbb R$ and a point process $\Lambda$ with its $m$-th correlation function denoted $r_m$, we have
\begin{equation} \label{tuples}
\begin{aligned}
&\textrm{Expectation}(\# \, \,\textrm{ordered } m\textrm{-tuples in }B)=
\frac{1}{m!} \int_{B^m}r_m(x_1,\dots,x_m)dx_1\dots dx_m
\\ &=\sum_{j=0}^\infty \binom{m+j}{m} \textrm{Prob}(\#(\Lambda \cap B)=m+j).
\end{aligned}\end{equation}
The process with kernel $K_{\textrm{Leg}}$ is a $k$-point process. Thus we obtain from \eqref{convergence KLeg} and the first equation of \eqref{tuples} that, as $s\to \infty$ and $\nu \to0$ such that $k\in \mathbb N$ and $|x|<1/2$ remain fixed, the expected number of $(k+1)$-tuples of $\mathbb P_A$ on $\left(-\nu,\nu\right)$ converges to $0$, while the expected number of $k$-tuples on the same interval converges to $1$. It follows from the second equation of \eqref{tuples} that
\begin{equation}
\textrm{Prob}\left(\mathbb P_A \textrm{ has $k$ particles in $\left(-\nu,\nu\right)$}\right)\to 1.
\end{equation}

\end{remark}

\bigskip

Thus the asymptotics of $\log P_s(A)$ as $s\to \infty$ depends on the value of $\nu$, and we give an overview of the various scaling limits:
 \begin{itemize}
\item If $\nu=0$, the asymptotics are given by \eqref{1 gap formula}.
\item  If $\nu \to 0$ as $s\to \infty$, such that $s\nu\log \nu^{-1}\to 0$, the asymptotics are given by Theorem \ref{Theorem u0s to 0}.
\item If $\nu\log \nu^{-1}$ is of order $1/s$ or larger, the asymptotics of Theorem \ref{Theorem u0s to 0} breaks down and the transition to the asymptotic formula (\ref{2 gap formula int}) containing $\theta$-functions takes place. 
This is discussed in Section \ref{sec5}.
 \item If $\nu>0$ is fixed, the asymptotics are given by the $\theta$-function regime \eqref{2 gap formula int}.
 \end{itemize}
\bigskip

For Toeplitz determinants, we obtain the following result.
Let $D_n(J)$ be given by (\ref{T}) with $J=J^{(n)}=J_1^{(n)}\cup J_2^{(n)}$ where 
\begin{equation}  \label{def J1J2}
J_1^{(n)}=\left\{e^{i\theta}: \theta\in \left(-\frac{2s\nu}{n},\frac{2s\nu}{n}\right)\right\},\qquad
J_2^{(n)}=\left\{e^{i\theta}: \theta\in \left(-\pi,\frac{2s\alpha}{n}\right)\bigcup\left(\frac{2s\beta}{n},\pi\right]\right\},
\end{equation}
with some $\alpha<0<\nu(s)<\beta$. Then, with the notation of Theorem \ref{Theorem u0s to 0}, we have

\begin{Thm} \label{prop toeplitz}
 As $s,n\to \infty$, uniformly for $\nu \in (0, \nu_0)$, where $ s\nu_0 \log \nu_0^{-1}\to 0$ and \\
 $s^3/n\to 0$,
 \begin{multline}\label{Toeplitz 2 arcs} 
 \log D_n \left(J^{(n)}\right)=
 \log D_n\left(J^{(n)}_2\right)
 +s\sqrt{|\alpha \beta|}\left(\omega-\frac{x^2}{\omega}\right)+c(k)+\delta_k(x)
\\+\mathcal O (\max \{s^3/n,s\nu_0\log \nu_0^{-1},1/\log \nu_0^{-1},s^{-1}\}),\quad \qquad
\end{multline} 
where the expansion of $\log D_n\left(J^{(n)}_2\right)$ is given in (\ref{Toeplitz 1 arc}) with $\theta_1=2s\beta/n$, $\theta_2=2s\alpha/n$.
\end{Thm}
We use Theorem \ref{prop toeplitz} to prove Theorem \ref{Theorem u0s to 0}.\\
\noindent{\it Proof of Theorem \ref{Theorem u0s to 0}.}
It is well-known that 
\begin{equation}\label{Fredholm Toeplitz}
 |D_n(J^{(n)})-\det(I-K_s)_A|\to 0
\end{equation}
as $n\to \infty$ for fixed $s$, a fact which we also prove in the appendix for the reader's convenience.
Taking the limit $n\to\infty$ in (\ref{Toeplitz 2 arcs}), we then obtain (\ref{equation 1}). 
To obtain (\ref{equation 2}), we substitute (\ref{1 gap formula}) for $P_s(A_0)$, and note that the standard asymptotics of
the Barnes G-function $G(z+1)$ as $z\to\infty$
\begin{equation}
\log G(z+1)=\frac{z^2}{2}\log z-\frac{3}{4}z^2+\frac{z}{2}\log 2\pi-\frac{1}{12} \log z+\zeta'(-1)+\mathcal O(z^{-2}),
\end{equation} 
imply that as $k\to \infty$, 
 \begin{equation}
c(k)=-\frac{1}{4}\log k+\frac{1}{12}\log 2+3\zeta'(-1)+\mathcal O(1/k^2).
\end{equation}
Furthermore, we note that $2\pi \kappa_k^2=1+\bigO (k^{-1})$ as $k\to \infty$.
$\Box$

\subsection{Outline of the proof of Theorem \ref{prop toeplitz}}
It remains to prove Theorem \ref{prop toeplitz}. 
%The proof below is based on a differential identity for $D_j(J^{(n)})$ and on an
%asymptotic analysis of the orthogonal polynomials associated with this Toeplitz determinant.
\begin{figure}
\begin{tikzpicture}
\draw [decoration={markings, mark=at position 0 with {\arrow[thick]{>}}},
        postaction={decorate},decoration={markings, mark=at position 0.5 with {\arrow[thick]{>}}},
        postaction={decorate},dashed, gray] (0,0) circle[radius=2]; 
\draw[thick](2,0) arc [radius=2,start angle=0, end angle=10];
\draw[thick](2,0) arc [radius=2,start angle=360, end angle=350];
\draw[decoration={markings, mark=at position 0.99 with {\coordinate(b1);}},
        postaction={decorate},thick](-2,0) arc [radius=2,start angle=180, end angle=40];
\draw[decoration={markings, mark=at position 0.99 with {\coordinate(b2);}},
        postaction={decorate},thick](-2,0) arc [radius=2,start angle=180, end angle=330];
\node [left] at (2,0) {$J_1$};
\node [left] at (-2,0) {$J_2$};
\node [left] at (1.8,0.8) {$\Sigma_1$};
\node[left] at (1.8,-0.7){$\Sigma_2$};
\node [right] at (2,0.5) {$a=e^{i\theta_0}$};
\node [right] at (2,-0.5) {$\bar{a}=e^{-i\theta_0}$};
\node [right] at (b1) {$b_1=e^{i\theta_1}$};
\node [right] at (b2) {$b_2=e^{i\theta_2}$};
\end{tikzpicture}
\caption{Interval $J$.}
\label{interval J}
\end{figure}

Let $-\pi<\theta_2<0<\theta_0<\theta_1<\pi$ and define $J=J_1\cup J_2$ where $J_1, \, J_2$ are as in Figure \ref{interval J}:
\begin{equation}\label{defJ}
J_1=J_1(\theta_0)= \{e^{i\theta}|\theta \in( -\theta_0,\theta_0)\}, \quad J_2=\{ e^{i\theta}|\theta \in (\theta_1,\pi]\cup(-\pi,\theta_2)\}.\end{equation} 
We denote the complement of $J$ as $\Sigma=\Sigma_1\cup \Sigma_2$ where
\begin{equation}
\Sigma_1=\{ e^{i\theta}|\theta \in (\theta_0,\theta_1)\}, 
\quad \Sigma_2= \{e^{i\theta}|\theta \in( \theta_2,-\theta_0)\}.
\end{equation}
It follows from the integral representation for Toeplitz determinants (see \eqref{Toeplitz eigs} in the Appendix) that $D_j(J)>0$ for all $j\in \mathbb N$. 
Consider the polynomials $\phi_j$ for $j=0,1,2,\dots$ given by
\begin{equation}\begin{aligned} \label{def orthog}
\phi_0(z)&=\frac{1}{\sqrt{D_1(J)}},\\
\phi_j(z)&=\frac{1}{\sqrt{D_j(J)D_{j+1}(J)}}\det 
{\small \begin{pmatrix} f_0 & f_{-1} & \dots &f_{-j+1} &f_{-j}\\
f_1& f_0 &\dots&f_{-j+2}&  f_{-j+1}\\
&&\ddots &\\
f_{j-1}&f_{j-2}& \dots& f_0 &f_{-1}\\
1&z&\dots &z^{j-1}&z^j \end{pmatrix}}=\chi_jz^j+\dots,\quad j>0,
\end{aligned}\end{equation}
where the leading coefficient $\chi_j$ is given by
\begin{equation}
\chi_j=\sqrt{\frac{D_j(J)}{D_{j+1}(J)}}, \quad j=0,1,2,\dots,
\end{equation}
and we set $D_0(J)=1$. The polynomials $\phi_j$ are orthonormal with weight $1$ on $J$:
\begin{equation} \label{property orthog}
\int_{J} \phi_k(z)\overline{ \phi_j(z)} \frac{d\theta}{2\pi}= \delta_{jk}, \quad z =e^{i\theta}, \quad j,k=0,1,2,\dots.
\end{equation}

Define a $2\times 2$ matrix $Y(z)=Y_n(z)$ in terms of these orthogonal polynomials as follows:
\begin{equation}\label{Soln Y}
Y(z)=\begin{pmatrix}
\chi_n^{-1}\phi_n(z)&\chi_n^{-1} \int_{J} \frac{\phi_n(\zeta)}{\zeta-z}\frac{d\zeta}{2\pi i \zeta^n} \\
-\chi_{n-1}z^{n-1} \overline{\phi_{n-1}}(z^{-1})& -\chi _{n-1} \int_{J} \frac{\overline{\phi_{n-1}}(\zeta^{-1})}{\zeta-z} \frac{d\zeta}{2\pi i \zeta}
\end{pmatrix}.
\end{equation}
Then $Y$ is the unique solution of the following Riemann-Hilbert (RH) Problem
\begin{itemize}
\item[(a)] $Y:\mathbb C \setminus J \to \mathbb C^{2\times 2}$ is analytic;
\item[(b)] $Y$ possesses $L^2$ boundary values $Y_+$ and $Y_-$ on the $+$ and $-$ side of $J$, respectively, related by the condition:
\begin{equation} \nonumber Y_+(z)=Y_-(z)\begin{pmatrix} 1&z^{-n}\\0&1 \end{pmatrix}$ for $z\in J; \end{equation}
\item[(c)] $Y(z)=(I+\mathcal{O}(1/z))\begin{pmatrix}z^n&0\\0&z^{-n}\end{pmatrix}$ as $z\to \infty$.
\end{itemize}
The fact that orthogonal polynomials satisfy a RH problem was first observed for polynomials orthogonal on the real line by Fokas, Its, Kitaev \cite{FIK}, and extended to polynomials orthogonal on the unit circle by Baik, Deift, Johansson \cite{BDJ}.  
The RH problem provides an efficient tool, via the Deift-Zhou steepest descent method, for the asymptotic analysis of the polynomials, see e.g. \cite{Deift}.

In Section \ref{section diffid} we express the logarithmic derivative of the Toeplitz determinant $\frac{d}{d\theta_0}\log D_n(J)$ in terms of the polynomials $\phi_{n}$ and $\phi_{n-1}$. These are, in turn, expressed in terms of $Y_n$. In Sections \ref{section analysis RH} and \ref{section asymptotic} we analyse the RH problem for $Y_n$ as $n\to \infty$ in a double scaling limit where $J$ depends on $n$ such that $\theta_j=\frac{s}{n}u_j$ for $j=0,1,2$; where $s\to \infty$ such that
 $s^3/n \to0$; where $u_0\to 0$ such that $s u_0\log u_0^{-1} \to 0$, while $u_1$ and $u_2$ remain fixed. As a result, we obtain the asymptotics of $Y_n$. Substituting these into the differential identity for $\frac{d}{d\theta_0}\log D_n(J)$, and integrating with respect to $\theta_0$, we obtain the asymptotics of $D_n(J)$, where $u_1=2\beta,\,\, u_2=2\alpha$, and $u_0=u_0(\nu)$ is a function of $\nu$, which proves Theorem \ref{prop toeplitz}.

\section{Differential Identity} \label{section diffid}

We will now obtain the following:

\begin{Prop} \label{Prop Diffid} (Differential identity) Let $a=e^{i\theta_0}$. The Toeplitz determinant $D_n(J)$ satisfies
\begin{equation}\label{diff id}
\frac{\partial}{\partial \theta_0}\log D_n(J)=-\frac{1}{2\pi}[F(\overline a)+F(a)],
\end{equation}
where 
\begin{equation} \label{F(z)} F(z)=n\chi_n^2|Y_{11}(z)|^2-2 \chi_n^2 \Re\left(z\overline{Y_{11}(z)}\frac{d}{dz}Y_{11}(z) \right),
\end{equation}
and $J$ was given in \eqref{defJ}.
\end{Prop}

\begin{proof}
 From the definition of the orthogonal polynomials it is clear that
\begin{align}\label{Toeplitz and coeffs}
D_n(J)=&\prod_{j=0}^{n-1} \chi_j^{-2}.\end{align}
The orthogonality conditions imply that
\begin{align} \frac{1}{2\pi} \int_J \frac{\partial \phi_j(z)}{\partial \theta_0} \overline{\phi_j(z)}d\theta =& \frac{1}{2\pi} \int_J \frac{\partial \chi_j}{\partial \theta_0} (z^j +\textrm{poly of deg $j-1$}) \overline{\phi_j(z)}d \theta
=\frac{1}{\chi_j} \frac{\partial \chi_j}{\partial \theta_0} ,
\end{align}
and similarly,
\begin{equation}\label{diff tricks}
 \frac{1}{2\pi} \int_J \phi_j(z)\frac{\partial \overline{\phi_j(z)}}{\partial \theta_0}d\theta=\frac{1}{\chi_j} \frac{\partial \chi_j}{\partial \theta_0} .
\end{equation}
By \eqref{Toeplitz and coeffs}--\eqref{diff tricks} we obtain:
\begin{equation}\label{Dnsumphij}
\begin{aligned}
  \frac{\partial }{\partial \theta_0} \log (D_n(J))=&-2\sum_{j=0}^{n-1} \frac{\partial \chi_j}{\partial \theta_0} /\chi_j\\
=&-\frac{1}{2\pi} \int_J \frac{\partial}{\partial\theta_0} \left( \sum_{j=0}^{n-1} |\phi_j(z)|^2 \right)d\theta.
\end{aligned}
\end{equation}
On the other hand, one can express $F$ given in \eqref{F(z)} in terms of the orthogonal polynomials:
\begin{equation}F(z)=n|\phi_n(z)|^2 -2\Re \left(z\overline{\phi_n(z)}\frac{d}{dz}\phi_n(z)\right).\end{equation}
Now the Christoffel-Darboux formula for orthogonal polynomials gives
\begin{equation} \label{CD formula}
\sum_{k=0}^{n-1}| \phi_k(z)|^2=-F(z) \quad \textrm{for } z \in C,
\end{equation}
(see eg. \cite{K04}), and hence \eqref{Dnsumphij} can be written as
\begin{equation}
\frac{\partial}{\partial \theta_0}\log D_n(J)=  \frac{1}{2\pi} \int_J \frac{\partial}{\partial \theta_0} \left(F(z)\right) d\theta.
\end{equation}
Since, by \eqref{CD formula} and orthogonality, $\int_JF(z)\frac{d\theta}{2\pi}=n$, we obtain
\begin{align}
0=\frac{\partial}{\partial \theta_0} \left(\int_{J} F(z) d\theta\right)=&
F(\overline a)+F(a)+\int_{J} \frac{\partial}{\partial \theta_0}F(z)d\theta,
\end{align}
 upon which Proposititon \ref{Prop Diffid} follows immediately.\end{proof}

\section{Analysis of Riemann-Hilbert problem} \label{section analysis RH}
We start by setting $\theta_0=0$ so that $J_1$ is a point, and then let $J_1$ develop into an arc. Throughout the rest of the paper we use the notation
\begin{equation}a=e^{i\theta_0}=e^{iu_0s/n}, \quad b_1=e^{i\theta_1}=e^{iu_1s/n}, \quad
b_2=e^{i\theta_2}=e^{iu_2s/n}
.\end{equation}
 We let $s,n \to \infty$ such that $s^3/n\to 0$, and let $u_0 \to 0$ as $s\to \infty$ such that $su_0 \log u_0^{-1} \to 0$, while $u_2<0<u_1$ remain constant. Denote $\Sigma^o=\Sigma_1^o\bigcup \Sigma_2^o$ where
\begin{align}
 \Sigma_1^o=\{e^{i\theta}| 0<\theta<\theta_1\},\quad
 \Sigma_2^o=\{e^{i\theta}| \theta_2<\theta<0\}.
\end{align}
Let $g_1$ be the function:
\begin{equation}\begin{aligned} \label{g_1}
g_1(z) &=\log \left( \frac{1}{b_1^{1/2}+b_2^{1/2}} \left( z+ \left(b_1b_2\right)^{1/2}+((z-b_1)(z-b_2))^{1/2}\right)\right),
\end{aligned} \end{equation}
where the square root has branch cut on $J_2$ and is positive as $z\to +\infty$, and the logarithm $\log x$ has a branch cut for $x<0$ and is positive for $x>1$. At infinity,
\begin{equation}
g_1(z)= \log z -\log \frac{\sqrt{b_1}+\sqrt{b_2}}{2}+o(1) \quad \textrm{as }z\to \infty.
\end{equation}
The boundary values of the function $g_1$ satisfy
\begin{equation}\label{jump g1}
g_{1,+}(z)+g_{1,-}(z)=\log z, \quad \textrm{for } z\in J_2,
\end{equation}
 and at $b_1,b_2$ we have
\begin{equation} \label{value gb1}
g_1(b_1)=\frac{1}{2}\log b_1, \quad g_1(b_2)=\frac{1}{2}\log b_2.
\end{equation}
Alternatively, for $z=e^{i\theta}\in \Sigma^{o}$ we can write $g_1$ in the following form:
\begin{equation} \label{g1 Sigma}
\exp\left(g_1\left(e^{i\theta}\right)\right)=e^{i\theta/2}\frac{\cos \frac{1}{2}\left(\theta-\frac{\theta_1+\theta_2}{2}\right)+\sqrt{| \sin \frac{\theta-\theta_1}{2}\sin \frac{\theta-\theta_2}{2}}|}{\cos \frac{1}{4} (\theta_1-\theta_2)}.
\end{equation}
On the $+$ and $-$ side of $J_2$, we have
\begin{equation} \label{g1 J2}
\exp\left((g_1)_{\pm}\left(e^{i\theta}\right)\right)=e^{i\theta/2}\frac{\cos \frac{1}{2}\left(\theta-\frac{\theta_1+\theta_2}{2}\right)\mp i\sqrt{| \sin \frac{\theta-\theta_1}{2}\sin \frac{\theta-\theta_2}{2}}|}{\cos \frac{1}{4} (\theta_1-\theta_2)},
\end{equation}
from which it follows that $e^{g_1}$ maps the $+$ side of $J_2$ to $\Sigma^o$ and the $-$ side to $J_2$, and that $e^{g_1}$ maps $\mathbb C \setminus J_2$ to the outside of the unit disc.

Set 
\begin{equation} \label{def w,Bj} w=\frac{n}{s}(z-1), \quad B_j=\frac{n}{s}(1-b_j) \end{equation} 
for $j=1,2$. Note that $B_1$ and $B_2$ remain bounded as $s/n\to 0$. Then 
\begin{equation}\begin{aligned}
g_1(z)&=g_1(1)+\log\left(1+\frac{s}{n}\frac{\left(1-\frac{s}{n}B_1\right)^{1/2}+\left(1-\frac{s}{n}B_2\right)^{1/2}}{e^{g_1(1)}}wH(w) \right),\\
H(w)&=\frac{1}{w}\left(w+((w+B_1)(w+B_2))^{1/2}-(B_1B_2)^{1/2}\right),
\end{aligned}\end{equation} 
where $H(w)$ is analytic in $w$ at the point $0$.
Thus $g_1$ has the following expansion in $w$ at the point $w=0$
\begin{equation}\label{expansion g1a}
g_1(z)=g_1(1)+\frac{s}{n}(c_1w+\mathcal O(w^2)),
\end{equation}
where
\begin{equation}
 \label{expansion g1} \begin{aligned}g_1(1)&=\log \left( \frac{\cos (\theta_1+\theta_2)/4+\sqrt{|\sin \theta_1/2| |\sin \theta_2/2|}}{\cos (\theta_1-\theta_2)/4} \right)=\frac{s\sqrt{|u_1u_2|}}{2n}\left(1+\mathcal O\left(\frac{s}{n}\right)\right)\\
c_1&=\frac{1}{\sqrt{(1-b_1)(1-b_2)}}\left(1-\frac{\sqrt{b_1}+\sqrt{b_2}}{2e^{g_1(1)}}\right)=\left(\frac{1}{2}+\frac{i}{4}(u_1^{-1}+u_2^{-1})\sqrt{|u_1u_2|}\right)\left(1+\mathcal O\left(\frac{s}{n}\right)\right),
\end{aligned}\end{equation}
as $\frac{s}{n}\to 0$.

Define 
\begin{equation}
r(z)=((z-b_1)(z-b_2))^{1/2},
\end{equation}
where the square root has a branch cut on $ J_2$, and is positive as $z\to +\infty$. 
Let 
\begin{equation}
 h(z)=r(z)\int_{\Sigma^o_2} \frac{d\xi}{r(\xi)(\xi-z)},\qquad z\in\mathbb C\setminus(J_2\cup\Sigma_2^o),
\end{equation}
where integration is taken in counter-clockwise direction.
It is easily verified by differentiation that
\begin{equation}
-r(z)\int_{\widetilde C}^t \frac{d\xi}{r(\xi)(\xi-z)}=\log \left(
\frac{2r^2(z)+(2z-b_1-b_2)(t-z)+2r(z)r(t)}{t-z}
\right)+C(z,\tilde C),
\end{equation}
for any constant $\widetilde C$ and some function $C(z)$. Thus
\begin{equation}
\label{def h}\begin{aligned}
 &h(z)=r(z)\int_{b_2}^1 \frac{d\xi}{r(\xi)(\xi-z)}\\
 &=\log \frac{b_2-b_1}{2}(z-1)-\log \left(z\left(1-\frac{b_1+b_2}{2}\right) +b_1b_2 - \frac{b_1+b_2}{2}+r(z)r(1)\right).
\end{aligned}\end{equation}
 The function $h$ has a logarithmic singularity at $z=1$ and a jump on $\Sigma_2^o \bigcup J_2$, such that
 \begin{equation} \label{jump h}
\begin{aligned}
h_+-h_-&=\begin{cases} 0& \textrm{for } z\in \Sigma_1^o,\\  2\pi i&\textrm{for } z\in \Sigma_2^o, \end{cases}\\
h_++h_-&=0 \quad \textrm{for } z \in J_2.
\end{aligned}
\end{equation}
The jump conditions \eqref{jump h} also imply that
\begin{equation} \label{value hb1}
\begin{aligned} h(b_1)&=0\\
h_+(b_2)&=-h_-(b_2)=\pi i. \end{aligned}
\end{equation}
As $z \to \infty$, 
\begin{equation}\label{hinfty}
h(z) \to \log \frac {b_2-b_1}{\left((1-b_1)^{1/2}+(1-b_2)^{1/2}\right)^2}\equiv h(\infty)
.\end{equation}
On the interval $\Sigma^o$ we can alternatively write $h$ in the following form:
\begin{equation} \label{h Sigma}
\exp(h(z))=\frac{\sin \frac{\theta_1-\theta_2}{2}\sin \frac{\theta}{2}}{\cos\frac{\theta-\theta_1-\theta_2}{2}-\cos \frac{\theta_1-\theta_2}{2}\cos \frac{ \theta}{2}+2\sqrt{|\sin \frac{\theta-\theta_1}{2}\sin \frac{\theta-\theta_2}{2}\sin \frac{\theta_1}{2}\sin \frac{\theta_2}{2}}| }  .
\end{equation}
With the notation of \eqref{def w,Bj} we can write
\begin{equation}
h(z)=\log w\frac{B_1-B_2}{(B_1+B_2)w+2B_1B_2+2\sqrt{(w+B_1)(w+B_2)}\sqrt{B_1B_2}}.
\end{equation}
Then we can expand $h$ at the point $z=1$:
\begin{equation} \label{expansionh}
h(z)=\log w+c_0'+c_1'w + \mathcal O(w^2),
\end{equation}
where
\begin{equation}
\label{expansions} \begin{aligned}c_0'&=-\log 4 \frac{B_1B_2}{B_1-B_2}=\left(-\frac{\pi i}{2}+\log \frac{(u_1^{-1}-u_2^{-1})}{4}\right)+\mathcal O\left(\frac{s}{n}\right),\\
c_1'&=-\frac{B_1+B_2}{2B_1B_2}=-\frac{i}{2}(u_1^{-1}+u_2^{-1})+\mathcal O\left(\frac{s}{n}\right),
\end{aligned}\end{equation}
as $\frac{s}{n}\to 0$.

We define the $g$-function by:
\begin{align} \label{def g birth}
g(z)=g_1(z)+\frac{\Omega}{2\pi} h(z),
\end{align}
where $\Omega>0$ is a constant yet to be fixed. The jump conditions for $g_1$ and $h$ imply that
\begin{equation}\label{jump g}
\begin{aligned}
g_++g_-&=\log z && \textrm{for } z\in J_2,\\
g_+-g_-&=0 && \textrm{for } z\in \Sigma_1^o,\\
g_+-g_-&=i\Omega && \textrm{for } z\in \Sigma_2^o.
\end{aligned}\end{equation}

 We define the local variable $ \zeta$ on a disc $U_0$ containing the interval $J_1$ (but not the points $b_1,b_2$), by
\begin{equation} \label{def zeta}
 \zeta(z)=e^{\frac{2\pi}{\Omega}\left(g(z)-\frac{1}{2}\log z\right)}
=e^{h(z)+\frac{2\pi}{\Omega}\left(g_1(z)-\frac{1}{2}\log z\right)}.
\end{equation}
The jump conditions for $g$ imply that the function $\zeta$ is analytic in $U_0$. The precise radius of $U_0$ will be determined later on by requiring that the mapping $\zeta$ be conformal on $U_0$.

Since $e^{g_1}$ maps $\mathbb C\setminus J_2$ to the exterior of the unit disc,  we have
\begin{equation} \label{eg1>0}
e^{g_1(e^{\pm i \theta_0})\mp i\theta_0/2}>1.\end{equation}
For $u_0<\epsilon$ with some $\epsilon>0$, it follows from \eqref{h Sigma} that
\begin{equation}\label{eh}
 e^{h(e^{\pm i \theta_0})}\in \begin{cases} (0,1) &\textrm{for ``$+$",}\\
(-1,0)&\textrm{for ``$-$".} \end{cases}
\end{equation}
Now consider, as a function of $\Omega$,
\begin{equation} \label{zeta-zeta}
 \zeta(a)-\zeta(\overline a). \end{equation}
By \eqref{eg1>0} and \eqref{eh} it follows that if we let $\Omega=+\infty$ in \eqref{def zeta} then  \eqref{zeta-zeta} is smaller than 2, and if we instead set $\Omega = +0$ then \eqref{zeta-zeta} is equal to $+\infty$. Since \eqref{def Omega} is monotone in $\Omega$, there exists, for $u_0<\epsilon$, a unique value for $\Omega>0$ such that
\begin{equation} \label{def Omega}
 \zeta(a)-\zeta(\overline a)=4. \end{equation} 
We define $\Omega$ so that $\zeta$ satisfies \eqref{def Omega}.
From \eqref{g1 Sigma} and \eqref{h Sigma}, it follows that $\zeta(\Sigma^o)\subset \mathbb R$.
By \eqref{expansion g1a} and \eqref{expansionh} we have the following expansion at the point $z=1$:
\begin{equation} \label{zeta exp} \begin{aligned}
\zeta(z)=&w\zeta_0
 \left(1+\zeta_1w+\mathcal O\left(\left(\frac{sw}{n\Omega}\right)^2\right)\right),
 \qquad w=\frac{n}{s}(z-1),
 \\
 \zeta_0&=e^{c_0'+\frac{2\pi}{\Omega}g_1(1)},\\
 \zeta_1&=c_1'+\frac{2\pi s}{n \Omega}(c_1-1/2).\\
\end{aligned} \end{equation}

In what follows, it will be apparent that $\Omega \to 0$ in the limit $s,n\to \infty$ and $u_0,s/n \to 0$. Moreover, by \eqref{expansion g1} and \eqref{expansions},
\begin{equation}\label{zeta0zeta1}
\begin{aligned}
\zeta_0&=-\frac{i}{4}(u_1^{-1}-u_2^{-1})e^{\frac{\pi}{\Omega}\frac{s\sqrt{|u_1u_2|}}{n}(1+\mathcal O(s/n))}(1+\mathcal O(s/n))\\
\zeta_1&=\frac{2\pi s}{n\Omega}(c_1-1/2)-\frac{i}{2}(u_1^{-1}+u_2^{-1})+\mathcal O(s/n).
\end{aligned}
\end{equation}
Substituting these expansions into \eqref{zeta exp}, which we in turn substitute into \eqref{def Omega}, we obtain 
\begin{equation} \label{expansion u0}
u_0\left(1+\mathcal O\left( \frac{1}{\Omega^2}\frac{u_0^2 s^2}{n^2}\right)\right)=\frac{8}{u_1^{-1}-u_2^{-1}}e^{-\frac{2\pi}{n\Omega}\frac{s\sqrt{|u_1u_2|}}{2}(1+\mathcal O(s/n))}(1+\mathcal O(s/n)),
\end{equation}
or upon taking the logarithm,
\begin{equation}
\log\left(u_0\frac{u_1^{-1}-u_2^{-1}}{8}\right)^{-1}=\frac{\pi s \sqrt{|u_1u_2|}}{\Omega n}\left(1+\mathcal O(s/n)+\mathcal O(\Omega)+\mathcal O \left( \frac{1}{\Omega}\frac{u_0^2s}{n}\right)\right).
\end{equation}
Therefore,
\begin{equation}\label{expansion Omega} 
\Omega=\frac{\pi s \sqrt{|u_1u_2|}}{n}\Big/\log \frac{8}{(u_1^{-1}-u_2^{-1})u_0}\left( 1+\mathcal O(s/n)+\mathcal O\left(u_0^2\log u_0^{-1}\right)\right).
\end{equation}
%+
%\zeta_2\frac{s^2}{n^2}\left(\frac{8}{u_1^{-1}-u_2^{-1}}e^{-\frac{2\pi}{n\Omega}\frac{s\sqrt{|u_1u_2|}}{2}}\right)^2

Using the definition of $g_1$, $h$ and $\zeta$ in \eqref{g_1}, \eqref{def h}, \eqref{def zeta}, and the expansion of $\Omega$ in \eqref{expansion Omega}, it is easily seen that there are constants $m_1<m_2$ independent of $s,n,u_0$ such that $\zeta'(z)$ has at least one zero in the set
\begin{equation}
\left\{z:\frac{sm_1}{n\log u_0^{-1}}<|z-1|<\frac{sm_2}{n\log u_0^{-1}}\right\}.
\end{equation}
By \eqref{zeta0zeta1} and expansion \eqref{expansion Omega}, $m_1$ may be chosen such that 
\begin{equation}
\left|\zeta_1 w+\mathcal O \left( \left( \frac{sw}{n\Omega}\right)^2 \right)\right|<1
\end{equation}
as $\frac{sw}{n\Omega} \to 0$
and so $\zeta$ is conformal on the following disc
\begin{equation} \label{U0}
\left\{z:|z-1|<\frac{sm_1}{n\log u_0^{-1}}\right\},
\end{equation}
for $u_0,s/n<\epsilon$ for some fixed $\epsilon>0$. Thus we define $U_0$ to be the set \eqref{U0}.

We define $\widetilde g$ as 
\begin{equation}\label{def tilde g}
\widetilde g=\lim_{z\to\infty} e^{g(z)-\log z}=  \frac{2}{b_1^{1/2}+ b_2^{1/2}}\left(\frac{b_2-b_1}{\left((1-b_1)^{1/2}+(1-b_2)^{1/2}\right)^2}\right)^\frac{\Omega}{2\pi} ,
\end{equation}
and $T$ as
\begin{equation} \label{defT}
T(z)=\widetilde g^{n \sigma_3}Y(z)e^{-ng(z)\sigma_3}.
\end{equation}
It follows by \eqref{jump g} that $T$ satisfies the following RH problem:
\begin{itemize}
\item[(a)] $T:\mathbb C \setminus C \to \mathbb C^{2\times 2}$ is analytic.
\item[(b)] $T$ has the following jumps on $C$:
\begin{align*}
T_+&=T_-\begin{pmatrix}
e^{n(g_--g_+)(z)}&1\\ 0& e^{n(g_+-g_-)(z)}
\end{pmatrix} &&  \textrm{for $z\in J_2$,}\\
T_+&=T_- e^{-in \Omega\sigma_3} && \textrm{for $z\in \Sigma_2$.}\\
T_+&=T_- \begin{pmatrix} e^{-in\Omega}&z^{-n}e^{n(g_++g_-)}\\0&e^{in\Omega} \end{pmatrix} && \textrm{for $z\in (\overline a,1)=\Sigma_2^o\bigcap J_1$,}\\
T_+&=T_- \begin{pmatrix}1&z^{-n}e^{2ng}\\0&1\end{pmatrix} && \textrm{for $z\in (1,a)=\Sigma_1^o \bigcap J_1$.}\\
T_+&=T_-  && \textrm{for $z \in \Sigma_1$.}
\end{align*}
\item[(c)] As $z\to \infty$,
\[ T(z)= I+\mathcal O (z^{-1}). \]
\end{itemize}
The jump of $T$ on $J_2$ factorizes as
\begin{equation}
\begin{pmatrix}
e^{n(g_--g_+)(z)}&1\\ 0& e^{n(g_+-g_-)(z)}
\end{pmatrix} =
\begin{pmatrix}1&0\\ e^{n(g_+-g_-)(z)}&1\end{pmatrix}
\begin{pmatrix} 0&1\\-1&0\end{pmatrix}
\begin{pmatrix}1&0\\e^{n(g_--g_+)(z)}&1\end{pmatrix}.
\end{equation}

\begin{figure}
\begin{tikzpicture}
\draw [decoration={markings, mark=at position 0 with {\arrow[thick]{>}}},
        postaction={decorate},decoration={markings, mark=at position 0.5 with {\arrow[thick]{>}}},
        postaction={decorate},dashed, gray] (0,0) circle[radius=2]; 
\draw[thick](2,0) arc [radius=2,start angle=0, end angle=10];
\draw[thick](2,0) arc [radius=2,start angle=360, end angle=350];
\draw[decoration={markings, mark=at position 0.99 with {\coordinate(b1);}},
        postaction={decorate},thick](-2,0) arc [radius=2,start angle=180, end angle=40];
\draw[decoration={markings, mark=at position 0.99 with {\coordinate(b2);}},
        postaction={decorate},thick](-2,0) arc [radius=2,start angle=180, end angle=330];
\node [right] at (2,0.5) {$a$};
\node [right] at (2,-0.5) {$\bar{a}$};
\node [right] at (b1) {$b_1$};
\node [right] at (b2) {$b_2$};
\node [right] at (-1.5,0) {$\Gamma_S^{\textrm{In}}$};
\node [left] at (-2.5,0) {$\Gamma_S^{\textrm{Out}}$};
\draw[decoration={markings, mark=at position 0.5 with {\arrow[thick]{>}}},
        postaction={decorate}] (b1)  to [out=90,in=0] (0,2.5) to [out=180,in=90] (-2.5,0) to [out=-90,in=180] (0,-2.5) to [out=0,in=-90]  (b2);
\draw[decoration={markings, mark=at position 0.5 with {\arrow[thick]{>}}},
        postaction={decorate}] (b1)  to [out=180,in=0] (0,1.5) to [out=180,in=90] (-1.5,0) to [out=-90,in=180] (0,-1.5) to [out=0,in=-180]  (b2);
\end{tikzpicture}
\caption{Contour $\Gamma_S$.}\label{Contour GammaS2}
\end{figure}

Define 
\begin{equation} \label{def phi} \phi(z)=2g(z)-\log(z),\end{equation}
 Then the jumps of $ e^{\phi}$ are induced by $g$ and we obtain by \eqref{jump g} that for $z\in J_2.$
\begin{equation} \begin{aligned} \label{phi+J}
\exp( \phi_+(z))&=\exp[(g_+(z)+g_-(z)-\log(z))+g_+(z)-g_-(z)]\\ &=\exp(g_+(z)-g_-(z)),\end{aligned} \end{equation}
and similarly
\begin{align}
& \exp (\phi_-(z))=\exp(g_-(z)-g_+(z)).\label{phi-J}
\end{align}
 We proceed to open the lenses around $J_2$ as in Figure \ref{Contour GammaS2}.
\begin{Prop} For $z$ on the edges of the lense $\Gamma_S^{\textrm{in}}\cup \Gamma_S^{\textrm{out}}$ in Figure \ref{Contour GammaS2} such that $|z-b_1|,|z-b_2|>\epsilon s/n$ for some fixed $\epsilon>0$, there exists a constant $C>0$ independant of $s,n,z$ such that 
\begin{equation}
e^{-ng(z)}<e^{-sC},
\end{equation}
as $s,n \to \infty$ and for $u_0$ sufficiently small.
\end{Prop}
\begin{proof}
Since $e^{g_1}$ sends $\mathbb C \setminus J_2$ to the outside of the unit disc, it is clear that for $|z|<1$ we have 
\begin{equation}\label{ineq g1} |e^{2g_1(z)-\log (z)}|>1. \end{equation}
Let $g_1^{(-)}$ denote the function defined as in \eqref{g_1},
 but with $+((z-b_1)(z-b_2))^{1/2}$ replaced with $-((z-b_1)(z-b_2))^{1/2}$
  Then $e^{g_1^{(-)}}$ maps $\mathbb C \setminus J_2$ to the inside of the circle and for $z\in \mathbb C \setminus J_2$ we have the relation
\begin{equation}e^{g_1(z)+g_1^{(-)}(z)}=z. \end{equation}
It follows that if $|z|>1$, then 
\begin{equation} \label{ineq g1 2}
|e^{2g_1(z)-\log z}|=|e^{\log z- 2g_1^{(-)}(z)}|>1.
\end{equation}
Using \eqref{ineq g1}, \eqref{ineq g1 2}, the definition of $g$, and the fact that $\Omega=\mathcal O(s/(n\log u_0^{-1})),$
as $s/(n\log u_0^{-1})\to 0$, it follows that $e^{-\phi(z)}$ lies in interior of the unit disc for $z$ sufficiently close to the interval $J_2$, and in particular that 
\begin{equation}e^{-n\phi(z)}=\mathcal O\left(e^{-cn}
\right), \end{equation}
uniformly  for $z$ on the lense that is opened around $J_2$ in Figure \ref{Contour GammaS2}, except near the endpoints $b_1$ and $b_2$, for some constant $c>0$. 

Consider $h(z)$ and $g_1(z)$ at $z=b_1$. Let $w_1=\frac{n}{s}(z-b_1)$. From \eqref{def h} we have, with $B_j=\frac{n}{s}(1-b_j)$, 
\begin{equation}\label{hw1}
h(z)=\log \frac{B_1-B_2}{2}(w_1-B_1)-\log \left(
w_1\frac{B_1+B_2}{2}+\frac{B_2-B_1}{2}B_1+(w_1(w_1+B_2-B_1)B_1B_2)^{1/2}
\right).
\end{equation}
It follows from \eqref{jump h} and \eqref{hw1} that $h(z)/w_1^{1/2}$ is analytic at $w_1=0$, and we have
\begin{equation}\label{hexpb1}
h(z)=-w_1^{1/2} \frac{2}{(B_2-B_1)^{1/2}}\left(\frac{B_2}{B_1}\right)^{1/2}(1+\mathcal O(w_1)).
\end{equation}
Likewise, we let $w_2=\frac{n}{s}(z-b_2)$. Then at the point $w_2=0$ we have the expansion
\begin{equation}\label{hw2}
h(z)=\pm \pi i -w_2^{1/2} \frac{2}{(B_1-B_2)^{1/2}}\left(\frac{B_1}{B_2}\right)^{1/2}(1+\mathcal O(w_2)),
\end{equation}
where $\pm$ means $+$ on the $+$side and $-$ on the $-$side of the unit circle $C$ (so the jumps agree with \eqref{jump h}).

We evaluate $g_1$ using the definition \eqref{g_1}:
\begin{equation}\label{g1w1}
g_1(z)=\log \frac{b_1+(b_1b_2)^{1/2} +\frac{s}{n}((w_1(w_1+B_2-B_1))^{1/2}+w_1)}{b_1^{1/2}+b_2^{1/2}}.
\end{equation}
From \eqref{g1w1} and \eqref{jump g1}, we have that $(g_1(z)-\frac{\log z}{2})/w_1^{1/2}$ is analytic at $w_1=0$, and that at the point $w_1=0$ we have the expansion
\begin{equation} \label{g1expb1}
g_1(z)-\frac{\log z}{2}=\frac{s}{n}w_1^{1/2}\frac{(B_2-B_1)^{1/2}}{b_1+(b_1b_2)^{1/2}}(1+\mathcal O(w_1)).
\end{equation}
Likewise we expand $g_1$ at the point $w_2=0$:
\begin{equation} \label{g1expb2}
g_1(z)-\frac{\log z}{2}=\frac{s}{n}w_2^{1/2}\frac{(B_1-B_2)^{1/2}}{b_2+(b_1b_2)^{1/2}}(1+\mathcal O(w_2)).
\end{equation}

Consider now a neighbourhood of $b_1$.
The error terms in \eqref{hexpb1}, \eqref{g1expb1} are uniform for $0<s/n<\delta$ for some sufficiently small $\delta$.
By \eqref{hexpb1}, \eqref{g1expb1} and the fact that $\Omega=\mathcal O\left(\frac{s}{n\log u_0^{-1}}\right)$,
it follows that there is a constant $C_1>0$ independant of $s,n,z$ for $u_0$ sufficiently small such that
\begin{equation} \label{ineqg}
\Re (g)> w_1^{1/2}C_1 s/n
\end{equation}
as $z \to b_1, s/n\to 0$. Thus from \eqref{ineqg} it follows that there exists $\epsilon,C>0$ such that for all $|z-b_1|>\epsilon s/n$ we have
\begin{equation}
e^{-ng(z)}<e^{-sC}
\end{equation}
as $n,s \to \infty$, and for $u_0$ sufficiently small. The same may be shown at the point $b_2$, concluding the proof.
\end{proof}

Let
\begin{equation}\label{defS}
S(z)=\begin{cases}T(z) & \textrm{for $z$ outside the lenses,}\\
T(z)\begin{pmatrix} 1&0\\e^{-n \phi(z)}&1 \end{pmatrix} & \textrm{for $z$ inside the lense and outside the unit disc,}\\
T(z)\begin{pmatrix} 1&0\\-e^{-n \phi(z)}&1 \end{pmatrix} & \textrm{for $z$ inside the lense and inside the unit disc.}
\end{cases}
\end{equation}

Then $S$ satisfies the following RH problem:
\begin{itemize}
\item[(a)] $S :\mathbb C \setminus \Gamma_S \to \mathbb{C}^{2\times 2}$ is analytic, where $\Gamma_S= C\cup \Gamma_S^{in} \cup \Gamma_S^{out}$ as shown in Figure \ref{Contour GammaS2}.
\item[(b)] On $\Gamma_S \setminus \{a,\overline{a},b_1,b_2\}$, $S$ has the following jumps:
\begin{align*}
S_+&=S_-\begin{pmatrix}0&1\\-1&0\end{pmatrix}& & \textrm{for $z \in J_2$,}\\
S_+&=S_-\begin{pmatrix} 1&0\\e^{-n\phi(z)}&1 \end{pmatrix}  &&\textrm{for $z\in \Gamma_S^{in}\cup \Gamma_S^{out}$,}\\
S_+&=S_- e^{-in \Omega\sigma_3} && \textrm{for $z\in \Sigma_2$.}\\
S_+&=S_- \begin{pmatrix} e^{-in\Omega}&z^{-n}e^{n(g_++g_-)}\\0&e^{in\Omega} \end{pmatrix} && \textrm{for $z\in (\overline{a},1)=J_1\bigcap \Sigma_2^o$,}\\
S_+&=S_- \begin{pmatrix}1&z^{-n}e^{2ng}\\0&1\end{pmatrix} && \textrm{for $z\in (1,a)=J_1\bigcap \Sigma_1^o$.}\\
S_+&=S_-  && \textrm{for $z \in \Sigma_1$.}\end{align*}
\item[(c)] As $z\to \infty$,
\begin{equation}
S(z)=I+\mathcal{O} \left( z^{-1}\right). \end{equation}
\end{itemize}

\subsection{Main parametrix}

In the region $\mathbb C\setminus (U_0\cup U_1 \cup U_2)$,
we approximate the RH problem for $S$ by a main parametrix $M$, which satisfies the RH problem:
\begin{itemize}
\item[(a)] $M:\mathbb C \setminus \overline{ \{J_2 \bigcup \Sigma_2^o\}} \to \mathbb{C}^{2\times 2}$ is analytic.
\item[(b)] On $J_2$ and $\Sigma_2^o$, $M$ has the following jumps:
\begin{align*}
M_+(z)&=M_-(z)\begin{pmatrix}0&1\\-1&0\end{pmatrix}& & \textrm{for $z \in J_2$,}\\
M_+(z)&=M_-(z) e^{-in \Omega\sigma_3} && \textrm{for $z\in \Sigma_2^o$}.\end{align*}
\item[(c)] As $z\to \infty$,
\[ M(z)=I +\mathcal O\left(z^{-1}\right).\]
\end{itemize}
A solution to the RH problem for $M$ is given by
\begin{align}\label{DefM}
 M(z)&=\left(I+\frac{F}{z-1}\right)D^{-1}(\infty)\begin{pmatrix}
\gamma_1(z)&
-\gamma_2(z)\\
\gamma_2(z)&
\gamma_1(z)
\end{pmatrix}D(z),
\end{align}
where $F$ is a constant matrix and
\begin{equation} \label{def epsilon12}
\begin{aligned}
\gamma_1(z)&=\frac{1}{2}\left(\left(\frac{z-b_1}{z-b_2}\right)^{1/4}+\left(\frac{z-b_2}{z-b_1}\right)^{1/4}\right) \\
\gamma_2(z)&=\frac{1}{2i}\left(\left(\frac{z-b_1}{z-b_2}\right)^{1/4}-\left(\frac{z-b_2}{z-b_1}\right)^{1/4}\right)
\end{aligned}
\end{equation}
with branch cuts on $J_2$ and such that $\gamma_1(z)\to 1$ and $\gamma_2(z)\to 0$ as $z \to \infty$. For $y\in \mathbb R$, let $\left<y\right>$ be defined such that
\begin{equation}\label{<>}
\left<y\right>\in [-1/2,1/2), \quad y-\left<y\right> \in \mathbb Z.
\end{equation}
Then $D$ is given by
\begin{equation} 
D(z)=\exp \left(-\left<\frac{n\Omega}{2\pi}\right> h(z)\sigma_3 \right),
\end{equation}
where it follows from the jumps of $ h$ \eqref{jump h} that $D$ is analytic for $z\in \mathbb C \setminus (\overline{\Sigma_2^o\bigcup J_2})$, and
\begin{equation}\begin{aligned}
D_-^{-1}D_+&= \exp\left( - 2\pi i \left< \frac{n\Omega}{2\pi}\right>\sigma_3\right)&&\textrm{for } z\in \Sigma_2^o, \\
D_-D_+&=I &&\textrm{for } z\in J_2.
\end{aligned}\end{equation}
The function $M$ defined in \eqref{DefM} will solve the RH problem for $M$ with any constant matrix $F$, which we will define later in \eqref{def F}.
The reason for the prefactor $I+\frac{F}{z-1}$ in \eqref{DefM}, which does not affect the jump conditions for $M$, will become apparent later on.

\subsection{Model RH problem $\Phi$}
Consider the following RH problem for $\Phi(\zeta;k)$, where $k\in \mathbb N$:
\begin{itemize}
\item[(a)] $\Phi(\zeta):\mathbb C\setminus [\eta_1,\eta_2] \to \mathbb C^{2\times 2}$ is analytic for given $\eta_1<\eta_2$.
\item[(b)] $\Phi$ has $L^2$ boundary values for $\zeta\in(\eta_1,\eta_2)$ satisfying
\begin{align*}
\Phi_+(\zeta;k)=\Phi_-(\zeta;k)
\begin{pmatrix} 1&1\\0&1 \end{pmatrix}. \end{align*}
\item[(c)]
As $\zeta\to \infty$, 
\begin{equation}\label{Phiinf}
\Phi(\zeta;k)=\left(I+\frac{\Phi_1}{\zeta}+\frac{\Phi_2}{\zeta^2}+\mathcal O\left( \zeta^{-3}\right)\right) \zeta^{k\sigma_3}. \end{equation}
\end{itemize}
It is well-known by the standard theory and is easy to verify that the unique solution to this RH problem is given by
\begin{equation}\label{Phi Soln}
\begin{aligned}
\Phi(\zeta;0)&=\begin{pmatrix}
1&\frac{1}{2\pi i}\log \left(\frac{\zeta-\eta_2}{\zeta-\eta_1}\right)\\0&1
\end{pmatrix}\\
\Phi(\zeta;k)&=\begin{pmatrix}
 \frac{1}{\kappa_k}L_k(\zeta)	&	\frac{1}{2\pi i \kappa_k} \int_{\eta_1}^{\eta_2} \frac{L_k(x)}{x-\zeta}  dx \\
-2\pi i \kappa_{k-1}L_{k-1}(\zeta)	&-\kappa_{k-1} \int_{\eta_1}^{\eta_2} \frac{L_{k-1}(x)}{x-\zeta}dx
\end{pmatrix}&\textrm{for $k\geq 1$,}
\end{aligned}
\end{equation}
where $L_k$ are the Legendre polynomials of degree $k$ with positive leading coefficients, orthonormal on $(\eta_1,\eta_2)$:
\begin{equation}
\int_{\eta_1}^{\eta_2} L_k(\zeta)L_j(\zeta)d\zeta= \delta_{jk}=\begin{cases} 0& \textrm{for $j\neq k$,}\\ 1& \textrm{for $j=k$,} \end{cases}
\end{equation}
and we denote the first 3 leading coefficients as follows:
\begin{equation} L_k(\zeta)=\kappa_k \zeta^k+\mu_k \zeta^{k-1}+\nu_k \zeta^{k-2}+\dots . \label{coeff} \end{equation}
Writing the large $\zeta$ expansion of \eqref{Phi Soln} and using orthogonality in the second column, we obtain that 
\begin{equation} \label{def Phi1} \begin{aligned}
\Phi_1&=\begin{pmatrix} \frac{\mu_k}{\kappa_k}&-\frac{1}{2\pi i}\kappa_k^{-2}\\-2\pi i \kappa_{k-1}^2&-\frac{\mu_k}{\kappa_k} \end{pmatrix}&&\textrm{for $k\geq 1$,}\\
\Phi_2 &=\begin{pmatrix} \frac{\nu_k}{\kappa_k}&\frac{\mu_{k+1}}{2\pi i\kappa_{k+1}\kappa_k^2} \\ -2\pi i \kappa_{k-1}\mu_{k-1} & \frac{1}{\kappa_k\kappa_{k+1}}(\mu_k\mu_{k+1}-\nu_{k+1}\kappa_k) \end{pmatrix}&& \textrm{for $k\geq2$}.
 \end{aligned}\end{equation}
 When $k=0$ we have
 \begin{equation}
 \Phi_1= \begin{pmatrix}
 0&-\frac{\eta_2-\eta_1}{2\pi i}\\0&0
 \end{pmatrix}.
 \end{equation}
It is well known that $L_k$ has the explicit representation for $k \geq 0$:
\begin{equation}\label{series Lk} \begin{aligned}
L_k(\zeta)&=\sqrt{\frac{2k+1}{\eta_2-\eta_1}} \sum_{j=0}^k \binom{ k}{j} \binom{ k+j}{ j}  \left(\frac{\zeta-\eta_2}{\eta_2-\eta_1}\right)^j\\
&= \sqrt{\frac{2k+1}{\eta_2-\eta_1}} \sum_{j=0}^k \binom{ k}{j} \binom{ k+j}{ j}  \left(\frac{\zeta-\eta_1}{\eta_2-\eta_1}\right)^j(-1)^{k-j},\end{aligned}\end{equation}
where $\binom{0}{0}=1$.
As a consequence, the coefficients in \eqref{coeff} are given by
\begin{equation}\label{kappa mu nu}
\begin{aligned}
\kappa_k &=(\eta_2-\eta_1)^{-k-1/2}\sqrt{2k+1}\binom{2k}{ k}  ,\\
\mu_k&=-(\eta_2-\eta_1)^{-k-1/2}\sqrt{2k+1}\binom{2k}{ k} \frac{k}{2} (\eta_1+\eta_2) ,\\
\nu_k &=(\eta_2-\eta_1)^{-k-1/2}\sqrt{2k+1}\binom{2k-2}{ k-2}\frac{k}{2}\left(k(\eta_1+\eta_2)^2-\left(\eta_1^2+\eta_2^2\right)\right),
\end{aligned}\end{equation}
for $k\geq 0,1,2$ respectively.
From \eqref{Phi Soln}, \eqref{series Lk}, \eqref{kappa mu nu}, it follows that for $k\geq 1$, 
\begin{equation}\label{Phi local}\begin{aligned}
\Phi(\zeta)&=\begin{pmatrix}(-1)^k \frac{(\eta_2-\eta_1)^k}{ \binom{2k}{ k} } \left(1-\frac{k(k+1)}{\eta_2-\eta_1}(\zeta-\eta_1)+\mathcal O((\zeta-\eta_1)^2)\right)&*\\(-1)^k2\pi i (2k-1) (\eta_2-\eta_1)^{-k}\binom{2k-2}{k-1} \left(1-\frac{k(k-1)}{\eta_2-\eta_1}(\zeta-\eta_1)+\mathcal O((\zeta-\eta_1)^2)\right)&*
\end{pmatrix}, &&  \textrm{as } \zeta\to \eta_1\\
\Phi(\zeta)&= \begin{pmatrix} \frac{(\eta_2-\eta_1)^k}{ \binom{2k}{ k} } \left(1+\frac{k(k+1)}{\eta_2-\eta_1}(\zeta-\eta_2)+\mathcal O((\zeta-\eta_2)^2)\right)&*\\
-2\pi i (2k-1)(\eta_2-\eta_1)^{-k}\binom{2k-2}{k-1}\left(1+\frac{k(k-1)}{\eta_2-\eta_1}(\zeta-\eta_2)+\mathcal O((\zeta-\eta_2)^2)\right)&*
\end{pmatrix}, &&  \textrm{as } \zeta\to \eta_2
\end{aligned} \end{equation}

\subsection{Local parametrix at $1$}\label{Local parametrix at 1}
Recall that $U_0$ defined by \eqref{U0} is an open disc containing $J_1$ and that as $s,n\to \infty, u_0\to 0$ the radius of $U_0$ is of length $\epsilon s/(n\log u_0^{-1})$ for some $\epsilon>0$. On $U_0$ we defined a local variable $\zeta$ in \eqref{def zeta}. 
We define the local parametrix $P$ on $U_0$ by
\begin{equation} \label{defP}
P(z)=E(z)\Phi \left(\zeta(z);  k  \right)e^{-n\left(g(z)-\frac{\log(z)}{2}\right)\sigma_3},\quad
k=\frac{\Omega n}{2\pi}-x,\quad x=\left<\frac{\Omega n}{2\pi}\right>,
\end{equation}
where $\Phi$ is given by \eqref{Phi Soln}, and where $E$ is an analytic function on $U_0$ given by
\begin{equation}\label{defE}\begin{aligned}
E(z)&=\left(I+\frac{F}{z-1}\right)D(\infty)^{-1}
\begin{pmatrix}
\gamma_1(z)&
-\gamma_2(z)\\
\gamma_2(z)&
\gamma_1(z)
\end{pmatrix}
B(z)\left(I-\frac{X}{\zeta(z)}\right),
\\ B(z)&=e^{\frac{2\pi}{\Omega} x\left(g_1(z)-\frac{1}{2}\log z\right)\sigma_3},
\end{aligned}
\end{equation}
with constant matrices $F$ and $X$ defined below.
From \eqref{g1 Sigma} and \eqref{h Sigma} we see that $\zeta(J_1)\subset \mathbb R$. We let $\eta_1=\zeta(\overline a)$ and $\eta_2= \zeta(a)$. Then $\zeta(J_1)=(\eta_1,\eta_2)$, and so $P(z)$ has a jump on $J_1$ induced by that of $\Phi$ on $(\eta_1,\eta_2)$.
$F$ is a constant, nilpotent matrix
\begin{equation}\label{def F}
F=\widetilde h ^{-\sigma_3}\begin{pmatrix} f& \psi f\\ -f/\psi &- f \end{pmatrix}\widetilde h^{\sigma_3},
\end{equation}
where 
\begin{equation} \label{def f}\begin{aligned}\psi &=\begin{cases}-\frac{\gamma_1(1)}{\gamma_2(1)}
\\
\frac{\gamma_2(1)}{\gamma_1(1)}
\end{cases}&&\begin{matrix*}[l]\textrm{for } 0\leq x<1/2,\\  \textrm{for } -1/2\leq x<0, \end{matrix*}\\
\widetilde h&=D_{11}(\infty)=\exp(-xh(\infty))=\left(\frac{(1-b_1)^{1/2}+(1-b_2)^{1/2}}{b_2-b_1}\right)^{-x}\\
f&=\begin{cases}
-\frac{(1-b_1)^{1/2}(1-b_2)^{1/2} \rho}{1+\rho} \\
-\frac{(1-b_1)^{1/2}(1-b_2)^{1/2} \rho}{1-\rho}\end{cases}
&&\begin{matrix*}[l]\textrm{for } 0\leq x<1/2,\\  \textrm{for } -1/2\leq x<0, \end{matrix*}
\\
\rho&=\begin{cases}
\frac{1}{2\pi\kappa_k^2} \exp\left[\frac{2\pi}{\Omega}g_1(1)\left(-1+2x\right)\right]\\
-2\pi\kappa_{k-1}^2 \exp\left[\frac{2\pi}{\Omega}g_1(1)\left(-1-2x\right)\right]
\end{cases}
&&\begin{matrix*}[l]\textrm{for } 0\leq x<1/2,\\  \textrm{for } -1/2\leq x<0 ,\end{matrix*}
\end{aligned}\end{equation}
where $h(\infty)$ was defined in \eqref{hinfty}
and $X$ is a constant matrix given in terms of elements of \eqref{def Phi1}
\begin{equation}
X=\begin{cases}\Phi_{1,12}\begin{pmatrix} 0&1\\0&0 \end{pmatrix} &\textrm{for }0\leq x <1/2,\\
\Phi_{1,21}\begin{pmatrix} 0&0\\1&0 \end{pmatrix} &\textrm{for }-1/2\leq x<0.
\end{cases}
\end{equation}
The factor $I-X/\zeta(z)$ in $E(z)$ is needed to cancel the would be $u_0^{1-2|x|}\log u_0^{-1}$ non-smallness in the matrix elements of $\Delta_1^{(1)}$ originating from $\Phi_{1,12}B_{11}^2$ for $0\leq x<1/2$ and $\Phi_{1,21}B_{11}^{-2}$ for $-1/2\leq x<0$ (see Proposition \ref{Proposition error term} below) so that $P$ and $M$ match to the main order on the boundary 
$\partial U_0$ for all $x\in[-1/2,1/2)$. This factor, however, has a pole at $z=1$, but we need $E(z)$ to be analytic
in $U_0$. As is easy to verify, the analyticity of $E(z)$ (i.e. the absence of a pole at $1$) is achieved by choosing $F$ as
defined in (\ref{def F}).  

\begin{Prop} \label{Proposition error term}
As $u_0\to 0$ and $ s,n\to \infty$, we have the following matching condition uniformly on the boundary $\partial U_0$:
\[
P(z)M^{-1}(z)=I+\Delta_1^{(1)}(z)+\Delta_2^{(1)}(z)+ \Xi^{(1)}+\mathcal O(\widetilde e u_0^{-2|x|}( u_0\log u_0^{-1})^{3}+\widetilde e( s^2u_0^4(\log u_0^{-1})^2),
\]
where
\begin{equation}\label{tilde e}
\widetilde e=\left(1+ u_0^{1-2|x|} \log u_0^{-1}\right)^2,
\end{equation}
and $\Delta_1^{(1)}$, $\Delta_2^{(1)}$ and $\Xi^{(1)}$ are given by (\ref{PM-12}) below. We have, uniformly for $z$ on the boundary $\partial U$,
\begin{equation} \label{prop error term eqn}
\begin{aligned}
\Delta_1^{(1)}(z)&=\mathcal O(\widetilde e u_0^{1+2|x|}\log u_0^{-1}+\widetilde e su_0^2\log u_0^{-1}),\\
\Delta_2^{(1)}(z)&=\mathcal O\left(\widetilde e u_0^2(\log u_0^{-1})^2\left(\frac{s}{\log u_0^{-1}}+u_0^{1-2|x|} \log u_0^{-1}
\right)\right),\\
\Xi^{(1)}&=\mathcal O\left(\widetilde e su_0^{-2|x|}u_0^3(\log u_0^{-1})^2\right)
.
\end{aligned}
\end{equation}

\end{Prop}

\begin{proof}
First, assume that $k$ is bounded. Since 
\begin{equation}
\zeta^k(z)e^{-n\left(g(z)-\frac{\log z}{2}\right)}=e^{-\frac{2\pi x}{\Omega}\left(g(z)-\frac{\log z}{2}\right)}, \qquad x=\frac{n\Omega}{2\pi}-k
\end{equation}
 we have on the boundary $\partial U$ that (recall \eqref{def g birth}, \eqref{DefM}, \eqref{Phiinf}, \eqref{defE})
\begin{equation}\label{PM-1}
P(z)M^{-1}(z)=E(z)\left(I+\frac{\Phi_1}{\zeta}+ \frac{\Phi_2}{\zeta^2}+\mathcal O (\zeta^{-3})\right)\left(I-\frac{X}{\zeta(z)} \right)  E^{-1}(z).
\end{equation} 
 It follows from \eqref{expansion g1}, \eqref{expansion Omega}, \eqref{def f} that 
\begin{equation}\label{order f}\begin{aligned}
\rho&=\mathcal O \left( u_0^{1-2|x|}\right),\\
f,F&=\mathcal O \left( su_0^{1-2|x|}/n\right),
\end{aligned}\end{equation}
as $u_0, s/n\to 0$ and $s,n \to \infty$.
Denote 
\begin{equation} \label{def Etilde}
\widetilde E(z)=\left(I+\frac{F}{z-1}\right)D(\infty)^{-1}\begin{pmatrix}
\gamma_1(z)&-\gamma_2(z)\\ \gamma_2(z) &\gamma_1(z)
\end{pmatrix} .
\end{equation}
From \eqref{order f}, the boundedness of $\widetilde h, \widetilde h^{-1}, \psi, \psi^{-1}, \gamma_1(z), \gamma_2(z)$ for $z\in \partial U_0$, and the fact that the radius of $U_0$ equals $\epsilon \frac{s}{n\log u_0^{-1}}$ for some $\epsilon>0$, we have 
\begin{equation} \label{tilde E}
\widetilde E(z)= \mathcal O(\sqrt{\widetilde e})
\end{equation}
 as $u_0\to 0$ and $s,n\to \infty$, uniformly for $z \in \partial U_0$. From \eqref{PM-1} and \eqref{def Etilde} it follows that $\Delta_1^{(1)},\Delta_2^{(1)}$ take the form
\begin{equation}\label{PM-12}\begin{aligned}
\Delta_1^{(1)}(z)&=\begin{cases}
\zeta^{-1}(z) \widetilde E(z)B(z)\begin{pmatrix} \Phi_{1,11}&0\\ \Phi_{1,21}&\Phi_{1,22} \end{pmatrix} B^{-1}(z) \widetilde E^{-1}(z) &0\leq x<1/2,\\
\zeta^{-1}(z)\widetilde E(z)B(z)(z)\begin{pmatrix} \Phi_{1,11}&\Phi_{1,12}\\ 0&\Phi_{1,22} \end{pmatrix} B^{-1}(z)\widetilde E^{-1}(z) &-1/2\leq x<0 ,\end{cases}\\
\Delta_2^{(1)}(z)&=\zeta^{-2}(z)\widetilde E(z)\widehat{g_1}^{x\sigma_3}(z)(
\Phi_2-X\Phi_1)
\widehat{g_1}^{-x\sigma_3}(z) \widetilde E^{-1}(z),\\
\Xi^{(1)}(z)&=-\zeta^{-3}(z)\widetilde E(z)\widehat{g_1}^{x\sigma_3}(z)
X\Phi_2
\widehat{g_1}^{-x\sigma_3}(z) \widetilde E^{-1}(z).
\end{aligned} \end{equation}
From \eqref{expansion g1a}--\eqref{expansion g1}, and recalling \eqref{expansion Omega} and the fact that $w=\mathcal O\left(\frac{1}{\log u_0^{-1}}\right)$ as $u_0\to 0$, it follows that 
\begin{equation}\label{order g1hat} \widehat{g_1}(z)=\mathcal O\left(u_0^{-1}\right)
\end{equation}
as $u_0\to 0$, uniformly on the boundary $\partial U_0$. Similarly, substituting \eqref{zeta0zeta1} into \eqref{zeta exp} and recalling \eqref{expansion Omega}, we have
\begin{equation} \label{order zeta}
 \zeta(z)=\mathcal O \left(\frac{1}{u_0\log u_0^{-1}}\right)
\end{equation}
as $u_0\to 0$, uniformly on the boundary $\partial U_0$. From \eqref{zeta exp} and \eqref{zeta0zeta1} we have 
\begin{equation} \label{alpha+beta}\zeta(a)+\zeta(\overline a)=\mathcal O(u_0\log u_0^{-1}),\end{equation}
as $u_0 \to 0$. Using \eqref{alpha+beta} it follows from \eqref{def Phi1} and \eqref{kappa mu nu} that
\begin{equation} \label{order Phi} \begin{aligned}
\Phi_{1,11},\Phi_{1,22},\Phi_{2,12},\Phi_{2,21}&=\mathcal O\left(u_0\log u_0^{-1}\right)\\
\Phi_{1,12},\Phi_{1,21},\Phi_{2,11},\Phi_{2,22}&=\mathcal O\left(1\right)
\end{aligned}\end{equation}
as $u_0\to 0$ (for finite $k$). Combining \eqref{tilde E}--\eqref{order Phi} the proposition is proven for bounded $k$.

Now consider $k\to \infty$.  
From Stirling's formula we have 
\begin{equation} \label{Stirling}
2\pi \kappa_k^2 \to 1
\end{equation}
as $k\to \infty$. Thus \eqref{order f} holds uniformly for $k\in \mathbb N$. We study the particular double scaling limit where $k,\zeta \to \infty$, and from \eqref{expansion u0},  \eqref{alpha+beta} we have that $\eta_1+\eta_2 \to 0$ in such a manner that $k/\zeta ,k(\eta_1+\eta_2)=\mathcal O(u_0s)\to 0$.
Thus using \eqref{Stirling} 
%and noting that
%\begin{equation}
%\mu_k=-(\alpha+\beta)(k/2+\mathcal O(1)), \quad \nu_k=-\frac{k}{4}(\alpha^2+\beta^2+\mathcal O(k(\alpha+\beta)^2))
%\end{equation}
 we find that as $k\to \infty$ 
\begin{equation}
\Phi_1=\begin{pmatrix}
-\frac{k}{2} (\eta_1+\eta_2)& i+\mathcal O(k^{-1})\\-i+\mathcal O(k^{-1})&\frac{k}{2}(\eta_1+\eta_2)
\end{pmatrix}.
\end{equation} 
 We also find that as $k\to \infty$ and  $(\eta_1+\eta_2)\to 0$ such that $(\eta_1+\eta_2)k\to 0$  
\begin{equation}
\Phi_2=\begin{pmatrix}
-\frac{k}{8}(\eta_1^2+\eta_2^2)+\mathcal O(1)&\frac{ik}{2}(\eta_1+\eta_2)+\mathcal O(\eta_1+\eta_2)\\
\frac{ik}{2}(\eta_1+\eta_2)+\mathcal O(\eta_1+\eta_2)&\frac{k}{8}(\eta_1^2+\eta_2^2)+\mathcal O(1)
\end{pmatrix}.
\end{equation}
Thus we know the large $k,\zeta$ behaviour of $\zeta^{-1}\Phi_1, \, \zeta^{-2}\Phi_2$, and upon substituting into \eqref{PM-12} this yields \eqref{prop error term eqn}.
 It remains to calculate the error terms of order $\zeta^{-3}\Phi_3$ and higher, and in particular establish their behaviour as $k\to \infty$ with $\zeta$. 
 We rely here on the work by Kuijlaars, McLaughlin, Van Assche and Vanlessen in \cite{KMVV} where the authors found uniform error terms for the Legendre polynomials $L_k$ as $k\to \infty$.
In the remaining part of the proof of the proposition, we let $\widehat Y, \widehat R, \widehat N$ denote the functions $Y,R,N$ found in \cite{KMVV}. We compare $\widehat Y$ to $\Phi$ from \eqref{Phi Soln} in the present paper:
\begin{equation}
\Phi(\zeta)=2^{k\sigma_3}\widehat Y(y(\zeta)), \quad y(\zeta)=\frac{1}{2}\left(\zeta-\frac{\eta_1+\eta_2}{2}\right)
\end{equation}
where the parameter $n$ in \cite{KMVV} is set to be $k$ here. For $y$ bounded away from $[-1,1]$ it follows from equations (3.1), (4.2), (5.5), (7.1) in \cite{KMVV} that
\begin{equation} \label{YKMVV}
\begin{aligned}
\widehat Y(y)&=2^{-k\sigma_3}\widehat R(y)\widehat N(y)y^{k\sigma_3}\left(1+\left(1-y^{-2}\right)^{1/2}\right)^{k\sigma_3},\\
\widehat N(y)&=\begin{pmatrix}
\frac{1}{2}(a(y)+1/a(y))&\frac{1}{2i}(a(y)-1/a(y))\\
\\-\frac{1}{2i}(a(y)-1/a(y))
&\frac{1}{2}(a(y)+1/a(y))
\end{pmatrix},&&
a(y)=\left(\frac{y-1}{y+1}\right)^{1/4}.
\end{aligned}
\end{equation}
By the form of $\widehat N$ in \eqref{YKMVV} above and formula (8.11) in \cite{KMVV} it is clear that
\begin{equation} \nonumber
\widehat R(y(\zeta);k)\widehat N(y(\zeta))= I+\frac{\chi_1}{\zeta}+\frac{\chi_2}{\zeta^2}+\mathcal O(\zeta^{-3})
\end{equation}
as $\zeta \to \infty$, where $\chi_1$ and $\chi_2$ are bounded for $k\in \mathbb N$ and the $\mathcal O(\zeta^{-3})$ term is uniform for $k\in \mathbb N$.
As $\zeta,k \to \infty$ and $\eta_1+\eta_2 \to 0$ such that $k/\zeta\to 0$, $k(\eta_1+\eta_2)\to 0$, we have
\begin{equation} \label{YKMVV2}
\left(1+\left(1-y^{-2}\right)^{1/2}\right)^{\pm k}
=1\mp k\zeta^{-1}\left(\frac{\eta_1+\eta_2}{2}+\zeta^{-1}\right)+\mathcal O\left(k^2|\zeta|^{-2}(|\eta_1+\eta_2|+|\zeta|^{-1})^2\right). 
\end{equation}
 It follows from \eqref{YKMVV}-\eqref{YKMVV2} that as $\zeta,k \to \infty$ and $\eta_1+\eta_2 \to 0$ \\ such that $k/\zeta\to 0$, $k(\eta_1+\eta_2)\to 0$,
\begin{equation}\label{expansionPhi}
\begin{aligned}
\Phi(\zeta;k)&=
\left(I+\chi_1/\zeta+\chi_2/\zeta^2+\mathcal O\left(\zeta^{-3}\right)\right)\zeta^{k\sigma_3}
 \\& \times \Bigg(
  \begin{matrix}  1-k\zeta^{-1}\left(\frac{\eta_1+\eta_2}{2}+\zeta^{-1}\right)+\mathcal O\left(k^2|\zeta|^{-2}(|\eta_1+\eta_2|+|\zeta|^{-1})^2\right) 
  \\0\end{matrix}\\ &\quad \quad \quad
 \begin{matrix}0\\ 1+k\zeta^{-1}\left(\frac{\eta_1+\eta_2}{2}+\zeta^{-1}\right)+\mathcal O\left(k^2\zeta^{-2}(|\eta_1+\eta_2|+|\zeta|^{-1}))^2\right)
\end{matrix}\Bigg),
\end{aligned}\end{equation}
where, in particular, $\chi_1$ and $\chi_2$ are bounded for $k\in \mathbb N$ and the $\mathcal O(\zeta^{-3})$ term is uniform for $k\in \mathbb N$. By comparing \eqref{expansionPhi} with \eqref{Phiinf} it follows that as $\zeta,k \to \infty$ and $\eta_1+\eta_2 \to 0$ such that $k/\zeta\to 0$, $k(\eta_1+\eta_2)\to 0$,
\begin{multline}
\Phi(\zeta;k)\zeta^{-k\sigma_3}-I-\Phi_1/\zeta-\Phi_2/\zeta^2=\\ \mathcal O(\zeta^{-3})+\begin{pmatrix}\mathcal O(k^2(\zeta^{-4}+|\eta_1+\eta_2|\zeta^{-3}))&0\\0&\mathcal O(k^2(\zeta^{-4}+|\eta_1+\eta_2|\zeta^{-3}))\end{pmatrix}.
\end{multline}
\end{proof}

\subsection{Model RH problem $\Psi$}

The following RH problem has a solution in terms of Bessel functions.

\begin{itemize}
\item[(a)] $\Psi: \mathbb C \setminus \overline{\Gamma_{\Psi}} \to \mathbb C ^{2\times 2}$ is analytic, where $\Gamma_{\Psi} =\mathbb R^- \cup \Gamma^\pm_{\Psi}$, with $\Gamma^\pm_{\Psi}=\{ xe^{\pm \frac{ 2\pi}{3}i}: x\in \mathbb R ^+\}$, and with orientation taken in the direction of increasing real part.
\item[(b)] $\Psi$ has continuous boundary values $\Psi_+,\Psi_-$ on $\Gamma_{\Psi}$  satisfying the following jump conditions:
\begin{align}
\Psi_+(\zeta)&=\Psi_-(\zeta)\begin{pmatrix}0&1\\-1&0\end{pmatrix} \quad \textrm{for $\zeta \in \mathbb R^-$,}\\
\Psi_+(\zeta)&=\Psi_-(\zeta)\begin{pmatrix}1&0\\1&1\end{pmatrix}\quad \textrm{for $\zeta \in \Gamma^\pm_{\Psi}$.}
\end{align}
\item[(c)] As $\zeta \to \infty$, $\Psi$ has the following asymptotics:
\begin{equation} \label{Psiasym}
\Psi (\zeta)=\left(\pi \zeta^{\frac{1}{2}}\right)^{-\frac{\sigma_3}{2}} \frac{1}{\sqrt{2}} \begin{pmatrix} 1&i \\i &1\end{pmatrix} \left( I +\frac{1}{8\sqrt \zeta} \begin{pmatrix} -1&-2i \\-2i &1\end{pmatrix}
+\mathcal O\left( \zeta^{-1}\right)
\right)e^{\zeta^{\frac{1}{2}}\sigma_3}.  \end{equation}
\item[(d)] As $\zeta \to 0$, the behaviour of $\Psi$ is 
\begin{equation}
\Psi(\zeta)=\mathcal O(\log |\zeta|) .\end{equation}
\end{itemize}

This RH problem has a solution given in \cite{KMVV}, in terms of Bessel functions. For definitions and properties of Bessel functions see \cite{NIST}. 
We take the principal branches of the Bessel functions.
For $|\arg \zeta| <2\pi/3$, we have
\begin{equation}
\Psi(\zeta)=\begin{pmatrix} I_0(\zeta^{1/2})& \frac{i}{\pi}K_0(\zeta^{1/2})\\
\pi i \zeta^{1/2} I'_0(\zeta^{1/2})&-\zeta^{1/2}K_0'(\zeta^{1/2})\end{pmatrix}.
\end{equation}

\noindent For $2\pi/3<\arg \zeta<\pi$ we the solution is given by
\begin{equation}
\Psi(\zeta)=\frac{1}{2}\begin{pmatrix} 
 H_0^{(1)}(e^{\frac{\pi i}{2}}\zeta^{1/2})&
 H_0^{(2)}(e^{\frac{\pi i}{2}}\zeta^{1/2})
\\ \pi \zeta^{1/2}\left( H_0^{(1)}\right)'(e^{\frac{\pi i}{2}}\zeta^{1/2})&
\pi \zeta^{1/2}\left(H_0^{(2)}\right)' (e^{\frac{\pi i}{2}}\zeta^{1/2})
\end{pmatrix}.
\end{equation}

\noindent For $-\pi <\arg \zeta <-2\pi/3$ it is defined as
\begin{equation}
\Psi(\zeta)=\frac{1}{2}\begin{pmatrix}
H_0^{(2)}(e^{\frac{\pi i}{2}}\zeta^{1/2})&
-H_0^{(1)}(e^{\frac{\pi i}{2}}\zeta^{1/2})\\
-\pi\zeta^{1/2} \left(H_0^{(2)}\right)'(e^{\frac{\pi i}{2}}\zeta^{1/2})&
\pi \zeta^{1/2}\left( H_0^{(1)}\right)'(e^{\frac{\pi i}{2}}\zeta^{1/2})
\end{pmatrix} .
\end{equation}

We have the following useful asymptotics as $\zeta\to 0$ for $I_0$ and $K_0$:
\begin{align}
I_0(\zeta)&=1+\frac{\zeta^2}{4}+\frac{\zeta^4}{64} +\mathcal O(\zeta^6),\\
K_0(\zeta)&=-\log \frac{\zeta}{2} \left( 1+\frac{\zeta^2}{4}+\frac{\zeta^4}{64}+\mathcal O (\zeta^6)\right).
\end{align}

\subsection{Local parametrix at $b_1$ and $b_2$}

Let $U_1$ and $U_2$ be discs of radius $\frac{\epsilon s}{n}$ for some fixed but sufficiently small $\epsilon>0$, centered at $b_1$ and $b_2$ respectively. Recalling $w_j=\frac{n}{s}(z-b_j)$ for $j=1,2$, we have $|w_j|=\epsilon$ on $\partial U_j$.
For $z\in U_1$, define
\begin{equation}
\zeta_1(z)=\frac{n^2}{4}\phi(z)^2,
\end{equation}
where $\phi$ was defined in \eqref{def phi}.
 Recall the notation $B_j=\frac{n}{s}(1-b_j)$ for $j=1,2$. By \eqref{hexpb1} and \eqref{g1expb1} we have the following expansion of $\zeta_1(z)$ for $w_1$ in a neighbourhood of $0$:
\begin{equation}
\begin{aligned}
\zeta_1(z)&=s^2\zeta_{1,0}w_1\left(1+\mathcal O (w_1)\right)\\
\zeta_{1,0}&=\frac{B_2-B_1}{(b_1+(b_1b_2)^{1/2})^2} \left(1-\frac{n\Omega}{\pi s}\frac{b_1+(b_1b_2)^{1/2}}{B_2-B_1}\left(\frac{B_2}{B_1}\right)^{1/2}\right)^2 ,
 \end{aligned}\end{equation}
 and by considering \eqref{jump g} in addition, one verifies that $\zeta_1$ is analytic on $U_1$.

Recall from \eqref{value gb1}, \eqref{value hb1} that 
$ \phi_\pm(b_2)=\pm \Omega i$ and define
\begin{equation} \begin{aligned}
\widetilde \phi(z)&=\begin{cases}  \phi(z)-\Omega i& \textrm{for $z\in U_2$ and $z \in \mathcal D$,}\\
 \phi(z)+\Omega i& \textrm{for $z\in U_2$ and $z\notin \mathcal D$.}\end{cases}
\end{aligned} \end{equation}
where $\mathcal D$ denotes the unit disc. Then $\widetilde \phi:U \setminus J_2 \to \mathbb C$ is analytic, with a square root singularity at $b_2$.
We define the local variable
\[ \zeta_2(z)=\frac{n^2}{4} \widetilde \phi ^2(z) ,\]
which is analytic on $U_2$. Then, by \eqref{hw2} and \eqref{g1expb2}, $\zeta_2(z)$ has the following expansion at $w_2=0$:
\begin{equation}\begin{aligned}
\zeta_2(z)&=s^2\zeta_{2,0}w_2\left(1+\mathcal O (w_2)\right)\\
\zeta_{2,0}&=\frac{B_1-B_2}{(b_2+(b_1b_2)^{1/2})^2} \left(1-\frac{n\Omega}{\pi s}\frac{b_2+(b_1b_2)^{1/2}}{B_1-B_2}\left(\frac{B_1}{B_2}\right)^{1/2}\right)^2 ,
 \end{aligned}\end{equation}
 and by considering \eqref{jump g} in addition, one verifies that $\zeta_2$ is analytic on $U_2$.
 
The local parametrix is given by
\begin{align}P_{j}( z)&= E_j(z) \sigma_3^{j} \Psi(\zeta_j(z)) \sigma_3^{j}  e^{-\frac{n}{2}\phi(z)\sigma_3}\\
E_j(z)&=M(z)e^{\pm \frac{n}{2} \phi_+(b_j)\sigma_3}\frac{1}{\sqrt{2}}\begin{pmatrix} 1&(-1)^{j+1}\,i\\(-1)^{j+1}\, i &1 \end{pmatrix} \left(\pi \zeta_j(z) ^{1/2}\right)^{\frac{1}{2}\sigma_3},
\end{align}
on the $\pm$ side of the contour $C$, where $\phi_+(b_1)=0$ and $\phi_+(b_2)=\Omega i$.  As a consequence of the expansions of $\zeta_j$ above, we have
\begin{equation}
\zeta^{-1/4}_{j,-}\zeta^{1/4}_{j,+}=(-1)^j i,
\end{equation}
 and recalling the definition of $M$ in \eqref{DefM}, one may verify that $E_j$ is analytic on $U_j$. Recalling the jumps of $\phi$ in \eqref{phi+J}--\eqref{phi-J} and jumps of $g$ in \eqref{jump g}, one verifies that the jumps of $P_j$ match those of $S$ on $U_j$.

Since, recalling \eqref{def f}, $F=\mathcal O(s/n)$ as $s/n \to 0$ while $D(\infty)$ remains bounded, we have that $E_j$ is uniformly bounded on $\partial U_j$, and it follows that uniformly for $z\in \partial U_j$ we have the following matching condition of $M$ and $P_j$ 
\begin{equation}\label{matchingbj}
 \left(P_{j}M^{-1}\right)(z)=I+\Delta_{1}^{(b_j)}(z)+\mathcal O(1/s^2),
\quad \Delta_{1}^{(b_j)}(z)=\mathcal O(1/s),
 \end{equation}
as $s\to \infty$.  
A simple calculation yields
\begin{equation} \label{Delta11}
\Delta_{1}^{(b_1)}(z)=\frac{(B_2-B_1)^{1/2}}{16s\sqrt{\zeta_{1,0}}w_1} \left( I+\frac{F}{b_1-1}\right) D(\infty)^{-1}\begin{pmatrix}1&i\\i&-1\end{pmatrix} D(\infty)\left( I-\frac{F}{b_1-1}\right)+\mathcal O (1) \end{equation}
as $z \to b_1$, where the $\mathcal O(1)$ part is analytic on $U_1$. Similarly, as $z \to b_2$ we have:
\begin{equation} \label{Delta12}
\Delta_{1}^{(b_2)}(z)=\frac{(B_1-B_2)^{1/2}}{16s\sqrt{\zeta_{2,0}}w_2} \left( I+\frac{F}{b_2-1}\right) D(\infty)^{-1} \begin{pmatrix} 1&-i\\-i&-1\end{pmatrix} D(\infty)\left( I-\frac{F}{b_2-1}\right)+\mathcal O (1) ,\end{equation}
again with $\mathcal O(1)$ analytic.

\subsection{Small norm RH problem}
We define $R$ as follows:
\begin{equation}\label{def R}
R(z)=\begin{cases}
SM^{-1} & \textrm{for $z \in \mathbb C \setminus \overline{\left(\cup_{j=0}^2 U_j \right)}$,}\\
SP^{-1} & \textrm{for $z \in U_0$,}\\
SP^{-1}_j & \textrm{for $z \in U_j$ where $j=1,2.$}
\end{cases}
\end{equation}

Using standard small norm analysis, it follows from Proposition \ref{Proposition error term}, \eqref{Delta11}-\eqref{Delta12} and the fact that the contour lengths are $\partial U_j=\mathcal O(s/n)$ for $j=1,2$ as $s/n\to 0$ and $\partial U_0=\mathcal O(s/n(\log u_0^{-1}))$ as $s/n, u_0\to 0$ that given $\epsilon>0$,
\begin{equation}\label{Rn(0)}
R(z)=I+\mathcal O(1/n),
\end{equation}
uniformly for $|z-1|>\epsilon$.

If $z\in U_0$ then it follows from Proposition \ref{Proposition error term}, and \eqref{Delta11}-\eqref{Delta12} that
\begin{equation}
 R(z)=I+R_1(z)+R_2(z)+\mathcal O(||R_1||\,  ||R_2||+\widetilde e u_0^{-2|x|}( u_0\log u_0^{-1})^{3}+\widetilde e( s^2u_0^4(\log u_0^{-1})^2)) ,\end{equation}
where $||R_j||$ is the largest element of $R_j$ in absolute value for $j=1,2$, and
where the matrices $R_j$ are given by
\begin{equation}
\label{form R1}\begin{aligned}
R_1(z)&=\int_{\partial U_0} \frac{\Delta_1^{(1)}(u)}{(u-z)}\frac{du}{2\pi i}
+\sum_{j=1,2}\int_{\partial U_j} \frac{\Delta_1^{(b_j)}(u)}{(u-z)}\frac{du}{2\pi i},\\
R_2(z)&=
\int_{\partial U_0} \frac{R_{1-}(u)\Delta_1^{(1)}(u)+\Delta_2^{(1)}(u)+\Xi^{(1)}(u)}{u-z}\frac{du}{2\pi i}\\&+
\sum_{j=1,2}\int_{\partial U_j} \frac{R_{1-}(u)\Delta_1^{(b_j)}(u)+\Delta_2^{(b_j)}(u)}{u-z}\frac{du}{2\pi i},
\end{aligned}\end{equation}
with clockwise orientation taken in the integrals.
\section{Asymptotic analysis of the differential identity and correlation functions} \label{section asymptotic}
\subsection{Asymptotics of $\chi_n$}
From \eqref{Soln Y} we have
\begin{equation} \label{formula chin}
\chi_{n-1}^2=-(Y_{n})_{21}(0).
\end{equation}
By the transformations $Y=\widetilde g ^{-n\sigma_3}Te^{ng\sigma_3}$ and $ T= S = RM$ at $z=0$, (see \eqref{defT}, \eqref{defS}, \eqref{def R}) and recalling that $\chi_n$ is positive, we find from \eqref{formula chin} that
\begin{equation} \label{chigtilde}
\left| \chi_{n-1}^2\widetilde{g}^{-2n}\right|=\left|\widetilde{g}^{-n}e^{ng(0)}(R(0)M(0))_{21}\right|.
\end{equation}
From the definition of $g$ in \eqref{def g birth} and $\widetilde g$ in \eqref{def tilde g} it follows by computing $g_1(0)$, $h(0)$ in \eqref{g_1}, \eqref{def h} that
\begin{equation}
\left|\widetilde g^{-1}e^{g(0)}\right|=\left|\frac{1-\frac{b_1+b_2}{2}+(b_1b_2)^{1/4}\sqrt{|(1-b_1)(1-b_2)|}}{1-\frac{b_1^{-1}+b_2^{-1}}{2}+(b_1b_2)^{-1/4}\sqrt{|(1-b_1)(1-b_2)|}}\right|^{\Omega/2\pi}=1,
\end{equation}
so that
\begin{equation} \label{chigtilde2}
\left| \chi_{n-1}^2\widetilde{g}^{-2n}\right|=\left|(R(0)M(0))_{21}\right|.
\end{equation}
By \eqref{Rn(0)}
\begin{equation} \nonumber
R(0)=I+\mathcal O(1/n),
\end{equation}
as $n \to \infty$. 
Furthermore, we note that $F=\mathcal O (s/n)$ and that $\gamma_2(0)=-1+\mathcal O(s/n)$ as $s/n \to 0$, and substitute this into the definition of $M$ in \eqref{DefM} to find
\begin{equation} \label{RM(0)}
(R(0)M(0))_{21}=-(I+\mathcal O(s/n))e^{-\left<\frac{n\Omega}{2\pi}\right>(h(\infty)+h(0))}=-(I+\mathcal O(s/n))|e^{-\left<\frac{n\Omega}{2\pi}\right>h(\infty)}|^2,
\end{equation}
as $ s/n \to 0$. Substituting \eqref{RM(0)} into \eqref{chigtilde2} and recalling the notation \eqref{def f} it follows that
\begin{equation}
\left|\widetilde g^{-2n} \chi_{n-1}^2 \right|=(1+\mathcal O(s/n))\left| \widetilde h\right|^2,
\end{equation}
as $s/n\to 0$. We note that $\widetilde g= 1+\mathcal O(s/n(\log u_0^{-1}))$ as $s/n(\log u_0^{-1})\to 0$, and thus  we have
\begin{equation}\label{expansion chi g}
\left|\widetilde g^{-2n} \chi_n^2 \right|=(1+\mathcal O(s/n))\left| \widetilde h\right|^2,
\end{equation}
as $s/n\to 0$.

\subsection{Convergence of correlation functions} \label{CorrKernel}
Let $H_n(x,y)$ be the kernel built out of the orthogonal polynomials on $J$
\begin{equation}
H_n(x,y)=\frac{s}{\pi n}\sum_{j=0}^{n-1}\phi_j^{(n)}\left(\exp\left(\frac{2sxi}{n}\right)\right)\overline{\phi_j^{(n)}\left(\exp\left(\frac{2syi}{n}\right)\right)}.
\end{equation}
By the Christoffel-Darboux formula, $H_n$ also has the following useful form
\begin{equation} \label{CD kernel}
H_n(y_1,y_2)=\frac{s}{\pi n}\frac{z_1^nz_2^{-n}\phi_n^{(n)}(z_2)\overline{\phi_n^{(n)}(z_1)}-\overline{\phi_n^{(n)}(z_2)}\phi_n^{(n)}(z_1)}{1-z_2^{-1}z_1},
\end{equation}
where $z_j=\exp\left(\frac{2sy_ji}{n}\right)$ for $j=1,2$.
Let $\widehat K_n$ be defined similarly, but for the special case where $J=C$, namely:
\begin{equation} 
\widehat K_n(y_1,y_2)=\frac{s}{\pi n}\frac{z_1^{n/2}z_2^{-n/2}-z_2^{n/2}z_1^{-n/2}}{1-z_2^{-1}z_1}.
\end{equation}
Let $\rho^{(n)}_m$ be the $m$'th correlation function of the determinantal point process with correlation kernel $\widehat K_n$, and let $\rho^{(n,A)}_m$ be the $m$'th correlation function of the same process conditioned to have no points in $A=(\alpha ,-\nu)\cup(\nu,\beta)$.
Then 
\begin{equation}\begin{aligned}
\rho^{(n)}_m(x_1,\dots,x_m)&=\det(\widehat K_n(x_i,x_j))_{i,j=1}^m,\\
\rho^{(n,A)}_m(x_1,\dots,x_m)&=\det(H_n(x_i,x_j))_{i,j=1}^{m}.
\end{aligned}\end{equation}
 The two correlation functions are also related as follows:
\begin{equation} \label{rhona/rhon}
\rho_m^{(n,A)}(x_1,\dots,x_m)=\frac{\sum_{j=0}^\infty \frac{(-1)^j}{j!} \int_{A^j}\rho^{(n)}_{j+m}(x_1,\dots,x_{j+m})dx_{m+1}\dots dx_{j+m}}{D_n(J^{(n)})}.
\end{equation}
Similarly, we can write $\rho_m^A$ in terms of $\rho_m$ (both defined in Remark \ref{remark kernel}):
\begin{equation} \label{rhoA/rho}
\rho_m^A(x_1,\dots,x_m)=\frac{\sum_{j=0}^\infty \frac{(-1)^j}{j!} \int_{A^j}\rho_{j+m}(x_1,\dots,x_{j+m})dx_{m+1}\dots dx_{j+m}}{\det(I-K_s)_A}.
\end{equation}
The infinite sums \eqref{rhona/rhon} and \eqref{rhoA/rho} can be seen to converge for fixed $s$ by Hadamard's inequality.
Since 
\begin{equation} \left|\widehat K_n(x,y)-K_s(x,y)\right|=\mathcal O(1/n),
\end{equation}
as $n\to \infty$, it follows by formulas \eqref{rhoA/rho} and \eqref{rhona/rhon} that 
\begin{equation} \label{limrhonA}
|\rho_m^{(n,A)}(x_1,\dots,x_m)- \rho_m^{A}(x_1,\dots,x_m)|\to 0,
\end{equation}
as $n\to \infty$ for fixed $s$ (similarly to convergence of the determinants \eqref{Fredholm Toeplitz}, see the Appendix). 

By the definition of $Y$ in \eqref{Soln Y} and the formula for $H_n$ in \eqref{CD kernel} we have
\begin{equation}
H_n(y_1,y_2)=\frac{s\chi_n^2}{\pi n(1-z_1z_2^{-1})}(z_1^nz_2^{-n}Y_{11}(z_2)\overline{Y_{11}(z_1)}-\overline{Y_{11}(z_2)}Y_{11}(z_1)),
\end{equation}
where $z_j=\exp\left(\frac{2sy_ji}{n}\right)$ for $j=1,2$.
For the asymptotics of the correlation kernel we are less ambitious and choose not to proceed with all the detail in last section, and work with $|x|$ bounded away from $1/2$ as $n\to \infty$. Since the intention of $F$ and $X$ was to obtain uniform asymptotics up to the points $|x|=1/2$, we can let $F,X=0$ in $M$ and $P$ when we consider $|x|$ bounded away from $1/2$.
Then, in place of Proposition \ref{Proposition error term}, we have $PM^{-1}=I+o(1)$ as $u_0\to 0$ and $s,n \to \infty$ such that $s/n\to 0$ and $k\in \mathbb N, |x|<1/2$ remain fixed, uniformly on the boundary $\partial U_0$. Thus $R=I+o(1)$ in the same limit, and
 tracing back the transformations $Y\to T = S = RP$ we have, by \eqref{defT}, \eqref{defS}, \eqref{def R}, that for $z\in J_1$:
\begin{equation}\begin{aligned}
Y_{11}(z)&=\widetilde g^{-n}(RP)_{11}(z) e^{ng(z)}
\\&=\widetilde g^{-n} z^{n/2} \widetilde h^{-1}\left(\gamma_1B_{11}(1)\Phi_{11}(\zeta(z))-\gamma_2 B_{22}(1) \Phi_{21}(\zeta(z))
\right)(1+o(1)),
\end{aligned}\end{equation}
where $\widetilde g$ is given in \eqref{def tilde g}, and $\widetilde h$ in \eqref{def f}. 
 Thus it follows by \eqref{Phi Soln}, \eqref{expansion chi g}, the fact that $\widehat{g_1}$ is real to the main order and that
 \begin{equation}
 \gamma_1\overline{\gamma_2}=-1/2+o(1)
 \end{equation}
 as $s/n\to 0$, 
  that as $u_0\to 0$ and $s,n\to \infty$ such that $s/n\to 0$, while $k\in \mathbb N$ and $|x|<1/2$ remain fixed we have
\begin{equation}\label{Legendre kernel}
H_n(y_1,y_2)=\frac{\kappa_{k-1}}{\kappa_k}\frac{
L_k\left(2\zeta_0 y_1\right)L_{k-1}\left(2\zeta_0 y_2\right)-L_k\left(2\zeta_0 y_2\right)L_{k-1}\left(2\zeta_0 y_1\right)}{y_1-y_2}(1+o(1)).
\end{equation}
Since $\eta_1=-2+\mathcal O(u_0)$ and $\eta_2=2+\mathcal O(u_0)$ as $u_0\to 0$, it follows by continuity of the polynomials that $L_k^{(\eta_1,\eta_2)}$ can be replaced by $L_k^{(-2,2)}$ in \eqref{Legendre kernel} without modifying the error term. Similarly, by \eqref{expansion u0}, $2|\zeta_0|$ can be replaced by $4u_0^{-1}$ without modifying the error terms. Thus, combining \eqref{Legendre kernel} and \eqref{limrhonA}, we prove the statement in Remark \ref{remark kernel}.

\subsection{Expansion of Differential Identity}
In this section we start by writing the differential identity in a more convenient form, and find an expansion for it as $s, n \to \infty$ and $u_0\to0$ such that $su_0\to 0$ and $s/n\to0$, before proceeding to integrate it in Section \ref{integration of diffid}. Throughout the rest of the paper, the implicit constants in $\mathcal O(\dots)$ are independent of $s,u_0,n$. For example, if we write 
$\mathcal O(\frac{u_0}{s}+u_0^2),$ then in particular it is uniform in $n$, and if we write $\mathcal O(1)$, it means this expression is bounded in the double scaling limit described above.

Write the parametrix $P$ in \eqref{defP} in $U_0$ by grouping the factors as follows
\begin{equation}
P(z)=A(z)B(z)C(z)e^{-n\left(g(z)-\frac{1}{2}\log z\right)\sigma_3},
\end{equation}
where $A(z)$ and $C(z)$ are by
\begin{equation}
\begin{aligned}
A(z)&=\left(I+\frac{F}{z-1}\right) \widetilde h^{-\sigma_3} \begin{pmatrix}
\gamma_1(z)&-\gamma_2(z)\\\gamma_2(z)&\gamma_1(z)
\end{pmatrix}\\
C(z)&=\left(I-\frac{X}{\zeta(z)}\right)\Phi(\zeta(z))
\end{aligned}
\end{equation}
 By the transformations $Y=\widetilde g ^{-n\sigma_3}Te^{ng\sigma_3}$ and $ T= S = RM$ at $z=1$, (see \eqref{defT}, \eqref{defS}, \eqref{def R}) we have for $z\in U_0$
\begin{equation}\label{form Y}
Y_{11}(z)= \widetilde g^{-n}z^{n/2}\left[RABC\right]_{11},
\end{equation}
where 
\begin{equation} \label{funA} \begin{aligned}
A(z)&=A_1(z)+\frac{A_2(z)}{z-1},\\
A_1(z)&=\widetilde h^{-\sigma_3}\begin{pmatrix}\gamma_1(z) &-\gamma_2(z) \\ 
\gamma_2(z)  &\gamma_1(z)  \end{pmatrix},\\
A_2(z)&=f\widetilde h^{-\sigma_3}\begin{pmatrix} \gamma_1(z) +\gamma_2(z)  \psi &
-\gamma_2(z) + \gamma_1(z) \psi  \\ 
-\gamma_2(z)  -\gamma_1(z)  /\psi &
-\gamma_1(z)  +\gamma_2(z) /\psi \end{pmatrix},
\end{aligned}\end{equation}
and
\begin{equation}\label{funC}
\begin{aligned}
C(z)&=\begin{cases}\begin{pmatrix}\Phi_{11}(\zeta) -\frac{\Phi_{1,12}}{\zeta}\Phi_{21}(\zeta) &*\\ \Phi_{21}(\zeta)&*\end{pmatrix} & \textrm{for } 0\leq x<1/2,\\
\begin{pmatrix}\Phi_{11}(\zeta)&*\\ \Phi_{21}(\zeta)-\frac{\Phi_{1,21}}{\zeta}\Phi_{11}(\zeta)&*
\end{pmatrix}&\textrm{for } -1/2\leq x<0.
\end{cases}
\end{aligned}
\end{equation}
The expression for $C$ in \eqref{funC} is valid for $k\geq 1$, while for $k=0$, we have
\begin{equation}
C(z)=\begin{pmatrix}
1&*\\0&*
\end{pmatrix}.
\end{equation}
It follows that  
\begin{equation} \label{form Y'}
Y_{11}'(z)=\widetilde g^{-n}z^{n/2}\Bigg[ \frac{n}{2z} RA BC+R'A BC+RA'BC+RA B'C+RA BC'\Bigg]_{11},
 \end{equation}
where we suppress dependency on the variable $z$ on the right hand side, and where $'$ denotes differentiation with respect to $z$.
Substituting \eqref{expansion chi g}, \eqref{form Y}, \eqref{form Y'} into \eqref{diff id} we find that
\begin{equation}\label{formula F(z)}
F(z)=-2|\widetilde h|^2\Re \left[z \overline{ (RABC)_{11}}
\left( R'A BC+RA'BC+RA B'C+RA BC'\right)_{11}\right](1+\mathcal O(s/n)),
\end{equation}
for $z\in U_0$.
We would now like to evaluate $F(a)$ and $F(\overline a)$.
Since $\zeta\left(e^{i\theta}\right)$ is real for real $\theta$ on $U_0$, it follows that $\frac{d}{d\theta} \zeta\left(e^{i\theta}\right)$ is real. Consider the entries of $C$ and recall that $\Phi_{11}(x)$ is real for $x\in \mathbb R$, and that $\Phi_{21}(x)$ and $ \Phi_{1,12}$ are purely imaginary. By \eqref{funC}, $C_{11}(e^{i\theta})$ is real and so $z\frac{d}{dz} C_{11}(e^{i\theta})$ is purely imaginary, while $C_{21}(e^{i\theta})$ is purely imaginary and $z\frac{d}{dz} C_{21}(e^{i\theta})$ is real. From \eqref{g1 Sigma}, we recall that $B_{jj}\left(e^{i\theta}\right)$ is real for $j=1,2$. Thus $\overline{B_{11}}B_{22}=\overline{B_{22}}B_{11}=1$. From these observations, we find that
\begin{equation}
\label{form RABC'} \begin{aligned}
\Re \left[z\overline{\left( RABC\right)_{11}}
\left(RA BC'\right)_{11}\right]&= z(C_{11}C_{21}'-C_{11}'C_{21})\Re \left[ \overline{(RA )_{11}}\left(RA\right)_{12}\right],\\
\Re \left[z\overline{( RABC )_{11}}
\left(RA B'C\right)_{11}\right]&=z(B_{11}B_{22}'-B_{11}'B_{22})C_{11}C_{21}\Re \left[ \overline{(RA)_{11}}\left(RA\right)_{12}\right].
\end{aligned}
\end{equation}
When $k=0$, both expressions in \eqref{form RABC'} are equal to $0$.

\subsubsection{Evaluation of \eqref{form RABC'}}
Using the expansion for $\zeta$ from \eqref{zeta exp}--\eqref{zeta0zeta1}, and the fact that $\zeta(a)-\zeta(\overline a)=4$ from \eqref{def Omega}, we find that
\begin{equation}
\begin{aligned}
\zeta(e^{\pm i\theta_0})&=\pm 2\left(1\pm \zeta_1iu_0+\mathcal O\left(u_0\frac{s}{n}+u_0^2(\log u_0^{-1})^2\right)\right),\\
\frac{d}{dz}\zeta(z)\Bigg|_{z=e^{\pm i\theta_0}}&=-\frac{2in}{szu_0}\left(1\pm 2\zeta_1iu_0+\mathcal O\left(u_0\frac{s}{n}+u_0^2(\log u_0^{-1})^2\right)\right).
\end{aligned}
\end{equation}

We substitute the expansion of $\zeta$ from \eqref{zeta exp} into \eqref{funC}, and recall \eqref{def Phi1}-\eqref{Phi local}, to find
\begin{equation} \label{exp C11}\begin{aligned}
C_{11}(e^{\pm i \theta_0})&=\begin{cases}
\Phi_{11}(\zeta(e^{\pm i \theta_0}))\frac{k+1}{2k+1}(1 \pm iu_0 \zeta_1+\mathcal O(r_1)) &\textrm{for $0\leq x<1/2$,}\\
\Phi_{11}(\zeta(e^{\pm i \theta_0}))&\textrm{for $-1/2\leq x<0$,}
\end{cases}\\
C_{21}(e^{\pm i \theta_0})&=\begin{cases}
\Phi_{21}(\zeta(e^{\pm i \theta_0}))&\textrm{for $0\leq x<1/2$,}\\
\Phi_{21}(\zeta(e^{\pm i \theta_0}))\frac{k-1}{2k-1}(1\pm iu_0 \zeta_1+\mathcal O(r_1)) &\textrm{for $-1/2\leq x<0$,}
\end{cases}\\
r_1&=u_0(\log u_0^{-1})^2/s+(u_0\log u_0^{-1})^2+su_0/n.
\end{aligned}
\end{equation}
Using the expression
\begin{equation}\label{Phi'}
\begin{aligned}
\Phi'_{\zeta}(\eta_j)=\begin{pmatrix}
(-1)^j \frac{k(k+1)}{4}\Phi_{11}(\eta_j)&*\\
(-1)^j \frac{k(k-1)}{4}\Phi_{21}(\eta_j)&*
\end{pmatrix}\quad j=1,2.
\end{aligned}
\end{equation}
which follows from \eqref{Phi local}, we compute the following:
\begin{equation}\label{form CC'}
\begin{aligned}
&z\left(C_{11}(z)\frac{d}{dz}C_{21}(z)-\frac{d}{dz}C_{11}(z)C_{21}(z)\right)\Bigg|_{z=e^{\pi i \theta_0}}=\\
&=\begin{cases} \frac{2\pi k^2n}{su_0}\left(\frac{k+1}{2k+1}\pm 2i\zeta_1u_0\right) (1+\mathcal O ( r_1))& \textrm{for $0\leq x<1/2$,}\\
 \frac{2\pi k^2n}{su_0}\left(\frac{k-1}{2k-1}\pm 2i\zeta_1u_0\right) (1+\mathcal O ( r_1))& \textrm{for $-1/2\leq x<0$.}
\end{cases}
\end{aligned}\end{equation}
 We also have
\begin{equation} \label{form BB'}
\left(B_{11}(z)B_{22}'(z)-B_{11}'B_{22}(z)\right)\Big|_{z=e^{\pm i\theta_0}}=\mathcal O(n\log u_0^{-1}/s),
\end{equation}
 where the derivative is taken with respect to $z$.  
We now evaluate $RA$. 
Let $K$ denote the constant 
\begin{equation} \label{def K}
K=\frac{n}{s}\frac{u_1^{-1}-u_2^{-1}}{4}.
\end{equation}
We have the derivatives of $\gamma_1(e^{i\theta})$ and $\gamma_2(e^{i\theta})$ with respect to $\theta$ evaluated at $\theta=0$:
\begin{equation}
\begin{aligned} \label{epsilon'}
\frac{d}{d\theta}\gamma_1(1)&=-iK\gamma_2(1)(1+\mathcal O(s/n))\\
\frac{d}{d\theta}\gamma_2(1)&=iK\gamma_1(1)(1+\mathcal O(s/n))\\
\frac{d^2}{d\theta^2}\gamma_1(1)&=\frac{n^2}{s^22^4}\left(\gamma_1(1)(u_1^{-2}+u_2^{-2}-2u_1^{-1}u_2^{-1})-4i\gamma_2(1)(u_1^{-2}-u_2^{-2})\right)(1+\mathcal O(s/n))\\
\frac{d^2}{d\theta^2}\gamma_2(1)&=\frac{n^2}{s^22^4}\left(\gamma_2(1)(u_1^{-2}+u_2^{-2}-2u_1^{-1}u_2^{-1})+4i\gamma_1(1)(u_1^{-2}-u_2^{-2})\right)(1+\mathcal O(s/n)),\\
\frac{d^3}{d\theta^3}\gamma_1(1)&=\frac{n^3}{s^3}\Big(\frac{\gamma_1(1)}{8}(u_1^{-1}-u_2^{-1})(u_1^{-2}-u_2^{-2})
\\&\quad -\frac{3i}{2^6}\gamma_2(1)(11(u_1^{-3}-u_2^{-3})+u_1^{-1}u_2^{-2}-u_2^{-1}u_1^{-2})\Big)
(1+\mathcal O(s/n)),\\
\frac{d^3}{d\theta^3}\gamma_2(1)&=\frac{n^3}{s^3}\Big(\frac{\gamma_2(1)}{8}(u_1^{-1}-u_2^{-1})(u_1^{-2}-u_2^{-2})
\\&\quad +\frac{3i}{2^6}\gamma_1(1)(11(u_1^{-3}-u_2^{-3})+u_1^{-1}u_2^{-2}-u_2^{-1}u_1^{-2})\Big)
(1+\mathcal O(s/n)).
\end{aligned}
\end{equation}
Let $x_1$ and $x_2$ denote the following functions:
\begin{equation}\label{x1}
\begin{aligned}
x_1(z)&=\widetilde h ^{-1}R_{11}(z)\gamma_1(1)+\widetilde h R_{12}(z)\gamma_2(1),\\
x_2(z)&=-\widetilde h ^{-1}R_{11}(z)\gamma_2(1)+\widetilde h  R_{12}(z) \gamma_1(1).
\end{aligned}
\end{equation}
Then, using \eqref{funA}, \eqref{epsilon'} and \eqref{x1},  expand $A$. When $0\leq x<1/2$
\begin{equation}\label{formula A1}
\begin{aligned}
&(RA)_{11}(e^{i\theta})=\Bigg[ x_1(z)\left(1-Kf\left(\frac{\gamma_1}{\gamma_2}+\frac{\gamma_2}{\gamma_1}\right)\left( 1+\frac{n\theta}{2s}  (u_1^{-1}+u_2^{-1})\right)\right)+iKx_2(z)\theta
\\
&+\frac{n^2\theta^2}{s^22^5}\Bigg(x_1(z)(u_1^{-1}-u_2^{-1})^2-\frac{nf}{4s}x_1(z)\left(\frac{\gamma_1}{\gamma_2}+\frac{\gamma_2}{\gamma_1}\right)(11(u_1^{-3}-u_2^{-3})+u_1^{-1}u_2^{-2}-u_1^{-2}u_2^{-1})\\& +4ix_2(z)(u_1^{-2}-u_2^{-2})
\Bigg)\Bigg](1+\mathcal O (s/n))+\mathcal O \left(|\theta|^3n^3/s^3\right),\\
&(RA)_{12}(e^{i\theta})=\Bigg[\frac{ifx_1(z)}{\theta}\left(\frac{\gamma_1}{\gamma_2}+\frac{\gamma_2}{\gamma_1}\right)+x_2(z)+ix_1(z)K\theta\left(-1+K\frac{f}{2}\left(\frac{\gamma_1}{\gamma_2}+\frac{\gamma_2}{\gamma_1}\right)\right)\\
&+\frac{n^2\theta^2}{s^22^5}\Bigg(x_2(z)(u_1^{-1}-u_2^{-1})^2 -4ix_1(z)(u_1^{-2}-u_2^{-2})
\\&-\frac{2fnx_1(z)}{3s}\left(\frac{\gamma_1}{\gamma_2}+\frac{\gamma_2}{\gamma_1}\right)(u_1^{-3}+u_2^{-3}-u_1^{-1}u_2^{-2}-u_1^{-2}u_2^{-1})
\Bigg)\Bigg]
(1+\mathcal O (s/n))+\mathcal O(|\theta|^3n^3/s^3),
\end{aligned}
\end{equation} 
where we denote $\gamma_j=\gamma_j(1)$ for $j=1,2$.
When $-1/2\leq x<0$
\begin{equation}\label{formula A2}
\begin{aligned}
&(RA)_{11}(e^{i\theta})=\Bigg[\frac{ifx_2(z)}{\theta}\left(\frac{\gamma_1}{\gamma_2}+\frac{\gamma_2}{\gamma_1}\right)+x_1(z)+ix_2(z)K\theta\left(1+K\frac{f}{2}\left(\frac{\gamma_1}{\gamma_2}+\frac{\gamma_2}{\gamma_1}\right)\right)\\
&+\frac{n^2\theta^2}{s^22^5}\Bigg(x_1(z)(u_1^{-1}-u_2^{-1})^2-\frac{2fnx_2(z)}{3s}\left(\frac{\gamma_1}{\gamma_2}+\frac{\gamma_2}{\gamma_1}\right)
(u_1^{-3}+u_2^{-3}+u_1^{-1}u_2^{-2}+u_1^{-2}u_2^{-1})
\\& +4ix_2(z)(u_1^{-2}-u_2^{-2})
\Bigg)\Bigg](1+\mathcal O (s/n))+\mathcal O(|\theta|^3n^3/s^3),\\
&(RA)_{12}(e^{i\theta})=
\Bigg[x_2(z)\left(1+Kf\left(\frac{\gamma_1}{\gamma_2}+\frac{\gamma_2}{\gamma_1}\right)\left( 1+\frac{n\theta}{2s}  (u_1^{-1}+u_2^{-1})\right)\right)-iKx_1(z)\theta
\\&+\frac{n^2\theta^2}{2^5s^2}\Bigg(x_2(z)(u_1^{-1}-u_2^{-1})^2 
+\frac{nf}{4s}x_2(z)\left(\frac{\gamma_1}{\gamma_2}+\frac{\gamma_2}{\gamma_1}\right)(11(u_1^{-3}-u_2^{-3})+u_1^{-1}u_2^{-2}-u_1^{-2}u_2^{-1})\\&-4ix_1(z)(u_1^{-2}-u_2^{-2})\Bigg)\Bigg]
(1+\mathcal O (s/n))+\mathcal O(|\theta|^3n^3/s^3),
\end{aligned}
\end{equation} 
 We note that
\begin{equation} \label{epsilon1/2}
\begin{aligned}
|\gamma_1|^2,\, |\gamma_2|^2&= \frac{1}{4}\left(
\left|\frac{u_1}{u_2}\right|^{1/2}+\left|\frac{u_2}{u_1}\right|^{1/2}\right)(1+\mathcal O(s/n))
 \\
\overline{\gamma_1}/\overline{\gamma_2}, \,\gamma_2/\gamma_1&=\frac{-2\sqrt{|u_1u_2|}-i(u_1+u_2)}{u_1-u_2}(1+\mathcal O(s/n)).
\end{aligned}
\end{equation}
 Recalling \eqref{def f}, \eqref{def epsilon12}, \eqref{def K}, it is readily checked that
\begin{equation} \label{signfgamma}
f\left( \frac{\gamma_1}{\gamma_2}+\frac{\gamma_2}{\gamma_1}\right)\in \mathbb R,
\end{equation}
and that as $n,s\to \infty, s/n\to 0$,
\begin{equation}\label{fepsilonK} 
f\left( \frac{\gamma_1}{\gamma_2}+\frac{\gamma_2}{\gamma_1}\right)K=
\begin{cases}
\frac{\rho}{1 + \rho}(1+\mathcal O(s/n)) &\textrm{for $0\leq x<1/2$,}\\
\frac{\rho}{1 - \rho}(1+\mathcal O(s/n)) &\textrm{for $-1/2\leq x<0$.}
\end{cases}
\end{equation}
From \eqref{formula A1}, \eqref{formula A2}, \eqref{signfgamma}--\eqref{fepsilonK} it follows that
\begin{equation}\label{RA11 RA12}
\begin{aligned}
\Re \left[ (\overline{RA)_{11}(z)}(RA)_{12}(z)\right]&=\Re \left[\overline{x_1(z)}x_2(z)\right]\left(1+\mathcal O(s/n+u_0^2)\right).
\end{aligned}\end{equation}
Thus, from \eqref{form RABC'} and \eqref{form CC'},
\begin{equation}\label{formula RABC'}
\begin{aligned}
&F_0(e^{\pm i\theta_0})=\Re \left[z\overline{\left( RABC\right)_{11}}
\left(RA BC'\right)_{11}\right]=\Re \left[\overline{x_1(e^{\pm i\theta_0})}x_2(e^{\pm i\theta_0})\right]\times
\\&\times \begin{cases} \frac{2\pi k^2n}{su_0}\left(\frac{k+1}{2k+1}\pm i\zeta_1u_0\right) (1+\mathcal O (r_1))& \textrm{for $0\leq x<1/2$,}\\
 \frac{2\pi k^2n}{su_0}\left(\frac{k-1}{2k-1}\pm i\zeta_1u_0\right) (1+\mathcal O (r_1))& \textrm{for $-1/2\leq x<0$,}
\end{cases}
\end{aligned}
\end{equation}
 where $r_1$ was defined in \eqref{exp C11}. 
\begin{Prop} \label{Prop x1x2}
We have
\begin{equation} \label{barx1x2}
\Re \left[ \overline{x_1}x_2 \right]=\frac{|\widetilde h |^{-2}}{2}+\mathcal O\left(s/n+\left(\widetilde e u_0^2s\log u_0^{-1}+\widetilde e u_0^{1+2|x|}\log u_0^{-1}+s^{-1}\right)^2\right),
\end{equation}
and $|\widetilde h|^{-1}=\mathcal O(1)$, where $\widetilde h$ was given in \eqref{def f} and $\widetilde e$ was defined in \eqref{tilde e}.
\end{Prop}
The main term of Proposition \ref{Prop x1x2} is easy to calculate from \eqref{x1} and \eqref{def epsilon12}, but we defer the rest of the proof to Section \ref{consideration of error terms}.

From \eqref{Phi local}, \eqref{exp C11} we obtain that
\begin{equation}\label{orderC}
C_{11}(e^{\pm i\theta_0}),C_{21}(e^{\pm i\theta_0})=\bigO(\sqrt{k}).
\end{equation}
Recall that $k=\bigO(s/\log u_0^{-1})$.
Combining  \eqref{form RABC'}, \eqref{form BB'}, \eqref{RA11 RA12}, \eqref{orderC}, and using Proposition \ref{Prop x1x2} gives us
\begin{equation} \label{formula RAB'C}
\Re \left[z\overline{(RABC)_{11}}(RAB'C)_{11}
\right]=\mathcal O(n).
\end{equation}
\subsubsection{Evaluation of \eqref{formula F(z)}}
Suppressing $z$ dependence, we write 
\begin{equation} \label{form RA'BC}
\begin{aligned}
\Re \left[ z \overline{(RABC)_{11}}(RA'BC)_{11}(z)\right]&=F_1+F_2+F_3+F_4,\\
F_1&=B_{11}^2C_{11}^2 \Re \left[ z\overline{(RA)_{11}}(RA')_{11}\right],\\
F_2&=\Re \left[ zC_{11}\overline{(RA)_{11}}C_{21}(RA')_{12}\right],\\
F_3&=-\Re\left[ zC_{11}(RA')_{11}C_{21}\overline{(RA)_{12}}\right],\\
F_4&=- B_{22}^2C_{21}^2\Re \left[z\overline{(RA)_{12}}(RA')_{12}\right].
\end{aligned}
\end{equation}
Recall that $K=\mathcal O(n/s)$,  and that $\theta_0=u_0\frac{s}{n}$. From \eqref{formula A1} and \eqref{formula A2} we obtain that for $k\geq 1$
\begin{equation} \label{formula F1-F4}
\begin{aligned}
\frac{F_1(e^{\pm i\theta_0})}{B_{11}^2C_{11}^2}&=\begin{cases} \Re \left[ \overline{x_1} x_2\right]\left[
\frac{K}{1+ \rho}
\right]+\mathcal O\left(u_0n/s+1\right)&\textrm{for $0\leq x<1/2$,}\\
- \Re \left[ \overline{x_1} x_2\right]\left[
\frac{f}{\theta_0^2} \left(\frac{\gamma_1}{\gamma_2}+\frac{\gamma_2}{\gamma_1}\right)\right]+\mathcal O (n/s)&\textrm{for $-1/2\leq x<0$,}
\end{cases} \\
F_2(e^{\pm i\theta_0})&+F_3(e^{\pm i\theta_0})=\pm iC_{11}C_{21} \Re \left[ \overline{x_1} x_2\right]\left(\frac{\gamma_1}{\gamma_2}+\frac{\gamma_2}{\gamma_1}\right)\left[\frac{2fK}{\theta_0}+\mathcal O (n/s)
\right],\\
\frac{F_4(e^{\pm i\theta_0})}{B_{22}^2C_{21}^2}&=\begin{cases} \Re \left[ \overline{x_1} x_2\right]\left[
\frac{f}{\theta^2_0} \left(\frac{\gamma_1}{\gamma_2}+\frac{\gamma_2}{\gamma_1}\right)\right]+\mathcal O (n/s)&\textrm{for $0\leq x<1/2$,}\\
\Re \left[ \overline{x_1} x_2\right]\left[
\frac{K}{1- \rho}
\right]+\mathcal O\left(u_0n/s+1\right)&\textrm{for $-1/2\leq x<0$.}
\end{cases}
\end{aligned}
\end{equation}
When $k=0$, we have $F_2,F_3,F_4=0$, while $F_1$ is as in \eqref{formula F1-F4} for $0< x <1/2$.

From \eqref{expansion Omega} we have that $\Omega=\bigO \left( \frac{s}{n\log u_0^{-1}}\right)$, and thus substituting $k+x=\frac{n\Omega}{2\pi}$ into \eqref{expansion u0} we have
\begin{equation}\label{formu0}
\frac{8}{u_1^{-1}-u_2^{-1}}e^{-\frac{s\sqrt{|u_1u_2|}}{2(k+x)}}=u_0\left(1+\mathcal O\left(
\frac{s\log u_0^{-1}}{n}+u_0^2(\log u_0^{-1})^2
\right)\right).
\end{equation}

Recalling the expansion of $g_1$ in \eqref{expansion g1a}--\eqref{expansion g1} and the definition of $B$ in \eqref{defE}, and substituting the values of $\kappa_j$ and $\Phi$ from \eqref{kappa mu nu}--\eqref{Phi local} into the expansion of $C$ from \eqref{exp C11}, we obtain that for $k\geq 1$,
\begin{equation} \label{formula B11C11}
\begin{aligned}
(B_{11}^2C_{11}^2)(e^{\pm i\theta_0})&=\begin{cases}
e^{2x\frac{s\sqrt{|u_1u_2|}}{2(k+x)}}
\frac{(k+1)^2}{4(2k+1)\kappa_k^2}\left(1+\mathcal O (r_2)\right) &\textrm{for $0\leq x<1/2$,}\\
e^{2x\frac{s\sqrt{|u_1u_2|}}{2(k+x)}}
\frac{(2k+1)}{4\kappa_k^2}\left(1+\mathcal O (r_2)\right) &\textrm{for $-1/2\leq x<0$,}
\end{cases} \\
(B_{22}^2C_{21}^2)(e^{\pm i\theta_0})&=\begin{cases}-e^{-2x\frac{s\sqrt{|u_1u_2|}}{2(k+x)}}
\pi^2\kappa_{k-1}^2(2k-1)
\left(1+\mathcal O (r_2)\right)&\textrm{for $0\leq x<1/2$,}\\
-e^{-2x\frac{s\sqrt{|u_1u_2|}}{2(k+x)}}
\frac{\pi^2\kappa_{k-1}^2(k-1)^2}{(2k-1)}
\left(1+\mathcal O (r_2)\right)
&\textrm{for $-1/2\leq x<0$,}
\end{cases}\\
\left(C_{11}C_{21}\right)\left(e^{\pm i \theta_0}\right)&=\begin{cases}
\mp \pi i k\frac{k+1}{2k+1}\left(1+\mathcal O (u_0\log u_0^{-1}+s/n)\right)&\textrm{for $0\leq x<1/2$,}\\
\mp \pi i k \frac{k-1}{2k-1}\left(1+\mathcal O (u_0\log u_0^{-1}+s/n)\right)&\textrm{for $-1/2\leq x<0$,}
\end{cases}\\
r_2&=u_0\log u_0^{-1}+\log u_0^{-1}s/n.
\end{aligned}\end{equation}
 When $k=0$ (and $0< x<1/2$) we have
 \begin{equation}
( B_{11}^2C_{11}^2)(e^{\pm i\theta_0})=e^{s\sqrt{|u_1u_2|}}(1+\bigO(u_0\log u_0^{-1}+s^2/n)).
 \end{equation}

Substituting \eqref{fepsilonK}, \eqref{barx1x2}, \eqref{formula B11C11} into \eqref{formula F1-F4} we find that 
\begin{equation} \label{eqnF1-F4}
\begin{aligned}
F_1&=\begin{cases}
\frac{|\widetilde h|^{-2}}{8}\frac{K}{1+\rho}\frac{(k+1)^2}{(2k+1)\kappa_k^2}e^{2x\frac{s\sqrt{|u_1u_2|}}{2(k+x)}}(1+\mathcal O(r_2+s^{-2}))
&\textrm{for $0\leq x<1/2$,}\\
-\frac{|\widetilde h|^{-2}}{8}\frac{\rho}{\theta_0^2K(1-\rho)}\frac{(2k+1)}{\kappa_k^2}e^{2x\frac{s\sqrt{|u_1u_2|}}{2(k+x)}}(1+\mathcal O(r_2+s^{-2}))
&\textrm{for $-1/2\leq x<0$,}
\end{cases}\\
F_2+F_3&=\begin{cases} |\widetilde h|^{-2} \frac{\rho}{1+\rho}\frac{1}{\theta_0} \frac{\pi k (k+1)}{2k+1}(1+\mathcal O(u_0\log u_0^{-1}+\frac{s}{n}+u_0^{2|x|}+s^{-2}))
&\textrm{for $0\leq x<1/2$,}\\
 |\widetilde h|^{-2} \frac{\rho}{1-\rho}\frac{1}{\theta_0} \frac{\pi k (k-1)}{2k-1}(1+\mathcal O(u_0\log u_0^{-1}+\frac{s}{n}+u_0^{2|x|}+s^{-2}))
&\textrm{for $-1/2\leq x<0$,}
\end{cases}\\
F_4&=
\begin{cases}
-\frac{|\widetilde h|^{-2}\rho\pi^2 \kappa_{k-1}^2(2k-1)}{2(1+\rho)K\theta_0^2}e^{-2x\frac{s\sqrt{|u_1u_2|}}{2(k+x)}}
(1+\mathcal O(r_2+s^{-2}))
&\textrm{for $0\leq x<1/2$,}\\
-\frac{|\widetilde h|^{-2}K}{2(1-\rho)}\frac{\pi^2 \kappa_{k-1}^2(k-1)^2}{(2k-1)}e^{-2x\frac{s\sqrt{|u_1u_2|}}{2(k+x)}}
(1+\mathcal O(r_2+s^{-2}))
&\textrm{for $-1/2\leq x<0$,}
\end{cases}
\end{aligned}
\end{equation} 
when evaluated at $e^{\pm i\theta_0}$.
When $k=0$ (and $0<x<1/2$), we have $F_2,F_3,F_4=0$, while $F_1$ is given by 
\begin{equation}
F_1(e^{\pm i\theta_0})=\frac{|\widetilde h|^{-2}K}{2(1+\rho)}e^{s\sqrt{|u_1u_2|}}(1+\bigO(u_0\log u_0^{-1}+s^2/n+s^{-2})).
\end{equation}

Finally we find the order of the term which includes $R'$ in \eqref{formula F(z)}. 
 Using the equation for $R_1$ in \eqref{form R1}, it is readily seen that
\begin{equation}\label{bound dzR}\begin{aligned}
\left|\frac{d}{dz}R(e^{\pm i \theta_0})\right|&\leq
\left(\sup_{u\in \partial U_0}|\Delta_1^{(1)}(u)|\right)\int_{\partial U_0} \frac{du}{2\pi |u-e^{\pm i \theta_0}|^2}
\\&+\sum_{j=1,2}\left(\sup_{u\in \partial U_j}|\Delta_1^{(b_j)}(u)|\right)\int_{\partial U_j} \frac{du}{2\pi |u-e^{\pm i \theta_0}|^2}.\end{aligned}\end{equation}
We recall that the radius of $U_0$ is of size $\mathcal O\left(\frac{s}{n\log u_0^{-1}}\right)$, and that the radius of $U_j$ is of size $\mathcal O(s/n)$ for $j=1,2$. Thus
\begin{equation} \nonumber
\int_{\partial U_0} \frac{du}{2\pi |u-e^{\pm i \theta_0}|^2}=\mathcal O\left(\frac{n}{s\log u_0^{-1}}\right),\qquad \int_{\partial U_j} \frac{du}{2\pi |u-e^{\pm i \theta_0}|^2}=\bigO(n/s) \quad \textrm{for } j=1,2.
\end{equation}
 Substituting the asymptotics for $\Delta_1^{(1)}$ from \eqref{prop error term eqn} and $\Delta_1^{(b_1)},\, \Delta_1^{(b_2)}$ from \eqref{matchingbj} into \eqref{bound dzR}, it follows that
 \begin{equation}
 \frac{d}{dz}R(e^{\pm i \theta_0})=\mathcal O(\widetilde e u_0^{1+2|x|}(\log u_0^{-1})^2n/s+\widetilde e n(u_0\log u_0^{-1})^2+n/s^2).
 \end{equation}
  Thus, we have from \eqref{x1}, since $\widetilde h, \,\, \widetilde h^{-1}, \, \gamma_j(1)=\bigO(1)$, that also
\begin{equation}\label{order derivative x1}
\frac{d}{dz}x_1(e^{\pm i \theta_0})=\mathcal O(\widetilde e u_0^{1+2|x|}(\log u_0^{-1})^2n/s+\widetilde e n(u_0\log u_0^{-1})^2+n/s^2).
\end{equation}
The formula for $R'A$ is given by \eqref{formula A1}, \eqref{formula A2} but with $x_1,x_2$ replaced by the derivatives $x_1',x_2'$. Recall that 
\begin{equation}\begin{aligned}C_{11},C_{21}&=\mathcal O\left(\sqrt{k}\right)=\mathcal O\left(\sqrt{s/\log u_0^{-1}}\right),\\
B_{11}&=\mathcal O(u_0^{-x}),\quad B_{22}=\mathcal O(u_0^{x}), \quad 
K=\mathcal O(n/s), \quad f=\mathcal O\left(\frac{s}{n}u_0^{1-2|x|}\right).\end{aligned}\end{equation} 
 From \eqref{order derivative x1} we obtain that 
\begin{multline} \label{order R'}
\Re \left[ \left(\overline{(RABC)_{11}}(R'ABC)_{11}\right)(e^{\pm i \theta_0})\right] \\=\mathcal O(\widetilde e nu_0\log u_0^{-1}+\widetilde e nsu_0^{2-2|x|}\log u_0^{-1}+nu_0^{-2|x|}/(s\log u_0^{-1})).
\end{multline}

By substituting \eqref{formula RABC'}, \eqref{formula RAB'C}, \eqref{form RA'BC}, \eqref{order R'} into the definition of $F(z)$ from \eqref{formula F(z)}, and substituting the resulting expression into the expression for the differential identity \eqref{diff id}, we obtain the following proposition.

\begin{Prop}
We have the following asymptotics for $\log D_n(J)$, as $u_0\to0$ and $s,n\to \infty$ such that $su_0\log u_0^{-1}\to 0$ and $s/n\to 0$:
\begin{multline}\label{asymdiffid}
\log D_n(J)=\log D_n(J_2)+\frac{|\widetilde h|^2}{\pi}\int_0^{\theta_0}
\Bigg[\left(\sum_{j=0}^4F_j(e^{i\theta})+F_j(e^{-i\theta})\right)(1+\mathcal O(s/n))\\+
\mathcal O(n+ns(1+u^{1-2|x|}\log u^{-1})^2u^{2-2|x|}\log u^{-1}+nu^{-2|x|}/(s\log u^{-1}))\Bigg]d\theta,
\end{multline}
where the integration variable $\theta=\frac{s}{n}u$, and where the asymptotics of $F_0(z)$ are given in \eqref{formula RABC'} and the asymptotics of $F_1,\,F_2,\, F_3,\, F_4$ are given in \eqref{eqnF1-F4}.
\end{Prop}

\subsection{Integration of Differential Identity} \label{integration of diffid}
We evaluate the integral in formula \eqref{asymdiffid} asymptotically 
to prove Theorem \ref{prop toeplitz}. 

Using \eqref{formu0}, \eqref{def f}  we find that
\begin{equation} \label{thetaalpha}
\begin{aligned}
\theta_0&=\theta_0(k;x)=\frac{su_0}{n}=\frac{8s}{(u_1^{-1}-u_2^{-1})n}e^{-\frac{s\sqrt{|u_1u_2|}}{2(k+x)}}\left(1+\mathcal O \left(\frac{s}{n}\log u_0^{-1}+u_0^2(\log u_0^{-1})^2\right)\right)\\
\frac{d\theta_0}{dx}&=\frac{s\theta_0\sqrt{|u_1u_2|}}{2(k+x)^2}\left(1+\mathcal O \left(\frac{s}{n}\log u_0^{-1}+u_0^2(\log u_0^{-1})^2\right)\right),\\
\rho&=\begin{cases}
\frac{1}{2\pi \kappa_k^2} e^{\frac{s\sqrt{|u_1u_2|}}{2(k+x)}(-1+2x)}\left(1+\mathcal O \left(\frac{s}{n}\log u_0^{-1}\right)\right)&\textrm{for $0\leq x<1/2$}\\
-2\pi \kappa_{k-1}^2 e^{\frac{s\sqrt{|u_1u_2|}}{2(k+x)}(-1-2x)}\left(1+\mathcal O \left(\frac{s}{n}\log u_0^{-1}\right)\right)&\textrm{for $-1/2 \leq x< 0$.}
\end{cases}
\end{aligned}
\end{equation}
 Letting $k$ in the expression for $\theta_0$ in \eqref{thetaalpha} be fixed, we integrate in $\theta_0$, denoting
\begin{equation}\int_{\theta_0(k,-1/2)}^{\theta_0(k,1/2)}*\,\,d\theta=\int_{x=-1/2}^{x=1/2}*\,\,d\theta_0.
\end{equation}  
Note that by \eqref{kappa mu nu} with $\eta_2-\eta_1=4$,
\begin{equation}
\label{kappa k-1/k}
\frac{\kappa_{k-1}^2}{\kappa_k^2}=\frac{4k^2}{(2k+1)(2k-1)}.
\end{equation}

 We integrate $F_1$ from \eqref{eqnF1-F4}, changing the variable of integration using \eqref{thetaalpha} and recalling $K$ from \eqref{def K}
  and $\rho$ from \eqref{thetaalpha}
  to find that for $k\geq 1$,

\begin{equation}\label{int F1} \begin{aligned}
&\frac{|\widetilde h|^2}{\pi} \int_{x=-1/2}^{x=1/2} F_1(e^{\pm i\theta_0}) d\theta_0\\
&=\frac{(k+1)^2}{2(2k+1)}\int_0^{1/2}
\left(\frac{1}{2\pi\kappa_k^2}e^{(2x -1)\frac{s\sqrt{|u_1u_2|}}{2(k+x)}}\right)\frac{s\sqrt{|u_1u_2|}}{2(k+x)^2} \frac{(1+\mathcal O(r_2+s^{-2}))dx}{1+\frac{1}{2\pi \kappa_k^2}e^{(2x -1)\frac{s\sqrt{|u_1u_2|}}{2(k+x)}}}\\
&+\frac{(2k+1) \kappa_{k-1}^2}{8\kappa_k^2}\int_{-1/2}^0 \frac{s\sqrt{|u_1u_2|}}{2(k+x)^2}\left(1-\frac{2\pi \kappa_{k-1}^2e^{-(2x +1)\frac{s\sqrt{|u_1u_2|}}{2(k+x)}}}{1+2\pi \kappa_{k-1}^2e^{-(2x +1)\frac{s\sqrt{|u_1u_2|}}{2(k+x)}}}\right)
(1+\mathcal O(r_2+s^{-2}))dx\\
&=\left[\frac{(k+1)^2}{2(2k+1)^2}\log \left(1+\frac{1}{2\pi \kappa_k^2}e^{(2x -1)\frac{s\sqrt{|u_1u_2|}}{2(k+x)}}\right)\right]_{x=0}^{1/2}
+\Bigg[\frac{ k^2}{2(2k-1)^2}\times \\
&\times \log \left(1+2\pi \kappa_{k-1}^2 e^{-(2x +1)\frac{s\sqrt{|u_1u_2|}}{2(k+x)}}\right)
-\frac{ s k^2\sqrt{|u_1u_2|}}{4(2k-1)(k+x)}
\Bigg]_{x=-1/2}^0 \\
&+\mathcal O\left(\max_{x\in [-1/2,1/2)} \left[r_2\log u_0^{-1}+s^{-2}\log u_0^{-1}\right] \right)\\
&=\frac{(k+1)^2}{2(2k+1)^2}\log \left(1+\frac{1}{2\pi \kappa_k^2}\right)
-\frac{ k^2}{2(2k-1)^2}\log \left(1+2\pi \kappa_{k-1}^2 \right)
+ s \sqrt{|u_1u_2|}\frac{k}{4(2k-1)^2} \\
&+\mathcal O\left(\max_{x\in [-1/2,1/2)}\left[r_2\log u_0^{-1}+s^{-2}\log u_0^{-1}\right]\right),
\end{aligned}
\end{equation}
where $r_2$ is given in \eqref{formula B11C11}, $u_0=u_0(k,x)$.

When $k=0$ we have for $x\in [0,1/2)$
\begin{equation} \label{F1k0}
\frac{ |\widetilde h|^2}{\pi}\int_{x=0}^{x=x} F_1(e^{\pm i\theta_0}) d\theta_0
=\frac{1}{2}\log \left(1+\frac{2}{\pi}e^{-(2x +1)\frac{s\sqrt{|u_1u_2|}}{2x}}\right)+\mathcal O\left(su_0^{1-2x}\left(
u_0\log u_0^{-1}+s^2/n+s^{-2}\right)\right),
\end{equation}
where, in the $\bigO$ error term, $u_0=u_0(k=0,x)$.
Similarly, we integrate $F_4$ for $k\geq 1$:
\begin{equation}
\begin{aligned}
&\frac{|\widetilde h|^2}{\pi} \int_{-1/2}^{1/2}F_4(e^{\pm i\theta_0}) d\theta_0\\
&  =\left[
\frac{(k-1)^2}{2(2k-1)^2}\log \left(1+2\pi \kappa_{k-1}^2e^{(-2x -1)\frac{s\sqrt{|u_1u_2|}}{2(k+x)}}\right)
\right]_{x=-1/2}^{0}\\
&+\left[
\frac{ k^2}{2(2k+1)^2}\log \left(1+\frac{1}{2\pi \kappa_k^2} e^{(2x -1)\frac{s\sqrt{|u_1u_2|}}{2(k+x)}}\right)
+\frac{ sk^2 \sqrt{|u_1u_2|}}{4(2k+1)(k+x)}
\right]_{x=0}^{1/2}\\&
+\mathcal O\left(\max_{x\in [-1/2,1/2)}\left[r_2\log u_0^{-1}+s^{-2}\log u_0^{-1}\right]\right)
\\&=
-\frac{(k-1)^2}{2(2k-1)^2}\log \left(1+2\pi \kappa_{k-1}^2\right)
+\frac{ k^2\log \left(1+\frac{1}{2\pi \kappa_k^2}\right)}{2(2k+1)^2}
-\frac{ sk \sqrt{|u_1u_2|} }{4(2k+1)^2}\\&+\mathcal O\left(\max_{x\in [-1/2,1/2)}\left[r_2\log u_0^{-1}+s^{-2}\log u_0^{-1}\right]\right).
\end{aligned}
\end{equation}
 When $k=0$ and $x\in[0,1/2)$, we have $F_4=0$.
 Thus, for $k\geq 1$,
\begin{equation}\label{F1+F4}
\begin{aligned}
&\frac{|\widetilde h|^2}{\pi}\int_{-1/2}^{1/2} (F_1+F_4)(e^{\pm i\theta_0}) d\theta_0=
- \frac{k^2+(k-1)^2}{2(2k-1)^2}\log \left(1+2\pi \kappa_{k-1}^2\right)
\\&+\frac{ k^2 +(k+1)^2}{2(2k+1)^2}\log \left(1+\frac{1}{2\pi \kappa_k^2}\right)
-\frac{ sk \sqrt{|u_1u_2|} }{4}\left(\frac{1}{(2k+1)^2}-\frac{1}{(2k-1)^2}\right)\\&+\mathcal O\left(\max_{x\in [-1/2,1/2)}\left[u_0(\log u_0^{-1})^2+\frac{s}{n}(\log u_0^{-1})^{2}+s^{-2}\log u_0^{-1}\right] \right).
\end{aligned}
\end{equation}

If $-1/2<x<1/2$, then for $k\geq 1$
\begin{equation}\label{F1+F4 x}\begin{aligned}
&\frac{|\widetilde h|^2}{\pi}\int_{-1/2}^{x} (F_1+F_4)(e^{\pm i\theta_0}) d\theta_0\\
&=- \frac{k^2+(k-1)^2}{2(2k-1)^2}\log(1+2\pi \kappa_{k-1}^2)-\frac{s}{4} \sqrt{|u_1u_2|}w_1(x)+r_3,\qquad r_3=o(1)\\
w_1(x)&=\begin{cases}
-\frac{k}{(2k-1)^2}+\frac{k}{2k+1}\frac{x}{k+x}&\textrm{for $0\leq x<1/2$,}\\
-\frac{k^2}{(2k-1)^2}\frac{1+2x}{k+x}& \textrm{for $-1/2\leq x<0$.}
\end{cases}
\end{aligned}\end{equation}
 We keep the term $r_3$ in \eqref{F1+F4 x} as it is not uniform in $x$, and is not small as $x$ approaches $\pm 1/2$.
\begin{equation} \label{o(1)1}
\begin{aligned}
r_3&= \frac{k^2+(k+1)^2}{2(2k+1)^2}\log \left(1+\frac{1}{2\pi \kappa_k^2} e^{(2x -1)\frac{s\sqrt{|u_1u_2|}}{2(k+x)}}\right)+\\&+ \frac{k^2+(k-1)^2}{2(2k-1)^2}\log \left(1+2\pi \kappa_{k-1}^2 e^{(-2x -1)\frac{s\sqrt{|u_1u_2|}}{2(k+x)}}\right)\\&+\mathcal O\left(\max_{x\in [-1/2,1/2)}\left[u_0(\log u_0^{-1})^2+\frac{s}{n}(\log u_0^{-1})^{2}+s^{-2}\log u_0^{-1}\right]\right).
\end{aligned}\end{equation}

Similarly but simpler, using \eqref{eqnF1-F4} and \eqref{thetaalpha}, we find that for $k\geq 1$,
\begin{equation}\label{int F2F3}
\begin{aligned}
&\frac{|\widetilde h|^2}{\pi}\int_{-1/2}^{x} (F_2+F_3)(e^{\pm i\theta_0}) d\theta_0\\
&=\begin{cases}
-  \frac{k(k-1)}{(2k-1)^2}\log \left(1+2\pi \kappa_{k-1}^2 \right)+r_4&\textrm{for $|x|<1/2$,}\\
\begin{matrix}-\frac{ k(k-1)}{(2k-1)^2}\log \left(1+2\pi \kappa_{k-1}^2 \right)+\frac{ k(k+1)}{(2k+1)^2}\log \left(1+\frac{1}{ 2\pi \kappa_k^2}\right)\\
\qquad\qquad +\bigO\left(\max_{x\in [-1/2,1/2)} \left[ u_0\log u_0^{-1}+\frac{1}{s}+\frac{s}{n}\right] \right)
\end{matrix}
&\textrm{for $x=1/2$,}
\end{cases}
\end{aligned}
\end{equation}
where $r_4=o(1)$.
 When $|x|<1/2$ we again keep track of the error term
\begin{equation}\label{o(1)2}
\begin{aligned}
&r_4= \frac{k(k+1)}{(2k+1)^2}\log \left(1+\frac{1}{2\pi \kappa_k^2} e^{(2x -1)\frac{s\sqrt{|u_1u_2|}}{2(k+x)}}\right)\\&+ \frac{k(k-1)}{(2k-1)^2}\log \left(1+2\pi \kappa_{k-1}^2 e^{(-2x -1)\frac{s\sqrt{|u_1u_2|}}{2(k+x)}}\right)+\bigO\left(\max_{x\in [-1/2,1/2)} \left[ u_0\log u_0^{-1}+\frac{1}{s}+\frac{s}{n}\right] \right).
\end{aligned}\end{equation}

When $k=0$, we have $F_2,F_3,F_4=0$, and thus the integral of $F_1+F_2+F_3+F_4$ over $[x=0,x=x_0<1/2]$ for $k=0$ is given by \eqref{F1k0}.
For any $k\geq 1$ and $-1/2\leq x<1/2$,
  combining \eqref{F1k0}, \eqref{F1+F4}, \eqref{F1+F4 x}, \eqref{int F2F3}  
\begin{equation} \label{integral F1+2+3+4}
\begin{aligned}
&\frac{|\widetilde h|^2}{\pi} \int_{0}^{\theta_0} (F_1+F_2+F_3+F_4)(e^{\pm i\theta}) d\theta
=-\frac{ s}{4} \sqrt{|u_1u_2|}w_1(x)+\frac{1}{2}  \delta_k (x)\\ &-\frac{1}{2} \sum_{j=0}^{k-1} \left( \log 2\pi \kappa_j^2+\frac{ s}{2} \sqrt{|u_1u_2|}\left(
\frac{j}{(2j+1)^2}-\frac{j}{(2j-1)^2}\right)\right)
\\&+\bigO\left(su_0\log (u_0)^{-1}+s^3/n+1/\log(u_0)^{-1}+1/s\right),\\
&\delta_k (x)=\begin{cases}
 \log (1+2\pi \kappa_{k-1}^2) &\textrm{for $x=-1/2$}\\
o(1)& \textrm{for $|x|<1/2$.}
\end{cases}
\end{aligned}
\end{equation}
By \eqref{o(1)1} and \eqref{o(1)2}, it follows that for $|x|<1/2$, $ \delta_k(x)$ is given explicitly as
\begin{equation}
 \delta_k(x)= \log \left(1+2\pi \kappa_{k-1}^2 e^{-(2x +1)\frac{s\sqrt{|u_1u_2|}}{2(k+x)}}\right)
+ \log \left(1+(2\pi \kappa_k^2)^{-1}e^{(2x -1)\frac{s\sqrt{|u_1u_2|}}{2(k+x)}}
\right).
\end{equation}

Furthermore, by \eqref{formula RABC'} (which holds for all $k\geq 0$) and Proposition \ref{Prop x1x2}, we find that 
\begin{equation} \label{integral F0}
\begin{aligned}
&\frac{|\widetilde h|^2}{\pi}\int_{0}^{\theta_0} \left(F_0\left(e^{i\theta}\right)+F_0\left(e^{-i\theta}\right) \right)d\theta \\&=
 s \sqrt{|u_1u_2|} \left(w_2(x)+\sum_{j=1}^{k-1}\frac{j(j+1)}{(2j+1)^2}+\frac{j(j-1)}{(2j-1)^2}\right)+\mathcal O\left(\frac{s^3}{n}+su_0\log u_0^{-1}+\frac{1}{\log u_0^{-1}}\right)
\\ &w_2(x)=\begin{cases}
\frac{k(k-1)}{(2k-1)^2}+\frac{k(k+1)}{2k+1}\frac{x}{k+x} &\textrm{for $0\leq x<1/2$}\\
\frac{k^2(k-1)}{(2k-1)^2}\frac{1+2x}{k+x}& \textrm{for $-1/2 \leq x <0$,}
\end{cases}
\end{aligned}\end{equation}
for $k\geq 1$ and $F_0=0$ for $k=0$.

\subsubsection{Proof of Theorem \ref{prop toeplitz}}
We sum together \eqref{integral F1+2+3+4}, \eqref{integral F0}, and substitute the result into \eqref{asymdiffid}, to find that
 \begin{equation}\label{formulaDn}
 \begin{aligned}
 \log D_n \left(J\right)&=
 \log D_n(J_2) 
 +s \sqrt{|u_1u_2|} \left(w_3(x)+\sum_{j=1}^{k-1}\frac{2j^2}{4j^2-1}\right) - \sum_{j=0}^{k-1} \log 2\pi \kappa_j^2\\&
+\delta_k(x)+\mathcal O (u_0+su_0\log u_0^{-1}+s^3/n+1/\log u_0^{-1}).\\
w_3&=-\frac{w_1}{2}+w_2= 
\frac{k^2}{2(2k-1)}\frac{1+2x}{k+x},
\end{aligned}\end{equation}  
where $\sum_{j=1}^0=0$.
The first sum can be evaluated by noting that
\begin{equation}
\sum_{j=1}^{k-1}\frac{2j^2}{4j^2-1}=\frac{k(k-1)}{2k-1},\end{equation}
and, it follows that 
\begin{equation}\label{formw}
\begin{aligned}
s\sqrt{|u_1u_2|}\left(w_3+\sum_{j=1}^{k-1}\frac{2j^2}{4j^2-1}\right)&=\frac{s}{2}\sqrt{|u_1u_2|}\left(\omega-\frac{x^2}{\omega}\right), \quad \omega=k+x.
\end{aligned}\end{equation}
The sum with the leading coefficients $\kappa_j$ is given by
\begin{equation}\label{Barnes}
 \prod_{j=0}^{k-1}\kappa_j^2=4^{-k^2} \frac{G(2k+1)}{G(k+1)^4},
\end{equation}
where $G$ is the Barnes G-function. 
By substituting \eqref{formw}, \eqref{Barnes} into \eqref{formulaDn} we find that
\begin{equation} \label{almost there}
 \begin{aligned}
 \log D_n \left(J(u_0)\right)&=
 \log D_n\left(J_2\right)
 +\frac{s}{2}\sqrt{|u_1 u_2|}\left(\omega-\frac{x^2}{\omega}\right)+c(k)+\delta_k(x)
\\&+\mathcal O (s^3/n+su_0\log u_0^{-1}+1/\log u_0^{-1}\}).
\end{aligned}\end{equation}  

Now define $\alpha=u_2/2$, $\beta=u_1/2$, and
\begin{equation}
\nu=\frac{8}{\beta^{-1}-\alpha^{-1}}e^{-\frac{s\sqrt{|\alpha \beta|}}{\omega}}.
\end{equation}
Then, by \eqref{thetaalpha},
\begin{equation} \label{u0nu}
\frac{u_0}{2}=\nu \left(1+\bigO \left(\frac{s}{n}\log \nu^{-1}+\nu^2(\log\nu^{-1})^2\right)\right).
\end{equation}
It is an easy exercise using \eqref{u0nu}, the continuity of $w,x$ as functions of $u_0$, and \eqref{almost there} to show that
\begin{equation}\label{really almost there}
\log D_n(\nu)=\log D_n(u_0)+\mathcal O (u_0+su_0\log u_0^{-1}+s^3/n+1/\log u_0^{-1}).
\end{equation}
Substituting the asymptotics from \eqref{almost there} into \eqref{really almost there}, and using uniformity of the error terms, we prove Theorem \ref{prop toeplitz}.

\subsection{Proof of Proposition \ref{Prop x1x2}} \label{consideration of error terms}
We consider the small norm matrices, and we prove Proposition \ref{Prop x1x2}.
From \eqref{x1} it follows that
\begin{multline}\label{Re x1x2}
\Re[(\overline{x_1} x_2)(e^{\pm i\theta_0})]=-|\widetilde h|^{-2}\Re[\overline{\gamma_1}\gamma_2]\left(1+2\Re R_{1,11}(e^{\pm i\theta_0})
+2\Re R_{2,11}(e^{\pm i\theta_0})\right)\\+\mathcal O \left(\left(|\gamma_1|^2-|\gamma_2|^2\right)+
(|R_{1,11}|+|R_{12}|)^2+||R_3||
\right),
\end{multline}
where $||\,.\,||$ denotes the value of the largest element of the matrix in absolute value.
 From \eqref{epsilon1/2} we have
\begin{equation}\begin{aligned}
\Re[\overline{\gamma_1}\gamma_2]&=-1/2+\mathcal O(s/n),\\
|\gamma_1|^2-|\gamma_2|^2&=\mathcal O(s/n).
\end{aligned}\end{equation}
 It follows from Proposition \ref{Proposition error term}, \eqref{Delta11}--\eqref{Delta12}, \eqref{form R1} that
\begin{equation} \label{error R1R2}
\begin{aligned}
R_1(e^{\pm i \theta_0})&=\mathcal O(s^{-1}+\widetilde eu_0^{1+2|x|}\log u_0^{-1}+\widetilde esu_0^2\log u_0^{-1}),\\
R_2(e^{\pm i \theta_0})&=\mathcal O((s^{-1}+\widetilde eu_0^{1+2|x|}\log u_0^{-1})^2+\widetilde esu_0^2\log u_0^{-1}),\\
R_3(e^{\pm i \theta_0})&=\mathcal O(||R_1||\,||R_2||+(u_0\log u_0^{-1})^3).
\end{aligned}\end{equation}
 From \eqref{Re x1x2}--\eqref{error R1R2}, it follows that 
\begin{multline}\label{error x1x2}
\Re[(\overline{x_1} x_2)(e^{\pm i\theta_0})]=\frac{1}{2}|\widetilde h|^{-2}\left(1+2\Re R_{1,11}(e^{\pm i\theta_0})+2\Re R_{2,11}(e^{\pm i\theta_0})\right)\\+\mathcal O \left(  \left(s^{-1}+\widetilde eu_0^{1+2|x|}\log u_0^{-1}+\widetilde esu_0^2\log u_0^{-1}\right)^2+s/n\right).
\end{multline}
We will evaluate $\Re R_{1,11}$ and $\Re R_{2,11}$ to prove Proposition \ref{Prop x1x2}.

We recall from \eqref{form R1} that $R_1$ is a sum of 3 terms. The first term is an integral of $\Delta_1^{(1)}$, and the two other terms are integrals of $\Delta_1^{(b_1)}$ and $\Delta_1^{(b_2)}$. We first evaluate the contribution from the terms $\Delta_1^{(b_1)}$ and $\Delta_1^{(b_2)}$.
\subsubsection{Contribution to $R_1$ from $\Delta_1^{(b_1)}$ and $\Delta_1^{(b_2)}$}
It follows from \eqref{Delta11} and \eqref{Delta12} that 
\begin{equation} \label{formula Delta11}
\frac{\Delta^{(b_1)}_{1,11}(z)}{b_1-e^{\pm i \theta_0}}=\frac{C_{\Delta,1}}{z-b_1}+\mathcal O(1),\qquad \frac{\Delta^{(b_1)}_{2,11}(z)}{b_2-e^{\pm i \theta_0}}=\frac{C_{\Delta,1}}{z-b_2}+\mathcal O(1)
\end{equation}
as $z\to b_1$ and $z\to b_2$ respectively, and that the matrices $C_{\Delta,1}$ and $C_{\Delta_2}$ are given as follows as $s\to \infty$
\begin{equation} 
\begin{aligned} \label{CDelta}
C_{\Delta,1}&=-\frac{
\Bigg[i
+\frac{fi}{\theta_1}(\psi+\psi^{-1})
+\frac{f^2}{\theta_1^2}\left(2i+\left(\frac{\gamma_1}{\gamma_2}-\frac{\gamma_2}{\gamma_1}\right)\right)
\Bigg]}{8s(u_1\mp u_0)\left(1-4\frac{k+x}{s(u_1-u_2)}\left|\frac{u_1}{u_2}\right|^{1/2}
\right)}(1+\bigO(s/n)),\\
C_{\Delta,2}&=\frac{
\Bigg[-i
+\frac{fi}{\theta_1}(\psi+\psi^{-1})
+\frac{f^2}{\theta_1^2}\left(-2i+\left(\frac{\gamma_1}{\gamma_2}-\frac{\gamma_2}{\gamma_1}\right)\right)
\Bigg]}{8s(u_2\mp u_0)\left(1-4\frac{k+x}{s(u_1-u_2)}\left|\frac{u_1}{u_2}\right|^{1/2}
\right)}(1+\bigO(s/n)).
\end{aligned} \end{equation}
 We recall that $f$ is real to the main order and
from \eqref{def f}, \eqref{epsilon1/2}  we have
\begin{equation}
\Im \left(\psi+\psi^{-1}\right)=\mathcal O(s/n),\qquad \Re\left(\frac{\gamma_1}{\gamma_2}-\frac{\gamma_2}{\gamma_1}\right)=\mathcal O(s/n).
\end{equation}
Since  the interior of the bracket $[\,\,]$ in \eqref{CDelta} is imaginary to main order, we can calculate the residue in the integral of \eqref{formula Delta11} to find:
\begin{equation} \label{small int1}
\Re \left[\sum_{j=1,2}\int_{\partial U_j}\frac{\Delta_{1,11}^{(b_j)}(u)}{u-e^{\pm i \theta_0}}\frac{du}{2\pi i}\right]=\mathcal O(s/n).
\end{equation}

\subsubsection{Contribution to $R_1$ from $\Delta_1^{(1)}$}
Denote
\begin{equation}\label{def x0y0}
\begin{aligned}
y_0(z)&=-2i(\gamma_1^2-\gamma_2^2)=-i \left( \left( \frac{z-b_2}{z-b_1}\right)^{1/2}+\left( \frac{z-b_1}{z-b_2}\right)^{1/2}\right),\\
x_0(z)&=-4\gamma_1\gamma_2=-i \left( \left( \frac{z-b_2}{z-b_1}\right)^{1/2}-\left( \frac{z-b_1}{z-b_2}\right)^{1/2}\right),
\end{aligned}
\end{equation}
with branch cuts on $J_2$ such that the square root is positive as $z\to \infty$.
Our goal is to evaluate the terms in \eqref{PM-12}, and given a matrix $X$ we denote
\begin{equation}
(LX)(z)= \widetilde E(z)B(z)X B^{-1}(z) \widetilde E^{-1}(z) .
\end{equation}
Define $D(z)$ and $E^{(\pm)}(z)$ by
\begin{equation}\nonumber
D(z)=L\begin{pmatrix}
1&0\\0&-1
\end{pmatrix},\quad
E^{(+)}(z)=L\begin{pmatrix}
0&1\\0&0
\end{pmatrix},\quad
E^{(-)}(z)=L\begin{pmatrix}
0&0\\1&0
\end{pmatrix}.
\end{equation}
Then from \eqref{PM-12} it follows that
\begin{equation} \label{Delta1}
\Delta_{1,11}^{(1)}(z)=\begin{cases}\zeta^{-1}(z)\left(\Phi_{1,11}D(z)+ B_{22}^2(z) \Phi_{1,21}E^{(-)}(z)\right),& \textrm{$0\leq x<1/2$,}
\\
\zeta^{-1}(z)\left(\Phi_{1,11}D(z)+ B_{11}^2(z) \Phi_{1,12}E^{(+)}(z)\right),& \textrm{$-1/2\leq x<0$.}\end{cases}
\end{equation}

Recalling the definition of $B$ in \eqref{defE}, the definition of $F$ in \eqref{def F}, and the definition of $\widetilde E$ in \eqref{def Etilde}, we find that
 \begin{align}
 \nonumber D(z)&=\frac{i}{2}y_0(z)-\frac{x_0(z)f}{2(z-1)}(\psi+\psi^{-1})
-\frac{f^2}{(z-1)^2}\left(iy_0(z)+\frac{x_0(z)}{2}\left(-\psi+\psi^{-1}\right)\right),\\
  \label{CandD} E^{(\pm)}(z)&=\frac{1}{4}\Bigg(x_0(z)+\frac{f}{z-1}\left(iy_0(z)(\psi+\psi^{-1})\pm 2(-\psi+\psi^{-1})\right)\\& \nonumber
\qquad \qquad -\frac{f^2}{(z-1)^2}\left(2x_0(z)+iy_0(z)\left(\psi-\psi^{-1}\right)\mp 2(\psi+\psi^{-1})
\right)\Bigg).
\end{align}

We analyze the sign of $\Delta_{1,11}^{(1)}$ in \eqref{Delta1}.
From \eqref{def f} and \eqref{epsilon1/2} we have
\begin{equation}
\Im \left(\psi+\psi^{-1}\right)=\mathcal O(s/n), \quad \Re \left(\psi-\psi^{-1}\right)=\mathcal O (s/n).	
\end{equation}
 From \eqref{def x0y0} we see that
\begin{equation}\begin{aligned}
\Im\left(x_0\left(e^{i\theta}\right) \right)&=\mathcal O(s/n),
&\Im\left(\frac{d^j}{d\theta^j}x_0\left(e^{i\theta}\right) \right)&=\mathcal O(n^{j-1}/s^{j-1}),\\
\Im\left(y_0\left(e^{i\theta}\right) \right)&=\mathcal O(s/n),
\qquad &\Im\left(\frac{d^j}{d\theta^j}y_0\left(e^{i\theta}\right) \right)&=\mathcal O(n^{j-1}/s^{j-1}),
\end{aligned}
\end{equation}
for $e^{i\theta}\in U_0\cap C$.  Write \eqref{Delta1} in the form
\begin{equation}
\Delta_1^{(1)}(z)=\frac{\Delta_{1,-3}^{(1)}(z)}{(z-1)^3}+\frac{\Delta_{1,-2}^{(1)}(z)}{(z-1)^2}+\frac{\Delta_{1,-1}^{(1)}(z)}{(z-1)},
\end{equation}
where $\Delta_{1,-j}^{(1)}$ are analytic functions in $z$ in $U_0$. Then a calculation of residues gives the following expansion as $z\to 1$:
\begin{multline}\label{residue int Delta}
\int_{\partial U_0}\frac{\Delta_1^{(1)}(u)}{u-z}\frac{du}{2\pi i}=
\frac{1}{6}\frac{d^3}{dz^3}\left(\Delta_{1,-3}^{(1)}\right)(1)+\\
\frac{1}{2} \frac{d^2}{dz^2}\left(\Delta_{1,-2}^{(1)}\right)(1)+
 \frac{d}{dz}\left(\Delta_{1,-1}^{(1)}\right)(1)+\mathcal O(z-1).
\end{multline}
We note that $\zeta$ is real on $J_1$, and recall the expansion of $\zeta$ in \eqref{zeta exp}--\eqref{zeta0zeta1}. We also note that $\Im f=\bigO (s^2/n^2)$, and that $\Phi_{1,11}$ is real but that $\Phi_{1,12}$ and $\Phi_{1,21}$ are imaginary.
Combining with \eqref{Delta1}--\eqref{residue int Delta}, we conclude that
\begin{equation} \label{small int2}
\Re \left(\int_{\partial U_0}\frac{\Delta_{1,11}^{(1)}(u)}{u-e^{\pm i\theta_0}}\frac{du}{2\pi i} \right)=\mathcal O(s/n).
\end{equation}
 As a consequence of \eqref{small int1} and \eqref{small int2} we have
\begin{equation}\label{small int4}
\Re R_{1,11}(e^{\pm i \theta_0})=\mathcal O(s/n).
\end{equation}

\subsubsection{Order of $R_2(e^{\pm i \theta_0})$}
 From \eqref{CandD} and \eqref{PM-12} we have
\begin{equation} \nonumber \begin{aligned}
\Delta_{2,11}^{(1)}(z)&=\begin{cases} \zeta^{-2}(z)\Big[\Phi_{2,11}D(z)+ B_{11}^2(z) (\Phi_{2,12}  -\Phi_{1,12}\Phi_{1,22})E^{(+)}(z)+\bigO (1)\Big], & \textrm{$0\leq x<\frac{1}{2}$,}
\\
\zeta^{-2}(z)\Big[\Phi_{2,11}D(z)+ B_{22}^2(z) (\Phi_{2,21}-\Phi_{1,21}\Phi_{1,11})E^{(-)}(z)+\bigO (1)\Big], & \textrm{$-\frac{1}{2}\leq x<0$,}\end{cases} \\
\Xi^{(1)}(z)&=\begin{cases}-\zeta^{-3}(z)\left[ B_{11}^2(z) \Phi_{1,12}\Phi_{2,22}E^{(+)}(z)+\bigO (1)\right],& \textrm{$0\leq x<\frac{1}{2}$,}
\\
-\zeta^{-3}(z)\left[ B_{22}^2(z) \Phi_{1,21}\Phi_{2,11}E^{(-)}(z)+\bigO (1)\right],& \textrm{$-\frac{1}{2}\leq x<0$.}\end{cases} 
\end{aligned}\end{equation} 
By inspection of the signs of each element, it follows that
\begin{equation}
\Re \left(\int_{\partial U_0}\frac{\Delta_{2,11}^{(1)}(u)+\Xi_{11}^{(1)}(u)}{u-e^{\pm i\theta_0}}\frac{du}{2\pi i} \right)=\mathcal O(s/n).
\end{equation}
The remaining contributions to $R_{2,11}$, defined in \eqref{PM-12}, are calculated using rougher estimates from \eqref{error R1R2}. Thus it follows that
\begin{equation} \label{small int3}
\Re R_{2,11}=\mathcal O\left(s/n +(\widetilde eu_0^{1+2|x|}\log u_0^{-1}+\widetilde esu_0^2\log u_0^{-1}+s^{-1})^2\right).
\end{equation}
Substituting \eqref{small int4} and \eqref{small int3} into \eqref{error x1x2} yields Proposition \ref{Prop x1x2}.

\section{Connection to the asymptotics of \cite{DIZ}}\label{sec5}
Consider the Deift-Its-Zhou asymptotics (\ref{2 gap formula}) for 2 fixed gaps $A=(\alpha_1,\beta_1)\cup(\alpha_2,\beta_2)$. Without loss of generality, we assume that $\nu\equiv \alpha_2=-\beta_1$. We also denote $\alpha=\alpha_1$, $\beta=\beta_2$. The following lemma shows that these asymptotics 
can be extended (with a worse error term) to the region where $\nu$ is decreasing at a sufficiently slow rate as $s\to\infty$.
This gives a connection to the asymptotics of Theorem \ref{Theorem u0s to 0} (see Remark \ref{Remark12} following Theorem \ref{Theorem u0s to 0}). 

\begin{Lemma}\label{lemmaconnect}
Let $\ep>0$.
As $s\to\infty$, uniformly for $\nu\in(\nu_1,\nu_2)$ where $\nu_2>0$ is fixed and $\nu_1=\nu_1(s)\to 0$ s.t.
\[
s\nu_1^{1/2+\ep}\to\infty,
\]
we have 
\be\label{Lemma-form}
\frac{\partial}{\partial s}\log P_s(A)=
-2G_0(\alpha,\beta,\nu) s +\frac{\partial}{\partial s} \log\theta(sV(\alpha,\beta,\nu);\tau(\alpha,\beta,\nu))+O((s\nu^{2|<sV>|+\ep})^{-1}),
\ee
where $V,\tau,G_0$ are defined in equations (\ref{Vtau11}), (\ref{G010}) above with 
$\nu=\alpha_2=-\beta_1$, $\alpha=\alpha_1$, $\beta=\beta_2$; $\gamma=(\beta^{-1}-\alpha^{-1})/8$, and $<x>=x-k$ ($-1/2< <x> \le 1/2$) with $k$ the closest integer to $x$.
\end{Lemma}

\begin{proof}

Consider the setup of \cite{DIZ} for 2 gaps $(\alpha,-\nu)\cup(\nu,\beta)$.
In the notation of \cite{DIZ}, $\alpha=a_0$, $\nu=-b_0=a_1$, $\beta=b_1$, $s=x$. 
We now verify that, if $\nu$ tends to zero at a sufficiently slow rate with $s\to\infty$, the jump matrices of the $R$ matrix in the Deift-Its-Zhou RH problem remain uniformly close to the identity, and therefore the analysis of \cite{DIZ} is extendable into that region.
We encircle the end-points of the gaps by nonintersecting discs. Note that the discs around $\nu$, $-\nu$ will have to contract
as $\nu$ tends to zero, we choose their radia to be $\nu/3$. For the matching of the local parametrices and the global one on the boundaries of the discs, we need, in particular, the parameter (see  (4.100), (4.102), etc in \cite{DIZ}) 
\be\label{rho}
\rho(z)=s\Omega^{(0)}(z)=s \int_{-\nu}^z \frac{q(t)}{\sqrt{p(t)}}dt,
\ee
to be uniformly large in absolute value on the boundary of the disc around $-\nu$. Here $p(z)=(z-\alpha)(z^2-\nu^2)(z-\beta)$,
$q(z)=z^2+q_1z+q_0$, where $q_1=-(\alpha+\beta)/2$, and the value of the constant $q_0$ is determined by the equation
\be\label{defq0}
0=\int_{\nu}^{\beta} \frac{q(t)}{\sqrt{p(t)}}dt=
\int_{\nu}^{\beta} \frac{t^2+q_1 t}{\sqrt{p(t)}}dt+q_0 \int_{\nu}^{\beta} \frac{dt}{\sqrt{p(t)}}
\ee
To analize the integrals in the limit $\nu\to\ 0$, we split the interval $(\nu,\beta)=(\nu,\sqrt{\nu}]\cup [\sqrt{\nu},\beta)$,
and change the integration variable $y=t/\sqrt{\nu}$ in the integration over the first one. We then obtain
\begin{align}
\int_{\nu}^{\beta} \frac{t^2}{\sqrt{|p(t)|}}dt&=\sqrt{|\alpha|\beta}+(\beta-|\alpha|)\arctan\sqrt{\frac{\beta}{|\alpha|}}+
\frac{\nu^2\log \nu^{-1}}{2\sqrt{|\alpha|\beta}}+O(\nu^2)\\
\int_{\nu}^{\beta} \frac{t}{\sqrt{|p(t)|}}dt&=2\arctan\sqrt{\frac{\beta}{|\alpha|}}-\frac{\beta-|\alpha|}{8(|\alpha|\beta)^{3/2}}\nu^2\log \nu^{-1}
+O(\nu^2)\\
\int_{\nu}^{\beta} \frac{1}{\sqrt{|p(t)|}}dt&=\frac{1}{\sqrt{|\alpha|\beta}}\left(
\log(\gamma\nu)^{-1}+\frac{1}{16}\left\{\frac{3}{\alpha^2}+\frac{3}{\beta^2}-\frac{2}{|\alpha|\beta}\right\}\nu^2\log \nu^{-1}\right)
+O(\nu^2),
\end{align}
where $\gamma=(\beta^{-1}+|\alpha|^{-1})/8$.
And therefore, (\ref{defq0}) gives
\be\label{q0}
q_0=-\frac{|\alpha|\beta}{\log(\gamma \nu)^{-1}}(1+O(\nu^2\log \nu^{-1})),\qquad \nu\to 0.
\ee
Substituting this expansion into (\ref{rho}), we obtain that
\be\label{ineqsec5}
|\rho(z)|\ge c\frac{s}{\log(\gamma \nu)^{-1}},\qquad c>0,
\ee
on the boundary of the disc around $-\nu$ ($c$ is independent of $z$, $\nu$, $s$). Similarly, we carry out the analysis around the other end-points of the gaps and obtain that the inequality (\ref{ineqsec5}) holds for the relevant quantities on the boundaries of all the disc around the end-points of the intervals.
In the notation of \cite{DIZ}, this means that
\be\label{vps}
v_{p,s}=I+O\left(\frac{\log(\gamma \nu)^{-1}}{s}\right)
\ee
uniformly on the boundaries of the discs. 

To prove the lemma, we need to verify that the jump matrices $v_{R,s}(z)$ in (4.123) in \cite{DIZ} have the form
$v_{R,s}(z)=I+o(1)$ on the jump contour Figure 4.122 in \cite{DIZ} in the asymptotic regime of the lemma.
First, on the boundaries of the discs (see Figure 4.122 in \cite{DIZ}),
\be
v_{R,s}(z)=f^{\infty}(v_{p,s})^{-1}(f^{\infty})^{-1},
\ee
where (see (3.42) in \cite{DIZ})
\[
f^{\infty}(z)=\begin{pmatrix}
\frac{\theta(u_{\infty}+d)}{\theta(u_{\infty}+sV+d)} & 0\cr
0 & \frac{\theta(u_{\infty}+d)}{\theta(u_{\infty}-sV+d)}
\end{pmatrix}
\begin{pmatrix}
\frac{\theta(u(z)+sV+d)}{\theta(u(z)+d)}\frac{\delta+\delta^{-1}}{2} & 
\frac{\theta(u(z)-sV-d)}{\theta(u(z)-d)}\frac{\delta-\delta^{-1}}{2i}e^{is\Omega_0}
\cr
\frac{\theta(u(z)+sV-d)}{\theta(u(z)-d)}\frac{\delta-\delta^{-1}}{-2i}e^{-is\Omega_0} & 
\frac{\theta(u(z)-sV+d)}{\theta(u(z)+d)}\frac{\delta+\delta^{-1}}{2}
\end{pmatrix}.
\]
Here $\theta(z)=\theta(z;\tau)$ and $V$, $\tau$ are as in (\ref{def theta}) and (\ref{Vtau11}),
\[
\Omega_0=2\alpha+2\int_{\infty}^{\alpha}\left(\frac{q(x)}{\sqrt{p(x)}}-1\right)dx,\qquad 
u(z)=\frac{\int_{\alpha}^z \frac{dt}{\sqrt{p(t)}}}{2\int_{\nu}^{\beta}\frac{dt}{\sqrt{p(t)}}},\qquad
\delta(z)=\left(\frac{(z+\nu)(z-\beta)}{(z-\nu)(z-\alpha)}\right)^{1/4}.
\]
The sheet of the Riemann surface $w=p(z)^{1/2}$ is chosen such 
that $p(z)^{1/2}\to 1$, $z\to\infty$.
The constant $u_{\infty}=u(\infty)$, and $d$ is chosen such that the zero of 
$\theta(u(z)-d)$ coincides
with the zero of $\delta(z)-\delta(z)^{-1}$ (which is inside $(-\nu,\nu)$). Note (see \cite{DIZ}) that $\theta(u(z)-d)$ has no other zeros, and  $\theta(u(z)+d)$
has no zeros. Thus $f^{\infty}(z)$ is analytic outside $A$ and clearly the limit $f^{\infty}(\infty)=I$. Moreover by standard arguments based on Liouville theorem, 
$\det f^{\infty}(z)=1$ for all $z$. Furthermore \cite{DIZ}, $u_{\infty}+d \equiv 0$ modulo the lattice
$m+n\tau$, $m,n\in\bbz$.

In the limit $\nu\to 0$, we have the expansions
\begin{align}
V&=\frac{1}{\pi}\int_{-\nu}^{\nu}\frac{q(t)}{\sqrt{|p(t)|}}dt=-
\sqrt{|\alpha|\beta}\left(\frac{1}{\log(\gamma\nu)^{-1}}-\frac{(\alpha+\beta)^2}{16\alpha^2\beta^2}\nu^2\right) +O(\nu^2/\log\nu^{-1}),\\
\tau&=\frac{i\pi}{\log(\gamma\nu)^{-1}}(1+O(\nu^2/\log\nu^{-1})),
\end{align}
and therefore
\be
\kappa=e^{-i\pi/\tau}=(\gamma\nu)^{1+O(\nu^2/\log\nu^{-1})}.
\ee
Note that, using the inversion formula ($\tau\to 1/\tau$) for the theta-functions, we can write
\be
\theta(z)=\frac{1}{\sqrt{-i\tau}}\sum_k e^{-\frac{i\pi}{\tau}(k-z)^2}=
\frac{\kappa^{<z>^2}}{\sqrt{-i\tau}}(1+O(\kappa^{1-2|<z>|})),
\ee
where
\[
z=j+<z>,\qquad -1/2< <\Re z> \le 1/2,\quad j\in\bbz.
\]

We can now esimate the matrix elements of $f^{\infty}$ on the boundaries of the discs.
On the boundary of the disc around $-\nu$, we have $u(z)=-1/2+r(z)$, where
$|r(z)|<\ep$ with a suitable $\ep>0$. Recalling periodicity properties of the theta-function,
$\theta(z+n\pm\tau)=e^{\mp 2\pi i z-i\pi\tau}\theta(z)$, we write for some $C>0$
uniformly on the boundary
\[\begin{aligned}
&\left|\frac{\theta(u_{\infty}+d)}{\theta(u_{\infty}+sV+d)}
\frac{\theta(u(z)+sV+d)}{\theta(u(z)+d)}\frac{\delta(z)+\delta(z)^{-1}}{2}\right|\\
&<C\left|\frac{\theta(0)}{\theta(sV)}
\frac{\theta(1/2+r(z)+sV+d)}{\theta(1/2+r(z)+d)}\right|
=O(\nu^{-(1+2\ep)|<sV>|}).  
\end{aligned}
\]
Similar estimate holds for the other elements of $f^{\infty}$ (we replace
$\delta-\delta^{-1}$ and $\theta(u(z)-d)$ in the off-diagonal elements with their derivatives 
at their zero). In the same vein, using the behaviour of $u(z)$, one obtaines similar estimates on the discs around the other end-points. (Note that, e.g., at $\alpha$, we can assume 
$|u(z)|<\ep$).
Recalling (\ref{vps}) we thus conlude that 
uniformly on these boundaries
\be\label{vRsin}
v_{R,s}=I+f^{\infty}O\left(\frac{\log(\gamma \nu)^{-1}}{s}\right)(f^{\infty})^{-1}=
I+O\left(\frac{\log(\gamma \nu)^{-1}}{s}(\nu^{-(1+2\ep)|<sV>|})^2\right).
\ee
Adjusting $\ep$, we can write this estimate as $v_{R,s}=I+O((s\nu^{2|<sV>|+\ep})^{-1})$.
The error term here is not small at the point $<sV>=1/2$, and we analyse the case 
of $|<sV>|$ close to $1/2$ separately below. Assume for now that $|<sV>|\le 1/4$. Then (\ref{vRsin}) is the estimate we need to prove the lemma in this case. It remains, however, to obtain the same, or better, estimate
for $v_{R,s}$ on the intervals outside the discs, where (Figure 4.122 in \cite{DIZ}),
\be\label{vRsout}
v_{R,s}=f_+^{\infty}\begin{pmatrix}
1 & -2e^{2isg_+(z)}\cr
0 & 1
\end{pmatrix}
(f_+^{\infty})^{-1},\qquad g(z)=z+\int_{\infty}^z\left(\frac{q(x)}{\sqrt{p(x)}}-1\right)dx.
\ee
Since by definition of $q(z)$ ((1.17) in \cite{DIZ} or (\ref{q9}) in the introduction),
\[
0=\int_{\nu}^{\beta} \frac{q(t)}{\sqrt{p(t)}}dt=\int_{\alpha}^{-\nu} \frac{q(t)}{\sqrt{p(t)}}dt,
\]
the estimation of $g_+(z)$ is similar to that of $\Omega^{(0)}(z)$ above, and we obtain that $\Re(ig_+(z))<0$ and
\[
-\Re(ig_+(z))\ge \frac{C}{\log(\gamma\nu)^{-1}},\qquad C>0,
\]
on the intervals outside the discs, and so
\[
e^{2isg_+(z)}=O\left(e^{-\frac{cs}{\log(\gamma\nu)^{-1}}}\right),
\]
with some constant $c>0$ independent of $s,\nu,z$.
Substituting this into (\ref{vRsout}), we obtain as above,
\be
v_{R,s}=I+O\left(\nu^{-1}e^{-\frac{cs}{\log(\gamma\nu)^{-1}}}\right)
\ee
uniformly on the intervals outside the discs.
Combining this result with (\ref{vRsin}), we see that the estimate
\[
v_{R,s}=I+O((s\nu^{2|<sV>|+\ep})^{-1})
\]
holds uniformly on the whole contour for $R$ in the asymptotic regime of the lemma, and therefore the lemma is proved in the case $|<sV>|\le 1/4$ by the arguments of \cite{DIZ}.

Now consider the remaining case $1/4<|<sV>|\le 1/2$. Let 
\[
t=<sV>+k/2,
\]
where $k=\pm 1$ is chosen so that $t\in(-1/4,1/4)$. Consider the following function which solves the same jump conditions (given in \cite{DIZ}) as $f^{\infty}$
\[
\wt  f^{\infty}=\frac{1}{\gamma(z_-)}\begin{pmatrix}
\frac{\theta(u(z_-)+d')}{\theta(u(z_-)+t+d')} & 0\cr
0 & \frac{\theta(u(z_-)+d')}{\theta(u(z_-)-t+d')}
\end{pmatrix}
\begin{pmatrix}
\frac{\theta(u(z)+t+d')}{\theta(u(z)+d')}\frac{\gamma+\gamma^{-1}}{2} & 
\frac{\theta(u(z)-t-d')}{\theta(u(z)-d')}\frac{\gamma-\gamma^{-1}}{2i}e^{is\Omega_0}
\cr
\frac{\theta(u(z)+t-d')}{\theta(u(z)-d')}\frac{\gamma-\gamma^{-1}}{-2i}e^{-is\Omega_0} & 
\frac{\theta(u(z)-t+d')}{\theta(u(z)+d')}\frac{\gamma+\gamma^{-1}}{2}
\end{pmatrix}.
\]
Here 
\[
\gamma(z)=\nu^{-1/4}\left(\frac{(z+\nu)(z-\nu)}{(z-\alpha)(z-\beta)}\right)^{1/4}.
\]
The sheet of the Riemann surface $w=p(z)^{1/2}$ is chosen as before such 
that $p(z)^{1/2}\to 1$, $z\to\infty$.
It is easy to verify that $\gamma(z)-\gamma(z)^{-1}$ has 2 zeros $z_+$, $z_-$.
As $\nu\to 0$, $z_{\pm}=\pm\sqrt{\nu|\alpha\beta|}(1+o(1))$.
The constant $d'$ is chosen such that the zero of 
$\theta(u(z)-d')$ coincides
with the zero $z_+$ of $\gamma(z)-\gamma(z)^{-1}$. As in \cite{DIZ}, Abel theorem then shows that $u(z_-)+d'\equiv 1/2$ modulo the lattice.
Furthermore, $\theta(u(z)-d')$ has no other zeros, and  $\theta(u(z)+d')$
has no zeros. Thus $\wt f^{\infty}(z)$ is analytic outside $A$ and the limit 
$\wt f^{\infty}(z_-)=I$. It follows by standard arguments that $\det\wt f^{\infty}(z)=1$ for all $z$. We also note that the limit 
\[
\Lambda=\wt f^{\infty}(\infty)
\]
has $\det\Lambda=1$ but is not the identity as before. By standard uniqueness arguments
\be\label{ff}
f^{\infty}(z)=\Lambda^{-1}\wt f^{\infty}(z).
\ee

The construction of local parametrices $\wt f_p$ around the edge points is similar to 
that in \cite{DIZ}. The definition of the new $R$-matrix is now as follows:
$R(z)=\Lambda f(z)\wt f_p^{-1}(z)$ in the discs around the end-points and 
$R(z)=\Lambda f(z)(\wt f^{\infty}(z))^{-1}$ outside. 
The jump matrices for $R$ at the boundaries of the discs
have the same form as before
\[
v_{R,s}=I+\wt f^{\infty}O\left(\frac{\log(\gamma \nu)^{-1}}{s}\right)(\wt f^{\infty})^{-1},
\] 
and a similar (to the one above) examination of the order of $\wt f^{\infty}$ 
on the boundaries shows that uniformly
\be\label{vrsnew}
v_{R,s}=I+O\left(\frac{1}{s\nu^{1/2+\ep}}\right),\qquad |t|<1/4.
\ee
As before, a better estimate holds on the rest of the jump contour of $R$.
Thus the asymptotics obtained holds in the regime $s\nu^{1/2+\ep}\to\infty$, for 
$|t|<1/4$. To finish the proof of the lemma it only remains to verify
(\ref{Lemma-form}) for $|t|< 1/4$. By Equation (3.9) in \cite{DIZ},
\[
\frac{\partial}{\partial s}\log P_s(A)=
-2G_0(\alpha,\beta,\nu) s +i(f_{1,22}-f_{1,11}),
\]
where $f_1$ is the coefficient in the large $z$ expansion $f(z)=I+f_1/z+O(1/z^2)$.
(Below, we also use $f^{\infty}(z)=I+f^{\infty}_1/z+O(1/z^2)$, $\wt f^{\infty}(z)=\Lambda+\wt f^{\infty}_1/z+O(1/z^2)$.)
By our definition of $R$,
\[
f(z)=\Lambda^{-1}R(z)\wt f^{\infty}(z)=
\Lambda^{-1}\left(I+\frac{R_1}{z}+O\left(\frac{1}{z^2}\right)\right)
\left(\Lambda+\frac{\wt f^{\infty}_1}{z}+O\left(\frac{1}{z^2}\right)\right)
\]
and therefore, using also (\ref{ff}),
\[
f_1=\Lambda^{-1}\wt f^{\infty}_1+\Lambda^{-1}R_1\Lambda=
f^{\infty}_1+\Lambda^{-1}R_1\Lambda.
\]
We have $\Lambda=O(\nu^{-1/4+|t|})$, and since $R_1$ has the same order as the error term in
(\ref{vrsnew}), $\Lambda^{-1}R_1\Lambda=O((s\nu^{1-2|t|+\ep})^{-1})=
O((s\nu^{2|<sV>|+\ep})^{-1})$.

Thus 
\[
\frac{\partial}{\partial s}\log P_s(A)=
-2G_0(\alpha,\beta,\nu) s+i(f^{\infty}_{1,22}-f^{\infty}_{1,11})+
O((s\nu^{2|<sV>|+\ep})^{-1}),
\]
But it was shown in \cite{DIZ} (Equation (3.48)) that 
$i(f^{\infty}_{1,22}-f^{\infty}_{1,11})=
\frac{\partial}{\partial s} \log\theta(sV;\tau)$, and we again obtain
(\ref{Lemma-form}) now for $1/4<|<sV>|\le 1/2$. The lemma is proved.

\end{proof}

\begin{remark}\label{Remark52}
In the overlap region $(0,\nu_0)\cap(\nu_1,\nu_2)$ of the asymptotics of Theorem \ref{Theorem u0s to 0} and the lemma, 
we can explicitly see, as an exercise, the coincidence of the main terms. 
Indeed, from (\ref{G010}), with $\alpha=\alpha_1$, $\beta=\beta_2$, $\nu=-\beta_1=\alpha_2$,
\[
G_0=q_0+\frac{1}{8}(\beta-\alpha)^2+\nu^2/2.
\]
Substituting here the expansion (\ref{q0}), we obtain
\[
G_0=\frac{1}{8}(\beta-\alpha)^2-\frac{|\alpha\beta|}{\log(\gamma \nu)^{-1}}+ O(\nu^2).
\]
Since $s\nu\to 0$, we see that this gives exactly the main (order $s^2$) term in (\ref{equation 2}). 
\end{remark}

\begin{remark}
Integration of the asymptotics of the lemma is related to the determination of the constant $c_1$ in (\ref{2 gap formula int})
which will be addressed in a future publication.
\end{remark}

\section*{Acknowledgements}
We are grateful to Tom Claeys for useful discussions and suggestions. 
The work of I.K. was partially supported by the Leverhulme Trust research fellowship
RF-2015-243.

\section*{Appendix}
We include a proof of the well-known formula \eqref{Fredholm Toeplitz}, using arguments from \cite{Deift}. 
As mentioned in the introduction, the gap probability of $m$ gaps in the bulk scaling limit is given by the sine-kernel Fredholm determinant \eqref {def Ps} for a wide class of random matrix ensembles.
A particular such ensemble is the Circular Unitary Ensemble (CUE), which is the group of $n\times n$ unitary matrices equiped with the Haar measure. The Haar measure induces a probability measure $p_n(\theta) d^n\theta$ on the eigenvalues of the matrix given by 
\begin{equation} \label{measure CUE eigs}
p_n(\theta)=\frac{1}{n!}\left(\frac{1}{2\pi}\right)^{n} \prod_{1\leq j<k\leq n}|e^{i\theta_j}-e^{i\theta_k}|^2, \quad \theta=(\theta_j)_{j=1}^n \in [0,2\pi)^n.
\end{equation}
From Heine's identity and \eqref{measure CUE eigs}, it follows that the probability that there are no eigenvalues on a set $\Sigma\subset C$, where $C$ is the unit circle, is given by the following
\begin{equation} \label{Toeplitz eigs}
D_n\left(J=C\setminus \Sigma\right)= \int_{e^{i\theta_j}\in J}p_n(\theta)d^n\theta,
\end{equation}
where $D_n(J)$ was defined in \eqref{T}.
 Denote  $J^{(n)}=C\setminus \Sigma^{(n)}$ where 
\begin{equation} \label{notation Sigman}
 \Sigma^{(n)}=\left\{ z\in : \arg z \in\left(\frac{2s\alpha}{n}, -\frac{2s\nu}{n}\right)\bigcup \left( \frac{2s\nu}{n}, \frac{2s \beta}{n}\right)\right\}.\end{equation}

Using the definition \eqref{measure CUE eigs} it is easily seen that
\begin{equation}
p_n(\theta)=\frac{1}{n!} 
\det \left( \widetilde{K}_n(\theta_j,\theta_k)\right)_{j,k=1}^n ,
\end{equation}
where $\widetilde K_n(x,y)=\frac{1}{2\pi}\sum_{j=0}^{n-1}e^{ji(x-y)}$. Let
\begin{equation}
K_n(x,y) =e^{-i\frac{n-1}{2}(x-y)}\widetilde{K}_n(x,y)=\frac{1}{2\pi}\frac{\sin \frac{n}{2} (x-y)}{\sin \frac{1}{2}(x-y)}.
\end{equation}
It follows that
\begin{equation}\label{pn and Kn}
p_n(\theta)=\frac{1}{n!}   \det\left( {K}_n(\theta_j,\theta_k)\right) .
\end{equation}
The kernel $K_n$ has the reproducing kernel property, meaning that for $r=1, \dots,n$
\begin{equation}\label{reproducing}
\int \det (K_n(\theta_j,\theta_k))_{j,k=1}^{n}d\theta_{n-r+1}\dots d\theta_n
=r! \det (K_n(\theta_j,\theta_k))_{j,k=1}^{n-r},
\end{equation}
where
\begin{equation}
\det (K_n(\theta_j,\theta_k))_{j,k=1}^0\equiv 1.
\end{equation}
From \eqref{Toeplitz eigs}, we see that
\begin{equation} \label{inclusion exlusion}\begin{aligned}
D_n(J)&= 
\int_{\theta\in(0,2\pi)^n} \prod_{j=1}^n (1-\chi_\Sigma(\theta_j))p_n(\theta)d^n\theta=\int_{\theta\in(0,2\pi)^n} p_n(\theta)d^n\theta
\\
&-n\int_{\theta\in(0,2\pi)^n} p_n(\theta) \chi_\Sigma(\theta_1)d^n\theta+\binom{n}{2} \int_{\theta\in(0,2\pi)^n} p_n(\theta)\chi_\Sigma(\theta_1)\chi_\Sigma(\theta_2)d^n\theta\\&+\dots +(-1)^n\binom{n}{n}\int_{\theta\in(0,2\pi)^n} p_n(\theta)\prod_{j=1}^n\chi_\Sigma(\theta_j)d^n\theta.
\end{aligned}\end{equation}
The Fredholm determinant of a trace-class operator $K$ acting on a set $S$ can be represented as
\begin{equation} \label{rep Fredholm}
\det(I-K)_S=1+\sum_{j=1}^{\infty}\frac{(-1)^j}{j!} \int_S \det(K(\theta_i,\theta_k))_{i,k=1}^jd^j\theta.
\end{equation}
For bounded $S$ and $K$, one may verify that the sum indeed converges using Hadamard's inequality. Let $J^{(n)}$ be given by \eqref{def J1J2}, and $\Sigma^{(n)}=C\setminus J^{(n)}$ be the complement. Recall $A$ from \eqref{notation A}.
Noting \eqref{pn and Kn}, we apply \eqref{reproducing} to \eqref{inclusion exlusion} to find that 
%\begin{equation}
%D_n(J)=\det\left(I-K_n\right)_{\Sigma}.
%\end{equation}
%It follows that 
\begin{equation} \label{Toep Fred}
D_n(J^{(n)})=\det(I-K_n)_{\Sigma^{(n)}}=\det\left(I-\widehat K_n\right)_{A},
\end{equation}
 where
\begin{equation}
\widehat K_n(x,y)=\frac{s\sin s(x-y)}{\pi n\sin\frac{s(x-y)}{n}}.
\end{equation}
For fixed $s$, as $n\to \infty$, we have 
\begin{equation}\label{Kn-Ks} \left|\widehat K_n(x,y)-K_s(x,y)\right|=\mathcal O(1/n).
\end{equation}
Since the sum \eqref{rep Fredholm} converges, 
\begin{equation}
\sum_{j=M}^{\infty}\frac{(-1)^j}{j!} \int_A \det(K(\theta_i,\theta_k))_{i,k=1}^jd^j\theta\to 0
\end{equation} as $M\to \infty$, for $K=\widehat K_n,K_s$, where $s$ remains fixed and uniformly for $n>N$ for some $N$. From \eqref{Kn-Ks}, it follows that for fixed but arbitrarily large $M$,
\begin{equation}\left|\sum_{j=1}^{M}\frac{(-1)^j}{j!} \int_A \det(\widehat K_n(\theta_i,\theta_k))_{i,k=1}^jd^j\theta-
\sum_{j=1}^{M}\frac{(-1)^j}{j!} \int_A \det(K_s(\theta_i,\theta_k))_{i,k=1}^jd^j\theta
\right|=\mathcal O(1/n)
\end{equation}
as $n\to \infty$.
 Thus it follows that 
 \begin{equation}
 \left|D_n(J^{(n)})-\det(I-K_s)_A\right|\to 0
 \end{equation}
 as $n \to \infty$ and $s$ remains fixed.

\end{document}